\documentclass[11pt]{amsart}
  \usepackage[margin=3.3cm]{geometry}
\usepackage{amsfonts,amsmath,amssymb,color,esint}
\usepackage{enumerate}
\usepackage{graphicx}
\usepackage{tikz}
\usetikzlibrary{decorations.markings}
\usetikzlibrary{plotmarks}
\usetikzlibrary{patterns}
\numberwithin{equation}{section}

\usepackage{amsthm}

\usepackage[babel=true,kerning=true]{microtype}

\usepackage{amsthm,leqno}
\usepackage{hyperref}

\newtheorem{conj}{Conjecture}[section]
\newtheorem{theo}{Theorem}[section]
\newtheorem{lem}{Lemma}[section]

\newtheorem{cl}{Claim}[section]
\newtheorem{prop}{Proposition}[section]
\newtheorem{cor}{Corollary}[section]

\newcommand{\eps}{\varepsilon}

\newcommand{\R}{\mathbb{R}}

\begin{document}

\title[]{Geometric spectral optimization on surfaces}
\author{Romain Petrides} 
\address{Romain Petrides, Universit\'e Paris Cit\'e, Institut de Math\'ematiques de Jussieu - Paris Rive Gauche, b\^atiment Sophie Germain, 75205 PARIS Cedex 13, France}
\email{romain.petrides@imj-prg.fr}

\begin{abstract}
We prove the existence of optimal metrics for a wide class of combinations of Laplace eigenvalues on closed orientable surfaces of any genus. The optimal metrics are explicitely related to Laplace minimal eigenmaps, defined as branched minimal immersions into ellipsoids parametrized by the eigenvalues of the critical metrics whose coordinates are eigenfunctions with respect to these eigenvalues. In particular, we prove existence of maximal metrics for the first Laplace eigenvalue on orientable surfaces of any genus. In this case, the target of eigenmaps are spheres. This completes a broad picture, first drawn by J. Hersch, 1970 (sphere), 
M. Berger 1973, N. Nadirashvili 1996 (tori). 

Our result is based on the combination of accurate constructions of Palais-Smale-like sequences for spectral functionals and on techniques by M. Karpukhin, R. Kusner, P. McGrath, D. Stern 2024, developped in the case of an equivariant optimization of the first Laplace and Steklov eigenvalues. Their result is significantly extended for two reasons: specific equivariant optimizations are not required anymore to obtain existence of maximizers of the first eigenvalue for any topology and our technique also holds for combinations of eigenvalues.
\end{abstract}

\maketitle

A classical goal of spectral geometry is 
to understand behaviours of the eigenvalues of operators with respect to the ambiant geometry. 
Indeed, estimates on eigenvalues of key geometric operators arise naturally when studying non linear PDEs that parametrize surfaces or describe a physical system via linearisations. 
This is a reason why looking for sharp bounds on the bottom of the spectrum of the Laplacian depending on the Riemannian metric was early questioned in seminal papers by Hersch \cite{hersch}, Yang-Yau \cite{yangyau}, Li-Yau \cite{liyau}, Berger \cite{berger}. The fundamental result by Hersch stated that the round metric is the only maximizer of the first eigenvalue among Riemannian metrics of fixed area on the sphere.
It is similar to the classical minimization of the first Laplace eigenvalue on domains in $\R^n$ of fixed volume with Dirichlet boundary conditions, conjectured by Rayleigh \cite{rayleigh}  and solved by Faber \cite{faber} and Krahn \cite{krahn}, or the maximization of the first Laplace eigenvalue on domains in $\R^n$ of fixed volume with Neumann boundary conditions (Szeg\"o \cite{szego}, Weinberger \cite{weinberger}).
In the context of the Riemannian shape optimization, the main difference is that the topology of the ambiant surface is fixed and the optimization among Riemannian metrics gives a richer geometry to the extremal metrics.

More precisely, Nadirashvili \cite{nadirashvili} discovered that extremal metrics of the first renormalized (by the area) Laplace eigenvalue on a closed surface $\Sigma$ directly
correspond to a significant geometric non-linear PDE: there is a family of first eigenfunctions associated to the optimal metrics that minimally immerse $\Sigma$ into a sphere. 
It gave a new light to the known optimizers (the round metric on the sphere \cite{hersch} and projective plane \cite{liyau} correspond to minimal embeddings into $\mathbb{S}^2$ and $\mathbb{S}^4$) and provided new techniques to obtain the unique optimizers on the torus \cite{nadirashvili} and the Klein bottle \cite{jnp}\cite{esgj} (corresponding to minimal embeddings into $\mathbb{S}^5$ and $\mathbb{S}^4$). In addition, Nadirashvili's result is related to Takahashi's characterization of isometric immersions into Euclidean spaces $x : \Sigma \to \R^n$ \cite{takahashi}: there is $\lambda > 0$ $\Delta x = \lambda x$ if and only if $x$ is a minimal immersion into a sphere. Indeed he also noticed that if there is a minimal immersion from $\Sigma$ into a sphere, its induced metric on $\Sigma$ has to be a \textit{critical metric} with respect to one renormalized Laplace eigenvalue functional on $\Sigma$.

Since then, many other outstanding critical metrics emerge for various geometric spectral functionals: one Laplace eigenvalue \cite{nadirashvili}\cite{EI3} \cite{EI4},  one Steklov eigenvalue \cite{fraserschoen2} \cite{karpukhinmetras}, one eigenvalue of the conformal Laplacian \cite{ammannhumbert}\cite{gurskyperez}, of the Paneitz operator in dimension 4 \cite{perez}, of the Dirac operator in dimension 2 \cite{ammann} \cite{karpukhinmetraspolterovich}, eigenvalues in K\"ahler geometry \cite{ajk}, combinations of Laplace or Steklov eigenvalues \cite{petrides-4} \cite{petrides-5}, Robin eigenvalues \cite{limamenezes} \cite{medvedev} and spectral functionals associated to various other operators \cite{pt}. The latter work unifies the previous ones, provides many other examples and gives a systematic way to compute these critical metrics by a Euler-Lagrange equation written via the theory of differentiation of locally-Lipschitz functionals with the concept of subdifferential.

All these works suggest a surprising way to build solutions of non-linear PDEs only focusing on the optimization of spectral functionals associated to linear operators. For instance, nodal solutions of the Yamabe equation can be built by optimization of eigenvalues of the conformal Laplacian \cite{ammannhumbert} \cite{gurskyperez}, harmonic maps into 2-spheres by minimization of the first Dirac eigenvalue \cite{ammann} or into larger spheres by maximization of the first Laplace eigenvalue \cite{petrides}, \cite{KS}, and in dimension $n\geq 3$ \cite{KS2} for harmonic maps and \cite{petrides-6} for n-harmonic maps. This approach also proved its efficiency for the construction of new minimal surfaces into specific target manifolds (spheres and ellipsoids), initiated by Fraser and Schoen \cite{fraserschoen}. In \cite{petrides-7} and \cite{petrides-8}, we proved the existence of embedded non planar minimal spheres into ellipsoids of $\R^4$ (resp. embedded non planar free boundary minimal disks into ellipsoids of $\R^3$) by equivariant optimization of combinations of renormalized first and second Laplace (resp. Steklov) eigenvalues. In \cite{kkms}, the authors constructed various examples of embedded minimal surfaces into $\mathbb{S}^3$ (resp. free boundary minimal surfaces into $\mathbb{B}^3$) of any topology by equivariant optimization of the first renormalized eigenvalue.

\subsection*{Main existence results}

In the current paper, we go back to the original fundamental problem: are there  maximal metrics for the first renormalized eigenvalue functional $\bar{\lambda}_1$
? 
Here, for $k \geq 1$
\begin{equation} \label{def:barlambdak} g \mapsto \bar{\lambda}_k(\Sigma,g) := \lambda_k(\Sigma,g)A_g(\Sigma) \end{equation}
denotes the $k$-th renormalized eigenvalue of a closed connected surface $\Sigma$ where for a Riemannian metric $g$, $\lambda_k(\Sigma,g)$ denotes the $k$-th non-zero Laplace eigenvalue of the surface and $A_g(\Sigma)$ its area.
In the current paper, we show that the answer is yes for orientable surfaces:

\begin{theo} \label{theomain1}
On any closed orientable surface, $\bar{\lambda}_1$ realizes a maximum at a smooth (outside a finite number of conical singularities) metric. 
\end{theo}

We completely extend a result that was explicitely known for the round sphere \cite{hersch}, the projective plane \cite{liyau}, the torus \cite{nadirashvili}, the Klein bottle \cite{jnp} \cite{esgj} \cite{ckm}, the orientable surfaces of genus two \cite{jnp} \cite{nayatanishoda} to all the orientable surfaces

With the characterization of \cite{nadirashvili} mentioned above, up to a dilatation, the maximal metrics of $\bar{\lambda}_1$ in Theorem \ref{theomain1} are induced metrics of (possibly branched) minimal immersions by first eigenfunctions into spheres. 
Notice that the regularity result is optimal: the maximizers on surfaces of genus 2 must have conical singularities corresponding to the branched points of the minimal immersions. 

As we shall explain later (see also \cite{petrides}, \cite{MS1}), this existence result is deeply related to the monotonicity of the topological invariant defined as the supremum of $\bar{\lambda}_1$ with respect to the topology:
$$ \Lambda_1(\gamma) := \sup_{g \in Met(\Sigma_\gamma)} \bar{\lambda}_1(\Sigma_\gamma,g) $$
if $\Sigma_\gamma$ is an orientable surface of genus $\gamma$ and Euler characteristic $2-2\gamma$ 
In the current paper, we prove the following monotonicity results that imply Theorem \ref{theomain1} by \cite{petrides}:
\begin{theo} \label{theomonotonicity} For all $\gamma \geq 1$ 
,
$$ \Lambda_1(\gamma) > \Lambda_1(\gamma-1) $$
\end{theo}
The large inequalities were already proved in \cite{ces} 
by a standard glueing method. It was already known from \cite{yangyau}, (for $\Lambda_1$) 
that these supremum are finite 
and bounded by constants depending linearly on the genuses. It is proved in \cite{karpukhin-2}, that the bound by Yang and Yau for $\Lambda_1$ is never sharp except in genuses $0$ and $2$. The actual best asymptotic of $\Lambda_1(\gamma)$ as $\gamma \to +\infty$ is proved in \cite{karpukhinvinokurov} and \cite{ros}. 

\medskip

In the current paper, we also give a generalization of Theorems \ref{theomain1} and \ref{theomonotonicity} to positive combinations of eigenvalues. Many works in spectral geometry ask for the optimization of combinations of eigenvalues because it is more related to physical systems that require more than information on the ground state. We also extract more geometric information from the whole spectrum (since it is related to Riemannian invariants) or interactions between eigenvalues (e.g gap estimates or bounds of one eigenvalue by another) than from its very bottom: see e.g in very various contexts \cite{yangyau} \cite{OPS} \cite{andrewsclutterbuck} \cite{bmpv} etc. 

As first noticed in \cite{petrides-4}, \cite{petrides-5} and then in \cite{pt}, extremal metrics for combinations of eigenvalues also enjoy a remarkable geometric property that can be used for the optimization: they are explicitely written with respect to minimal branched immersions into an ellipsoid paramatrized by the eigenvalues associated with the critical metric that appear in the combination. The coordinates of the branched immersion are eigenfunctions with respect to these eigenvalues.
A variational method based on the construction of almost extremal maximizing sequences is then available in \cite{petrides6} \cite{petrides-6}, is independently used in the current paper and is promising to be generalized to other contexts of eigenvalue optimization.

Let's set the combinations that appear in our general result. Let $F : \left(\R_+\right)^m \to \R \cup \{+\infty\}$ be a continuous map on $\left(\R_+\right)^m$, a $\mathcal{C}^1$ map on $int\left(\left(\R_+\right)^m\right)$ such that $F(x)=+\infty \Rightarrow x\in \partial \left(\R_+\right)^m$ and $\partial_i F : \left(\R_+\right)^m \to \R \cup \{+\infty\}$ is continuous. We assume that for all coordinate $x_i$, $F$ satisfies that either $F$ is independent of $x_i$ or strictly decreasing with respect to $x_i$ in the following sense:
$$(H) \hspace{10mm} \forall i \in \{1,\cdots,m\}, \begin{cases} \forall x \in int\left( \left(\R_+\right)^m\right), \partial_i F(x) < 0 \\
\text{ or } \forall x \in int\left( \left(\R_+\right)^m\right), \partial_i F(x)=0 \end{cases}.$$
Examples of such functions could be for $a_1,\cdots,a_m \in \R_+$
$$ F(\lambda_1,\cdots,\lambda_m) = \sum_{i=1}^m a_i f(\lambda_i)$$
where $f : \R_+ \to \R$ is a $\mathcal{C}^1$ function such that $f'(\lambda)<0$ if $\lambda>0$, e.g 
$$ f(\lambda) = e^{- \lambda t}, t>0 \text{ or } f(\lambda) = \lambda^{-s}, s>0 \text{ or } f(\lambda) = - \ln \lambda. $$
that can be used for partial sums of the trace of the heat kernel, the zeta function or the determinant.
We set for a closed surface $\Sigma$
$$ E(\Sigma,g) = F(\bar{\lambda}_1(g),\cdots,\bar{\lambda}_m(g)) .
$$
We also set
$$ E_0(\Sigma,g) = F(0,\bar{\lambda}_2(g),\cdots,\bar{\lambda}_m(g)) .
$$

\begin{theo} \label{theomain2}
Let $\Sigma$ be a closed orientable surface. If
$$ \inf_{g} E(\Sigma,g) < \inf_g E_0(\Sigma,g) $$
then $E(\Sigma, \cdot)$ realizes a maximum at a smooth (outside a finite number of conical singularities) metric. 
\end{theo}
The gap assumption in Theorem \ref{theomain2} ensures that the first renormalized eigenvalue of minimizing sequences does not converge to $0$, preventing from disconnection of minimizing sequences as explained in \cite{petrides-4}. Such an assumption is necessary because of the example of maximization of one eigenvalue $\bar{\lambda}_k$ on spheres or projective planes.
Indeed, it was proved by \cite{knpp21} with the combination of \cite{petrides-2} (see also \cite{knpp19}) and \cite{Ejiri} that $\Lambda_k(0) := \sup \bar{\lambda}_k(\mathbb{S}^2,\cdot) = 8\pi k$, corresponding to the $k$-th renormalized eigenvalue of $k$ disjoint spheres of same area, is never realized by a metric on the sphere for $k\geq 2$. 

In the previous examples of combinations $F$, this gap assumption is automatically satisfied if $a_1 >0$ and $f(0) = +\infty$, e.g for $f(\lambda) = \lambda^{-s}$ or $f(\lambda) = -\ln \lambda$. In particular the functionals $ \sum_{i=1}^m a_i \left(\bar{\lambda}_i\right)^{-1} $ for $a_1 >0$
have a minimizer for any topology. For $m=1$, Theorem \ref{theomain1} is nothing but a corollary of Theorem \ref{theomain2}. For $m=3$ and $a_i:=1$, we minimize Yang-Yau's functional \cite{yangyau} (modelled on Hersch's result on the sphere \cite{hersch}) for any topology. In \cite{petrides-7}, we proved that for $m=2$, $a_1 := 1$ and $a_2 := t$ large enough on the sphere, the maximizers correspond to critical metrics associated to non planar minimal immersed spheres into rotational ellipsoids of $\R^4$.

One interesting question among many others would be to know the geometry of optimizers of $ \sum_{i=1}^m \left(\bar{\lambda}_i\right)^{-1} $ (or other finite combinations) as $m\to +\infty$. We early knew from \cite{hersch} that for $m=1,2,3$, the round sphere is the unique minimizer among spheres and from \cite{berger} that for $m=6$, the flat equilateral torus (minimizer for $m=1$) cannot be a minimizer among tori.

\subsection*{Overview of variational approaches}
The first natural functionals early studied are for $k\geq 1$, $\bar{\lambda}_k$ (see \eqref{def:barlambdak}) and $\bar{\sigma}_k$ defined as
\begin{equation} \label{def:barsigmak} g \mapsto \bar{\sigma}_k(\Sigma,g) := \sigma_k(\Sigma,g)L_g(\partial\Sigma) \end{equation}
the $k$-th renormalized Steklov eigenvalues of a compact connected surface with boundary $\Sigma$ where for a Riemannian metric $g$, $\sigma_k(\Sigma,g)$ denotes the $k$-th non-zero Steklov eigenvalue of the surface and $L_g(\partial\Sigma)$ the length of the boundary. These functionals are known to be bounded in the set of Riemannian metrics on a fixed surface: \cite{yangyau} \cite{liyau} for $\bar{\lambda}_1$, \cite{korevaar} for $\bar{\lambda}_k$ and see also \cite{hassannezhad} \cite{kokarev2} for $\bar{\lambda}_k$ and $\bar{\sigma}_k$.

After the seminal papers by Nadirashvili \cite{nadirashvili} for $\bar{\lambda}_1$ on tori and Fraser and Schoen \cite{fraserschoen} for $\bar{\sigma}_1$ on surfaces with boundary of genus zero, several works looked for a systematic variational method to understand whether eigenvalue functionals admit extremal metrics or not (e.g \cite{petrides}, \cite{petrides-2}, \cite{petrides-3}, \cite{knpp19}, \cite{KS}, \cite{petrides6}  etc). One important first step was the optimization in a conformal class of metrics. It is very convenient since there is a family of first eigenfunctions associated to the extremal metrics that are coordinates of a harmonic map into a sphere (or free boundary harmonic maps into spheres). After this step, the maximization among conformal classes is reduced to a maximization on the finite dimensional Teichm\"uller space of the surface. In \cite{petrides}, we proved that $\bar{\lambda}_1$ always realizes a maximum among metrics in a given conformal class and we gave a gap assumption for the existence of a maximizer of $\bar{\lambda}_1$ on the set of metrics on an orientable surface $\Sigma_{\gamma}$ of genus $\gamma$:
$$ \Lambda_1(\gamma) > \Lambda_1(\gamma-1) \Rightarrow \text{Existence of a maximum for } \Lambda_1(\gamma)$$
where $ \Lambda_1(\gamma) := \sup \bar{\lambda}_1(\Sigma_\gamma,\cdot) $. A similar formula holds for non-orientable surfaces (see \cite{MS1}). Such a gap assumption prevents maximizing sequences from degenerating to lower topologies in the Teichmuller space. In other words, all the work in \cite{petrides} consisted in proving that it is the only obstruction for convergence of the maximizing sequences we constructed. The analogous gap assumption for the existence of a maximizer of $\bar{\sigma}_1$ on the set of metrics on an orientable surface $\Sigma_{\gamma,b}$ of genus $\gamma$ with $b$ boundary components was proved in \cite{fraserschoen} (genus 0) and \cite{petrides-3} for any topology (and higher eigenvalues):
$$ \begin{cases} \sigma_1(\gamma,b) > \sigma_1(\gamma,b-1)  \text{ for } \gamma \geq 0, b\geq 2 \\
 \sigma_1(\gamma,b) > \sigma_1(\gamma-1,b+1) \text{ for } \gamma \geq 1, b\geq 1
\end{cases} \Rightarrow \text{Existence of a maximum for } \sigma_1(\gamma,b) $$
where $\sigma_1(\gamma,b):=\sup \bar{\sigma}_1(\Sigma_{\gamma,b},\cdot)$.

These crucial strict inequalities were left to be proved in order to obtain existence of maximizers for any topology. The main attempt was to prove that the following topological perturbations of a Riemannian surface $(\Sigma,g)$ strictly increase the first renormalized eigenvalue of a new surface $(\tilde{\Sigma}_\eps,\tilde{g}_\eps)$ obtained from $(\Sigma,g)$ by "increasing" the topology:
\begin{itemize}
\item Gluing a small handle at a small neighborhood of two points of $\Sigma$ for $\bar{\lambda}_1$ or $\bar{\sigma}_1$ of area $\eps \to 0$. ($\tilde{\Sigma}_\eps$ is the connected sum of $\Sigma$ and a torus or a Klein bottle depending on orientation-preserving gluing or not)
\item Making a cross-cap at a small neighborhood of one point of $\Sigma$ for $\bar{\lambda}_1$ or $\bar{\sigma}_1$ of area $\eps \to 0$ ($\tilde{\Sigma}_\eps$ is the connected sum of $\Sigma$ and a projective plane)
\item Making a small hole in $\Sigma$ for $\bar{\sigma}_1$ by removing a disk of boundary length $\eps \to 0$ (the boundary of $\tilde{\Sigma}_\eps$ has one more connected component)
\item Gluing a small strip at a small neighborhood of two points of the boundary of $\Sigma$ for $\bar{\sigma}_1$ of length $\eps \to 0$ (There are four cases of new topology for $\tilde{\Sigma}_\eps$ depending on gluing along two points of the same boundary component or not and gluing preserving the orientation or not)
\end{itemize}
In all these cases, while we know that $\bar{\lambda}_1(\tilde{\Sigma}_\eps,\tilde{g}_\eps) \to \bar{\lambda}_1(\Sigma,g) $ or $\bar{\sigma}_1(\tilde{\Sigma}_\eps,\tilde{g}_\eps) \to \bar{\sigma}_1(\Sigma,g)$ as $\eps\to 0$,  we aimed at proving:
\begin{equation} \label{eq:monotonicity}\bar{\lambda}_1(\tilde{\Sigma}_\eps,\tilde{g}_\eps) > \bar{\lambda}_1(\Sigma,g)  \text{ or } \bar{\sigma}_1(\tilde{\Sigma}_\eps,\tilde{g}_\eps) > \bar{\sigma}_1(\Sigma,g) \end{equation}
for $\eps$ small enough with a fine asymptotic analysis. With such monocity results, we could use $(\tilde{\Sigma}_\eps,\tilde{g}_\eps)$ as a test Riemannian surface for the variational problem $\Lambda_1(\gamma)$ or $\sigma_1(\gamma,b)$, assuming that $(\Sigma,g)$ is a maximizer for a lower topology. We recently used such a glueing argument in the context of maximization of linear combinations of first and second eigenvalues in \cite{petrides-7} and \cite{petrides-8}. In \cite{fraserschoen}, \cite{MS3} and \cite{MP}, the authors betted that the monotonicity results \eqref{eq:monotonicity} should occur and developped subtle original techniques for the asymptotic analysis on eigenvalues under topological perturbation. All these papers contain a gap that prevents from the result they intended, but are still interesting for the accuracy of the asymptotic expansion of eigenvalues. After all these investigations it is not clear that such a monotonicity result holds true for any choice of $(\Sigma,g)$. 

\medskip

In \cite{kkms}, the authors managed to use in addition that $(\Sigma,g)$ is also a maximizer for the lower topology to obtain new existence results in various situations. In their proof, they replace the Riemannian surface $(\tilde{\Sigma}_\eps,\tilde{g}_\eps)$ by $(\tilde{\Sigma}_\eps,g_\eps)$, where $g_\eps$ is a maximizer of $\bar{\lambda}_1$ (or $\bar{\sigma}_1$) in the conformal class of $\tilde{g}_\eps$. Let's give details on the structure of their proof for $\bar{\lambda}_1$ when $\tilde{\Sigma}_\eps$ is obtained from an orientable surface $\Sigma$ of genus $\gamma$ by gluing a small handle at the neighborhood of $p,q \in \Sigma$. We obtain a maximizing sequence $(g_\eps)$ for $\Lambda_1(\gamma+1)$. In addition, for any $\eps$, $g_\eps$ is a maximizer of $\bar{\lambda}_1$ in its conformal class so that there is a sequence of harmonic maps $\Phi_\eps : (\tilde{\Sigma}_\eps,g_\eps) \to \mathbb{S}^{n_\eps}$ whose coordinates are first eigenfunctions. Assuming by contradiction that $\Lambda_1(\gamma+1) = \Lambda_1(\gamma)$, we obtain
$$\bar{\lambda}_1(\Sigma,g) = \Lambda_1(\gamma) = \Lambda_1(\gamma+1) \geq  \bar{\lambda}_1(\tilde{\Sigma}_{\eps},g_\eps) \geq  \bar{\lambda}_1(\tilde{\Sigma}_{\eps},\tilde{g}_\eps) \to \bar{\lambda}_1(\Sigma,g) $$
as $\eps\to 0$. Thanks to this strong property, they quantified the smallness of the energy of $\Phi_\eps$ inside the added handle, so that the limit $\Phi : (\Sigma,g) \to \mathbb{S}^n $ of $\Phi_\eps$ as $\eps\to 0$ on $\Sigma$ must satisfy $\Phi(p) = \Phi(q)$. It led to the following result
\begin{equation} \label{eq:KKMS} \forall p,q \in \Sigma, \exists \Phi_{p,q} : (\Sigma,g) \to \mathbb{S}^n, \text{ s.t } \Delta_g \Phi_{p,q} = \lambda_1(\Sigma,g)\Phi_{p,q} \text{ and } \Phi_{p,q}(p) = \Phi_{p,q}(q). \end{equation}
A similar property holds in the context of Steklov eigenvalues. In their proof by contradiction, the authors deduced that such a strong information is not possible if the multiplicity of the first eigenvalue is known to be small. This is the case for surfaces of small genus by classical topological bounds on the multiplicity \cite{cheng} \cite{besson} \cite{nadirashvili2} \cite{kkp}, or if we know that the maximal metric $g$ is equivariant. They also gave an argument that proves that $\Lambda_1(\gamma) > \Lambda_1(\gamma-2)$. However, although they conjectured Theorem \ref{theomain1} by conjecturing that \eqref{eq:KKMS} leads to a contradiction, we show in the current paper that something natural is missing in their approach (see \eqref{eq:mainprop}).

Notice also that they use a technology that first appeared in \cite{KS} for the maximization of $\bar{\lambda}_1$ in a conformal class which is very different from \cite{petrides-4} and \cite{petrides6}. While it brings more information, their technology is very specific to the maximization of the first eigenvalue and cannot be used for Theorem \ref{theomain2}.

\medskip

Let's generalize to other spectral functionals. Similar gap assumptions are given for the existence of any higher eigenvalue in \cite{petrides-2} and \cite{petrides-3} (see also an alternative proof for Laplace eigenvalues \cite{knpp19}). As already said, the sphere and projective plane for Laplace eigenvalues and the disk for Steklov eigenvalues do not realize the supremum of the $k$-th eigenvalues ($k\geq 2$) so that the strict inequalities do not occur. In \cite{petrides-4}, \cite{petrides-5} we give the gap assumption for general positive combinations of eigenvalues that can simply be stated as 
$$ \forall \hat{\Sigma} \in \mathcal{LT}(\Sigma), \inf E(\Sigma,\cdot) < \inf E(\hat{\Sigma},\cdot) \Rightarrow \text{ Existence of a minimum of } E(\Sigma,\cdot)  $$
where for closed surfaces $\Sigma$, $\mathcal{LT}(\Sigma)$ is the set of surfaces obtained from $\Sigma$ by cutting $\Sigma$ along a finite number of disjoint closed curves and glue disks along the connected components of the boundary of the (possibly disconnected) surfaces we obtain. The letters $\mathcal{LT}$ stand for "Lower Topologies". If $\Sigma$ is a surface with boundary, $\mathcal{LT}(\Sigma)$ is the set of surfaces obtain by cutting $\Sigma$ along a finite number of disjoint properly embedded curves with endpoints at the boundary. The advantage of this formulation is the unification of all the gap assumptions. In the current paper, we focus on the subset $\mathcal{LTC}(\Sigma)$ of connected surfaces in $\mathcal{LT}(\Sigma)$. We also proved in \cite{petrides-4}, \cite{petrides-5} the following alternative gap assumption:
$$ \begin{cases} 
 \forall \hat{\Sigma} \in \mathcal{LTC}(\Sigma), \inf E(\Sigma,\cdot) < \inf E(\hat{\Sigma},\cdot) \\
 \inf E(\Sigma,\cdot) < \inf E_0(\Sigma,\cdot)
  \end{cases} \Rightarrow \text{ Existence of a minimum of } E(\Sigma,\cdot)  $$
where the second assumption ensures that the first eigenvalue of minimizing sequences does not converge to $0$ so that disconnections cannot occur. Notice that
$$ \forall \Sigma \in \mathcal{LT}(\tilde{\Sigma}), \begin{cases} \inf E(\widetilde{\Sigma},\cdot) < \inf E_0(\widetilde{\Sigma},\cdot) \\  \inf E(\widetilde{\Sigma},\cdot) = \inf E(\Sigma,\cdot) \end{cases} \Rightarrow \begin{cases} \Sigma \in \mathcal{LTC}(\tilde{\Sigma}) \\ \inf E(\Sigma,\cdot) < \inf E_0(\Sigma,\cdot) \end{cases} $$
so that we can prove by induction gaps on $\tilde{\Sigma}$ assuming that the functional admits a minimum for any $\Sigma \in \mathcal{LTC}(\widetilde{\Sigma})$.

In the same spirit as \cite{kkms}, the main step to obtain Theorem \ref{theomain2} (and Theorem \ref{theomain1}) in the orientable closed case is then to prove that if $\inf E(\widetilde{\Sigma},\cdot) = \inf E(\Sigma,\cdot)$ where $\Sigma$ has genus $\gamma$ and $\widetilde{\Sigma}$ has genus $\gamma+1$, and if the minimum of $E(\Sigma,\cdot)$ is realized by $g$, we obtain a contradiction. Under these assumptions, we prove in the current paper:
\begin{equation} \label{eq:mainprop} \forall p,q \in \Sigma, \exists \Phi_{p,q} : (\Sigma,g) \to \mathcal{E}_{\Lambda(\Sigma,g)}, \text{ s.t } \begin{cases} \Delta_g \Phi_{p,q} = \Lambda(\Sigma,g)\cdot \Phi_{p,q} \\
\Phi_{p,q} \text{ is a (possibly branched) conformal map} \\ 
\Phi_{p,q}(p) = \Phi_{p,q}(q) \end{cases} \end{equation}
where $\Lambda(\Sigma,g) := diag\left( \lambda_1(\Sigma,g),\cdots,\lambda_n(\Sigma,g) \right)$ for some $n$ such that $\lambda_n(\Sigma,g) = \lambda_m(\Sigma,g)$ and 
$$\mathcal{E}_{\Lambda(\Sigma,g)} := \left\{ x \in \R^n ; \sum_{i=1}^n \lambda_i(\Sigma,g) x_i^2 = 1 \right\}.$$
Beyond the generalization to combination of eigenvalues, the new information compared with \eqref{eq:KKMS} is the conformality of $\Phi_{p,q}$. This implies that the harmonic map $\Phi_{p,q}$ is a branched conformal minimal immersion. Therefore, given $p\in \Sigma$ and letting $q \to p$ along $v \in T_p \Sigma$ we obtain a branched minimal immersion $\Phi_{p,v} : \Sigma \to \mathcal{E}_{\Lambda(\Sigma,g)}$ such that $D\Phi_{p,v}(v) = 0$. By conformality, we deduce $\nabla \Phi_{p,v} = 0$ and $p$ is a branched point of $\Phi_{p,v}$ and a conical singularity of $g$. Since the result holds for any $p$, we obtain a contradiction. 

\medskip

In order to obtain \eqref{eq:mainprop}, instead of replacing $(\tilde{\Sigma}_\eps,\tilde{g}_\eps)$ by a maximizing sequence $g_\eps$ which is maximal (and then critical) in a fixed conformal class leading to a sequence of harmonic maps, we replace it by a maximizing sequence which is \textit{almost} critical with respect to the \textit{whole} set of metrics, leading to a sequence of a \textit{almost harmonic maps} which are \textit{almost conformal}. We definitely take into consideration that our variational problem holds in the set of all metrics in the computation of the first variations. The selection of this maximizing sequence is based on a Palais-Smale-like trick for locally Lipschitz functionals. The choice of the variational space is crucial: not too weakly regular to be able to define eigenvalues and associated eigenfunctions and to compute the first variations, and not too strongly regular to have a significant Palais-Smale-like condition. This variational space is somewhat included into the space of Radon measures (dual space of continuous measures) studied in \cite{kokarev} and $H^{-1}$ (dual space of $H^1$) used in \cite{petrides} and \cite{knps}, and is not far smaller than their intersection. It is given in Section \ref{sec:1}. It is also used for an alternative (and somewhat simpler than \cite{petrides-4}) proof of optimization of positive combinations of eigenvalues in a conformal class in \cite{petrides6} for a more general result. 

\medskip

\subsection*{Side results and open questions}

\medskip

Our glueing method can be extended in the non orientable closed case with the glueing of a cylinder that reverses the orientation. As explained in \cite{MS1} the existence result on a non-orientable surface $\Sigma_\delta^K$ of non-orientable genus $\delta$ and Euler characteristic $2-\delta$ ($\delta$ is the number of projective planes $\mathbb{RP}^2$ that appear in the connected sum $\Sigma_\delta^K \simeq  \mathbb{RP}^2 \sharp \cdots \sharp \mathbb{RP}^2$) holds if 
$$ \Lambda_1^K(\delta) >  \Lambda_1^K(\delta-1) \text{ and } \Lambda_1^K(\delta) > \Lambda_1\left(\left\lfloor\frac{\delta-1}{2}\right\rfloor \right) $$
where 
$$ \Lambda_1^K(\delta) := \sup_{g \in Met(\Sigma_\delta^K)} \bar{\lambda}_1(\Sigma_\delta^K,g) $$
Notice that the large inequalities are true \cite{ces} \cite{MS1}. Our analysis then gives the following result
\begin{theo} \label{theomain1}
For any $\delta\geq 2$ and for one among the non-orientable surfaces of genus $\delta$ and $\delta+1$, $\bar{\lambda}_1$ realizes a maximum at a smooth (outside a finite number of conical singularities) metric. 
\end{theo}
However, our analysis does not tackle the cross-cap glueing. The complete non orientable case is then left open. In addition we would like to give a new conjecture in the Steklov case (more adapted than Conjecture 1.29 in \cite{kkms}) that would imply the very analogous results to Theorem \ref{theomain1} and Theorem \ref{theomonotonicity} in the context of Steklov eigenvalues. This conjecture is true if $\Sigma$ is a surface of genus $0$ (see \cite{kkms}).
\begin{conj} \label{conj} Let $\Sigma$ be a surface with a disconnected boundary $\partial\Sigma$, and $g$ be a maximizer of $\bar{\sigma}_1(\Sigma,\cdot)$ then
$$ \exists (p,q) \in \Delta(\partial \Sigma), \forall  \Phi \in \mathcal{C}(\Sigma,g), \Phi(p)\neq \Phi(q)  $$
where $\Delta(\partial\Sigma)$ denotes the couples of points $(p,q)$ such that $p$ and $q$ belong to disjoint connected components of $\partial\Sigma$ and $\mathcal{C}(\Sigma,g)$ denotes the set of maps $ \Phi : (\Sigma,g) \to \left(\mathbb{B}^n, \mathbb{S}^{n-1} \right)$ such that
\begin{equation} \label{eq:mainpropsteklov}  \begin{cases} \Delta_g \Phi = 0 \text{ in } \Sigma \\
\partial_\nu \Phi = \sigma_1(\Sigma,g)\cdot \Phi \text{ on } \partial \Sigma \\
\Phi \text{ is a (possibly branched) conformal map.} \end{cases} \end{equation} 
\end{conj}

\subsection*{Organization of the paper}

The first part of the paper is devoted to explain the transformation of a minimizing sequence $(\tilde{\Sigma}_\eps,\tilde{g}_\eps,\tilde{\beta}_\eps)$ into a Palais-Smale sequence $(\tilde{\Sigma}_\eps,g_\eps,\beta_\eps)$ in the adapted variational space. The second parts consider the necessary topological perturbations that suffice to prove Theorem \ref{theomain2}. Since the proofs are very similar, we chose to develop in detail the case of handle attachment for Laplace eigenvalues on closed surfaces (Section \ref{sec:handle}), while in the case of strip attachments for Steklov eigenvalues (Section \ref{sec:strip}) we only detail what differs much from Section \ref{sec:handle}. In every case we start with the construction of the initial minimizing sequence, then we prove the convergence of the Palais-Smale-like modified sequence and we conclude by the contradiction (up to conjecture \ref{conj} in the Steklov case).

\section{Variational framework} \label{sec:1}

In all the section, let $\Sigma$ be a smooth compact surface $\partial\Sigma =\emptyset$ (resp. $\partial \Sigma\neq \emptyset$) in the context of optimization of Laplace (resp. Steklov) eigenvalues.

\subsection{A distance between continuous Riemannian metrics}

We denote $Met_0(\Sigma)$ the set of continous metrics on $\Sigma$. We endow this set with the following distance between $g_1,g_2 \in Met_0(\Sigma)$
$$ \delta(g_1,g_2):= \max_{x\in \Sigma} \left( \ln\left( \max_{v\in T_x\Sigma \setminus\{0\}} \frac{g_1(x)(v,v)}{g_2(x)(v,v)}  \right)^2 + \ln\left( \max_{v\in T_x\Sigma \setminus\{0\}} \frac{g_2(x)(v,v)}{g_1(x)(v,v)}  \right)^2   \right)^{\frac{1}{2}} $$
where we notice that for $g \in Met_0(\Sigma)$ and a symmetric 2-tensor $h\in S^2_0(\Sigma) = T_g \left( Met_0(\Sigma) \right)$,
\begin{equation} \label{eq:deltadistancetangentspace}
\begin{split} \lim_{t\to 0}\frac{\delta(g,g+th)}{t} = & \max_{x\in \Sigma} \left( \left(\max_{v\in T_x\Sigma \setminus\{0\}} \frac{h(x)(v,v)}{g(x)(v,v)}\right)^2 + \left(\min_{v\in T_x\Sigma \setminus\{0\}} \frac{h(x)(v,v)}{g(x)(v,v)}\right)^2 \right)^{\frac{1}{2}} \\
=: & \max_{x\in \Sigma}  \sqrt{ \langle h,h \rangle_g(x)}
\end{split} \end{equation}
where we define for a local orthonormal frame $(e_1,e_2)$ with respect to $g$ at the neighborhood of $x$ and $h_1,h_2 \in S_0^2(\Sigma) = T_g \left( Met_0(\Sigma)\right)$, 
$$  \langle h_1,h_2 \rangle_g(x) =\sum_{i,j} h_1(x)(e_i(x),e_j(x))\cdot h_2(x)(e_i(x),e_j(x)) $$
a scalar product that is independent of the choice of the orthonormal frame. Notice also that $\left(Met_0(\Sigma),\delta\right)$ is locally complete.

\subsection{First properties of generalized eigenvalues on an adapted variational space}

We endow $\Sigma$ with a continuous metric $g \in Met_0(\Sigma)$. We let $B$ be the Banach space of symmetric continuous bilinear forms $\beta : H^1(\Sigma)\times H^1(\Sigma) \to \R$ endowed with the norms
$$ \left\Vert \beta \right\Vert_g = \sup_{\varphi,\psi \in H^1 \setminus\{0\}} \frac{\left\vert \beta(\varphi,\psi) \right\vert}{\Vert \varphi \Vert_g \Vert \psi \Vert_g}  $$
where we denote $\Vert \varphi \Vert_{H^1(g)}$ the $H^1$ norm of a function $\varphi$ with respect to the metric $g$:
$$ \Vert \varphi \Vert_{H^1(g)}^2 = \int_{\Sigma} \varphi^2 dA_g + \int_{\Sigma} \left\vert \nabla \varphi \right\vert_g^2 dA_g  $$
if $\Sigma$ is a closed surface and if we study Laplace eigenvalues or
$$ \Vert \varphi \Vert_{H^1(g)}^2 = \int_{\partial\Sigma} \varphi^2 dL_g + \int_{\Sigma} \left\vert \nabla \varphi \right\vert_g^2 dA_g  $$
if $\Sigma$ is a compact surface with boundary and if we study Steklov eigenvalues. Notice that the space $H^1$ is independent of the metric $g$ and that all the norms $\Vert \varphi \Vert_{H^1(g)}$ for $g\in Met_0(\Sigma)$ are equivalent. As a consequence, the space $B$ does not depend on the metric $g$ and all the norms $\Vert \beta \Vert_g$ are equivalent.
We denote $B_+$ the subspace of non-negative bilinear forms of $B$. Let $\beta \in B_+$ and $g \in Met_0(\Sigma)$. We set the $k$-th eigenvalue
$$ \lambda_k(g,\beta) = \inf_{V \in \mathcal{G}_k(V_\beta)} \max_{\varphi \in V\setminus\{0\}} \frac{\int_{\Sigma} \left\vert \nabla \varphi \right\vert^2_g dA_g }{ \beta(\varphi,\varphi) } $$
where $\mathcal{G}_k(V_\beta)$ is the set of $k$-dimensional vector subspace of
$$ V_\beta = \{ \varphi \in \mathcal{C}^\infty(\Sigma), \beta(1,\varphi) = 0 \} $$
Notice that we can replace $V_\beta$ by its closure in $H^1$:
$$ \overline{V_\beta} =  \{ \varphi \in H^1(\Sigma), \beta(1,\varphi) = 0 \}$$
in the definition of $\lambda_k(g,\beta)$. Notice also that $\left[0,+\infty \right]$ is the set of admissible values of $\lambda_k$ on $B_+ $. 
Finally, we set the $k$-th renormalized eigenvalue
$$ \bar{\lambda}_k(g,\beta) = \lambda_k(g,\beta) \beta(1,1). $$
and by convention $\bar{\lambda}_k = 0$ if $\beta(1,1)=0$. Notice that in the case of Steklov eigenvalues on compact surfaces with boundary, $\sigma_k(g,\beta)$ enjoys the same definition.

\begin{prop} \label{prop:Lipeigen}
$\lambda_k$ is an upper semi-continuous functional on
$$ G= Met_0(\Sigma) \times  \{ \beta \in B_+ ; \beta(1,1)\neq 0  \} $$
and 
$\lambda_k$ and $\bar{\lambda}_k$ are locally Lipschitz maps on the open set 
$$ F =  \{(g,\beta) \in Met_0(\Sigma) \times B_+ ; \beta(1,1)\neq 0 \text{ and } \lambda_k(g,\beta) < +\infty \} $$
Moreover, for any $\Lambda >0$,
$$  F_\Lambda =  \{\beta \in Met_0(\Sigma) \times  B_+ ; \bar{\lambda}_k(g,\beta) \leq \Lambda \}  $$
is a closed set in $Met_0(\Sigma) \times B$.
\end{prop}

\begin{proof}

\textbf{Step 1: $\lambda_k$ is upper semi-continuous on $G$.}

\medskip

Let $\beta , \beta_n \in G$ and $g_n, g \in Met_0(\Sigma)$ such that $\beta_n \to \beta$ in $B$ and $g_n \to g$ in $Met_0(\Sigma)$. If $\lambda_k(g,\beta)=+\infty$, then, there is nothing to prove. We assume that $\lambda_k(g,\beta) < +\infty$. Let $V \in \mathcal{G}_k(V_\beta)$ be such that
$$ \max_{\varphi \in V\setminus\{0\}} \frac{\int_{\Sigma} \left\vert \nabla \varphi \right\vert^2_g dA_g }{ \beta(\varphi,\varphi) } \leq \lambda_k(g,\beta)+\delta $$
Then
$$ \lambda_k(g_n,\beta_n) \leq \max_{\varphi \in V \setminus\{0\}} \frac{\int_{\Sigma} \left\vert \nabla \varphi \right\vert^2_{g_n} dA_{g_n}}{\beta_n\left( \varphi - \frac{\beta_n\left(1,\varphi\right)}{\beta_n(1,1)},\varphi - \frac{\beta_n\left(1,\varphi\right)}{\beta_n(1,1)}  \right)} = \max_{\varphi \in V \setminus\{0\}} \frac{\int_{\Sigma} \left\vert \nabla \varphi \right\vert^2_g dA_g (1 + c_g \delta(g_n,g))}{\beta_n\left( \varphi , \varphi \right) - \frac{\beta_n\left(1,\varphi\right)^2}{\beta_n(1,1)}}  $$
Let $\varphi \in V$ be such that $\Vert \varphi \Vert_{H^1(g)} = 1$
$$ \beta_n\left( \varphi , \varphi \right) - \frac{\beta_n\left(1,\varphi\right)^2}{\beta_n(1,1)} \geq \beta(\varphi,\varphi) - \Vert \beta_n - \beta \Vert_g - \frac{\Vert \beta_n - \beta \Vert_g^2}{\beta(1,1) - \Vert \beta_n - \beta \Vert_g } . $$
Since $\lambda_k(g,\beta) <+\infty$, we know that $\beta(\varphi,\varphi)>0$, and that $V$ is a finite-dimensional set,
$$ \inf_{\varphi \in V, \Vert \varphi \Vert = 1} \beta(\varphi,\varphi) > 0 $$
and since $\beta(1,1)\neq 0$, and $\beta_n \to \beta$, we obtain that
$$ \lambda_k(\beta_n) \leq \lambda_k(\beta) + \delta +o(1) $$
as $n\to +\infty$. Letting $n \to +\infty$ and then $\delta \to 0$, we obtain the property.

\medskip

\textbf{Step 2: $\lambda_k$ is continous on $F$ and $F_{\Lambda}$ is closed}

\medskip

Let $(g,\beta) \in F , (g_n,\beta_n) \in F$ be such that $\beta_n \to \beta$ in $B$ and $g_n \to g $ in $Met_0(\Sigma)$. We assume that
$$ \Lambda := \limsup_{n\to +\infty} \lambda_k(g_n,\beta_n) < + \infty. $$
Let $V_n \in \mathcal{G}_k(V_{\beta_n})$ be such that
$$ \max_{\varphi \in V_n\setminus\{0\}} \frac{\int_{\Sigma} \left\vert \nabla \varphi \right\vert^2_{g_n} dA_{g_n} }{ \beta_n(\varphi,\varphi) } \leq \lambda_k(\beta_n)+\delta \leq \Lambda + 2\delta $$
where the last inequality holds for $n$ large enough. Then
$$ \lambda_k(g,\beta) \leq \max_{\varphi \in V_n \setminus\{0\}} \frac{\int_{\Sigma} \left\vert \nabla \varphi \right\vert^2_g dA_g}{\beta\left( \varphi - \frac{\beta\left(1,\varphi\right)}{\beta(1,1)},\varphi - \frac{\beta\left(1,\varphi\right)}{\beta(1,1)}  \right)} = \max_{\varphi \in V_n \setminus\{0\}} \frac{\int_{\Sigma} \left\vert \nabla \varphi \right\vert^2_g dA_g}{\beta\left( \varphi , \varphi \right) - \frac{\beta\left(1,\varphi\right)^2}{\beta(1,1)}}  $$
In the current step, we denote all the $H^1$, $H^{-1}$ norms and norms on $B$ with respect to the metric $g$. Let $\varphi \in V_n$
$$ \beta\left( \varphi , \varphi \right) - \frac{\beta\left(1,\varphi\right)^2}{\beta(1,1)} \geq \beta_n(\varphi,\varphi) - \left( \Vert \beta_n - \beta \Vert_g - \frac{\Vert \beta_n - \beta \Vert_g^2}{\beta(1,1)  } \right) \left\Vert \varphi \right\Vert_{H^1}^2 $$
 We have the following general Poincar\'e inequality (see e.g \cite{Zie89}, lemma 4.1.3]):
$$ \int_{\Sigma}\left(\varphi - \frac{\beta_n(1,\varphi)}{\beta_n(1,1)} \right)^2dA_g \leq C \left\Vert \frac{\beta_n(1,.)}{\beta_n(1,1)} \right\Vert_{H^{-1}}^2 \int_{\Sigma} \left\vert \nabla \varphi \right\vert_g^2 dA_g $$
and we have that 
$$ \int_{\Sigma} \left\vert \nabla \varphi \right\vert_g^2 dA_g \leq  \int_{\Sigma} \left\vert \nabla \varphi \right\vert_{g_n}^2 dA_{g_n} (1 + c_g \delta(g_n,g)) $$
so that knowing that $\varphi \in V_n$, 
\begin{equation*}
\begin{split}
 \Vert \varphi \Vert_{H^1}^2 \leq & \left( C \left\Vert \frac{\beta_n(1,.)}{\beta_n(1,1)} \right\Vert_{H^{-1}}^2 + 1  \right)(1 + c_g \delta(g_n,g))\left( \lambda_k(\beta_n)+ \delta \right) \beta_n(\varphi,\varphi) \\
 \leq & \left( C \left( \frac{\left\Vert\beta\right\Vert_g + \Vert \beta_n - \beta \Vert_g}{\beta(1,1) - \Vert \beta_n - \beta \Vert_g } \right)^2 + 1  \right)(1 + c_g \delta(g_n,g))\left(\Lambda + 2 \delta \right)\beta_n(\varphi,\varphi)
\end{split} 
  \end{equation*}
and we obtain that
$$ \lambda_k(g,\beta) \leq \left(\lambda_k(g_n,\beta_n) + \delta\right)(1 + c_g \delta(g_n,g))(1+o(1)) $$
so that letting $n\to+\infty$ and then $\delta\to 0$, we obtain the expected result.

\medskip

\textbf{Step 3: $\lambda_k$ is locally Lipschitz on $F$}

\medskip

Let $(g,\beta)\in F$. Without loss of generality, we choose again all $H^1$, $H^{-1}$ norms and norms on $B$ with respect to the metric $g$. We set $\Lambda = \lambda_k(g,\beta)+1$. Let $\eps_0$ and let $(g_1,\beta_1),(g_2,\beta_2) \in F_\Lambda \cap B((g,\beta),\eps_0)$ be such that
$$ \max\left\{\left\Vert \beta_1 - \beta_2 \right\Vert_g , \delta(g_1,g_2)\right\} =: \eps \leq 2\eps_0 \text{ and } \sup_{B((g,\beta),\eps_0)} \lambda_k \leq \Lambda. $$
$\eps_0$ exists by continuity of $\lambda_k$. Let $0<\delta<1$ we shall fix later and let $V \in \mathcal{G}_k(V_{\beta_1})$ be such that
$$ \max_{\varphi \in V\setminus\{0\}} \frac{\int_{\Sigma} \left\vert \nabla \varphi \right\vert^2_{g_1} dA_{g_1} }{ \beta_1(\varphi,\varphi) } \leq \lambda_k(g_1,\beta_1)+\delta $$
Then, we test the space 
$$ \tilde{V} := \left\{ \varphi - \frac{\beta_2(1,\varphi)}{\beta_2(1,1)}  ; \varphi \in V \right\} \in \mathcal{G}_k(V_{\beta_2}) $$
in the variational characterization of $\lambda_k(g_2,\beta_2)$:
$$ \lambda_k(g_2,\beta_2) \leq \max_{\varphi \in V \setminus \{0\}}  \frac{\int_{\Sigma} \left\vert \nabla \varphi \right\vert^2_{g_2} dA_{g_2} }{ \beta_2\left(\varphi - \frac{\beta_2(1,\varphi)}{\beta_2(1,1)},\varphi - \frac{\beta_2(1,\varphi)}{\beta_2(1,1)}\right) } $$
for $\varphi\in V$, we have
\begin{equation*} 
\begin{split} \beta_2\left(\varphi - \frac{\beta_2(1,\varphi)}{\beta_2(1,1)},\varphi - \frac{\beta_2(1,\varphi)}{\beta_2(1,1)}\right) = & \beta_1(\varphi,\varphi) + \left(\beta_1-\beta_2\right)(\varphi,\varphi) - \frac{\left(\beta_2-\beta_1\right)(1,\varphi)^2}{\beta_2(1,1)} \\
\geq & \beta_1(\varphi,\varphi) - \left( \Vert \beta_1-\beta_2 \Vert_g + \frac{ \Vert \beta_1-\beta_2 \Vert_g^2 }{\beta(1,1) - 2\eps_0} \right) \Vert \varphi \Vert_{H^1}^2
\end{split}
 \end{equation*}
 We have the following general Poincar\'e inequality:
$$ \int_{\Sigma}\left(\varphi - \frac{\beta_1(1,\varphi)}{\beta_1(1,1)} \right)^2dA_g \leq C \left\Vert \frac{\beta_1(1,.)}{\beta_1(1,1)} \right\Vert_{H^{-1}}^2 \int_{\Sigma} \left\vert \nabla \varphi \right\vert_g^2 dA_g $$
so that knowing that $\varphi \in V$, 
\begin{equation*}
\begin{split}
 \Vert \varphi \Vert_{H^1}^2 \leq & \left( C \left\Vert \frac{\beta_1(1,.)}{\beta_1(1,1)} \right\Vert_{H^{-1}}^2 + 1  \right)(1+ c_g \delta(g,g_1))\left( \lambda_k(\beta_1)+ \delta \right) \beta_1(\varphi,\varphi) \\
 \leq & A_\Lambda(\eps_0) \beta_1(\varphi,\varphi)
\end{split} 
  \end{equation*}
  where
  $$ A_\Lambda(\eps_0):= \left( C \left( \frac{\left\Vert\beta\right\Vert_g + 2\eps_0}{\beta(1,1) - 2\eps_0} \right)^2 + 1  \right)(1+c_g\eps_0)\left(\Lambda + 1 \right), $$
and gathering all the previous inequalities, we obtain
\begin{equation*}
\begin{split} \lambda_k(g_2,\beta_2) \leq & \max_{\varphi \in V \setminus \{0\}} \frac{\int_{\Sigma}\left\vert \nabla \varphi \right\vert^2_{g_2}dA_{g_2}}{\beta_1(\varphi,\varphi)}  \left(1 - \left(\eps + \frac{\eps^2}{\beta(1,1)-2\eps_0}\right) A_\Lambda(\eps_0) \right)^{-1}\\
\leq & \left(\lambda_k(\beta_1)+\delta\right)(1+c_g \eps) ( 1 -  C_{\Lambda}(\eps_0) \eps)^{-1}
\end{split}\end{equation*}
where $ C_\Lambda(\eps_0) = \left(1 + \frac{2\eps_0}{\beta(1,1)-2\eps_0}\right)A_\Lambda(\eps_0) $. Choosing $\eps_0 < \beta(1,1)$ such that $C_\Lambda(\eps_0)\eps_0 \leq \frac{1}{2}$, we obtain
$$ \lambda_k(g_2,\beta_2) \leq \left(\lambda_k(g_1,\beta_1)+\delta\right)\left(1 + 2 C_\Lambda(\eps_0) \eps \right)(1+c_g \eps) $$
Now, letting $\delta \to 0$, we obtain
$$ \lambda_k(g_2,\beta_2)- \lambda_k(g_1,\beta_1) \leq \Lambda \left(2  C_\Lambda(\eps_0)+2c_g \right) d\left( (g_1,\beta_1),(g_2,\beta_2) \right) $$
Exchanging $\beta_1$ and $\beta_2$, the same argument leads to
$$ \left\vert \lambda_k(g_2,\beta_2)- \lambda_k(g_1,\beta_1) \right\vert \leq  \Lambda \left( 2 C_\Lambda(\eps_0) + 2c_g \right) d\left(g_1, \beta_1), (g_2,\beta_2) \right). $$
\end{proof}

For $g \in Met_{0}(\Sigma)$, we set $\overline{X}$ the closure of $X$ in $B$ where
$$ X = \left\{ (\varphi,\psi)\in H^1\times H^1 \mapsto \int_{\Sigma} e^{2u} \varphi\psi dA_g ; u \in \mathcal{C}^0\left( \Sigma \right) \right\} $$
in the context of Laplace eigenfunctions and
$$ X = \left\{ (\varphi,\psi)\in H^1\times H^1 \mapsto \int_{\partial\Sigma} e^{u} \varphi\psi dL_g ; u \in \mathcal{C}^0\left( \Sigma \right) \right\} $$
in the context of Steklov eigenfunctions. Notice that if $\tilde{g}$ is another metric, $dA_{\tilde{g}}$ is absolutely continuous with respect to $dA_{g}$ with a continuous density and $dL_{\tilde{g}}$ is absolutely continuous with respect to $dL_{g}$ with a continuous density. Therefore, $X$ and $\overline{X}$ are independent of the choice of the metric.

We denote $Q_+$ the set of squares of $H^1$ functions and $Q = Span(Q_+)$. 
One immediate property of $\beta \in \overline{X}$ is that $\beta$ acts as a linear map on $Q$.

\begin{prop} For any $\beta \in \overline{X}$, there is a unique linear map $L_\beta : Q \to \R$ such that
$$ \forall \phi,\psi \in H^1\left(\Sigma\right), L_{\beta}\left( \phi\psi \right) = \beta\left(\phi,\psi\right) $$
and in particular
$$ \forall \phi \in H^1\left(\Sigma\right), L_{\beta}\left( \phi^2 \right) = \beta\left( \phi , \phi \right) \geq 0. $$
In addition, $L_\beta \in H^{-1}$ and $L_\beta \in Mes_+(\Sigma)$, where $Mes_+(\Sigma)$ is the set of non negative Radon measures (dual set of $\mathcal{C}^0(\Sigma)$) in the sense that there is a unique extension of $L_\beta : H^1 \cap \mathcal{C}^0 \to \R$ to $\mathcal{C}^0(\Sigma)$.
\end{prop}

\begin{proof} Let $\theta \in Q$. Let $\{ \phi_i \}_{i\in I}$ and $\{ \psi_j \}_{j\in J}$ two finite families of $H^1$ functions and $\{ t_{i} \}_{i\in I}$ and $\{ s_j \}_{j\in J}$ associated families of real numbers such that
$$\theta = \sum_{i\in I} t_i \phi_i^2 = \sum_{j\in J} s_j \psi_j^2  $$
Then it is clear that
\begin{equation} \label{eqlinearity} \sum_{i\in I} t_i \beta\left(\phi_i,\phi_i\right) = \sum_{j\in J} s_j \beta\left(\psi_j,\psi_j\right). \end{equation}
Indeed, if $e^{2u_k}$ converges to $\beta$ in $B$. 
$$ \sum_{i\in I} t_i \int_\Sigma e^{2u_k} \phi_i^2 dA_g = \sum_{j\in J} s_j \int_\Sigma e^{2u_k} \psi_j^2 dA_g  $$
and letting $k\to +\infty$, we easily deduce \eqref{eqlinearity} (the Steklov case is analogous). Then we can set a unique linear map $L_\beta : Q \to \R$ such that
$$ \forall \phi \in H^1(\Sigma), L_{\beta}(\phi^2) = \beta(\phi,\phi). $$
Of course, $\beta(\phi,\phi)\geq 0$. More generality, we compute that
$$ L_{\beta}( 4 \phi \psi) = L_{\beta}( (\phi+\psi)^2 - (\phi - \psi)^2 ) = \beta(\phi+\psi,\phi+\psi) - \beta(\phi-\psi,\phi-\psi) = 4\beta(\phi,\psi). $$
It remains to prove that $L_\beta \in Mes_+(\Sigma)$. Let $\varphi \in H^1 \cap \mathcal{C}^0$. Then
$$ \vert L_\beta(\varphi)\vert = \vert\beta(1,\varphi)\vert = \left\vert\lim_{k\to +\infty} \int_{\Sigma} e^{2u_k} \varphi \right\vert \leq \Vert \varphi \Vert_{C^0} \lim_{k\to +\infty} \int_{\Sigma} e^{2u_k} = \Vert \varphi \Vert_{C^0} L_\beta(1).  $$
By unique extension of continuous linear operators, we obtain the expected result.
\end{proof}

We also obtain the immediate corollary for eigenvalues.
\begin{cor}
$$ \sup_{g\in Met_0(\Sigma), \beta\in \overline{X} \setminus \{0\} } \bar{\lambda}_k(g,\beta) = \sup_{g\in Met_0(\Sigma)} \sup_{\beta\in X \setminus \{0\} } \bar{\lambda}_k(g,\beta)  = \sup_{g\in Met_0(\Sigma)} \bar{\lambda}_k(g,1) < + \infty $$
\end{cor}
The finiteness can be deduced from Korevaar \cite{korevaar}: he proved the latter strict inequality among smooth metrics $g$ but it is still true for continuous metrics by smooth approximation. The equalities are consequences of Proposition \ref{prop:Lipeigen} and the definition of $X$.

We also have the very useful compactness property of bilinear forms in $\overline{X}$
\begin{prop} \label{prop:compactnessgeneigen}
Let $c,c'>0$. Let $\beta \in \overline{X}$ be such that $\beta(1,1)\neq 0$, then the image of 
$$ S_{c,c'} = \{ (\phi,\psi) \in H^1 \times H^1 ; \Vert \phi \Vert_{H^1}^2 \leq c \text{ and } \Vert \psi \Vert_{H^1}^2 \leq c'   \} $$
and of
$$ \tilde{S}_{\beta,c,c'} = \{ (\phi,\psi) \in \overline{V_{\beta}} \times \overline{V_{\beta}} ; \int_\Sigma \vert \nabla \phi \vert_g^2dA_g \leq c \text{ and } \int_\Sigma \vert \nabla \psi \vert_g^2dA_g \leq c'   \} $$
by $\beta$ is a compact set. More generally if $(\beta_n) \in \overline{X}$ satisfies $\beta_n \to \beta$ in $\overline{X}$ and if $(\phi_n,\psi_n) \in \tilde{S}_{\beta_n ,c,c'}$, then there is a subsequence $(\phi_{j(n)},\psi_{j(n)})$ that converges weakly to $ (\phi,\psi) \in \tilde{S}_{\beta ,c,c'}$ in $H^1 \times H^1$ and such that
$$ \beta_{j(n)}(\phi_{j(n)},\psi_{j(n)}) \to \beta(\phi,\psi) $$
as $n\to +\infty$
\end{prop}

\begin{proof}
We first notice that if $\phi \in \overline{V_{\beta_n}}$, then by the Poincaré inequality,
$$ \Vert \phi \Vert_{L^2}^2 \leq C \left\Vert \frac{\beta_n(1,.)}{\beta_n(1,1)} \right\Vert_{H^{-1}}^2 \int_{\Sigma}\left\vert \nabla \phi \right\vert^2_g dA_g $$
So that setting $a = \left(1 + C\left(\left\Vert \frac{\beta(1,.)}{\beta(1,1)} \right\Vert_{H^{-1}}^2 +1 \right) \right) c $ and $b = \left(1 + C\left(\left\Vert \frac{\beta(1,.)}{\beta(1,1)} \right\Vert_{H^{-1}}^2 +1\right)\right) c'  $, we obtain that $\tilde{S}_{\beta_n,c,c'} \subset S_{a,b}$ for $n$ large enough.

Let $(\phi_n,\psi_n) \in H^1 \times H^1$ be such that $ \Vert \phi_n \Vert_{H^1} \leq c$ and $\Vert \psi_n\Vert_{H^1} \leq c'$. By the weak compactness of the ball of $H^1$, up to the extraction of a subsequence, we have that $\phi_n$ and $\psi_n$ weakly converge to $\phi$ and $\psi$ in $H^1$. We aim at proving that
$$ \beta_n(\phi_n,\psi_n) \to \beta(\phi,\psi) $$
as $n\to +\infty$. Let $\delta>0$. Since $\beta \in \overline{X}$, there is a smooth positive function $e^{2u}$ such that $ \left\Vert \beta - e^{2u} \right\Vert \leq \delta $. By the compact injection of $W^{1,2} \subset L^{2}(e^{2u}g)$, we have up to the extraction of a subsequence that $\psi_n \to \psi$ and $\phi_n \to \phi$ in $L^2(e^{2u}g)$ so that
$$ \int_{\Sigma} \phi_n\psi_n e^{2u} dA_g \to \int_{\Sigma} \phi\psi e^{2u}dA_g. $$
We obtain that
$$ \left\vert \beta_n(\phi_n,\psi_n) - \beta(\phi,\psi) \right\vert \leq \left\vert \int_{\Sigma} \phi_n\psi_n e^{2u} dA_g - \int_{\Sigma} \phi\psi e^{2u}dA_g \right\vert + \left(  \Vert \beta_n - \beta \Vert + 2 \Vert \beta - e^{2u} \Vert \right) c c' $$
so that passing to the limit as $n \to +\infty$,
$$ \limsup_{n\to +\infty} \left\vert \beta_n(\phi_n,\psi_n) - \beta(\phi,\psi) \right\vert \leq \delta c c' $$
and letting $\delta\to 0$, we obtain the expected result.
\end{proof}

Notice also that all the norms $N_{g,\beta}(\phi)^2 := \int_\Sigma \left\vert \nabla \phi \right\vert^2_g +\beta(\phi,\phi)$ satisfy for $(g,\beta) \in Met_0(\Sigma)\times \overline X$ the existence of an open neighborhood $U_{g,\beta}$ and a constant $C_{g,\beta}$ such that
\begin{equation} \label{eq:equivalencenorms} \forall (\tilde{g},\tilde{\beta}) \in U_{g,\beta}, \forall \phi \in H^1, C_{g,\beta}^{-1} N_{\tilde{g},\tilde{\beta}}(\phi)^2 \leq N_{g,\beta}(\phi)^2 \leq C_{g,\beta} N_{\tilde{g},\tilde{\beta}}(\phi)^2 \end{equation}

\subsection{Existence of eigenfunctions and first variation of generalized eigenvalues}

In \cite{pt}, we obtain from the compactness property (Proposition \ref{prop:compactnessgeneigen}) that the spectrum associated to $(g,\beta)\in Met_0(\Sigma)\times \overline{X}$ is discrete, that is
$$ 0 = \lambda_0 \leq \lambda_1(g,\beta) \leq \lambda_2(g,\beta) \leq \cdots \leq \lambda_k(g,\beta) \to +\infty \text{ as } k\to +\infty $$ 
and in particular that the multiplicity of eigenvalues is finite and that there is a Hilbert basis (with respect to $\beta(\cdot,\cdot)$ or $N_{g,\beta}$) of eigenfunctions. Notice that an equation on eigenfunctions 
$$ \Delta_g \varphi = \lambda \beta(\varphi,\cdot) $$
does not provide more regularity of eigenfunctions than $H^1$ and has to be read in the weak sense with respect to $g$: $\Delta_g \varphi \in H^{-1}$ is the map $\psi \in H^1 \mapsto \int_{\Sigma} \langle\nabla \varphi\nabla\psi\rangle_g dA_g$. The same notation can be used in the Laplace and Steklov case. Notice also that if $\Sigma$ is connected, $\lambda_0 = 0$ is a simple eigenvalue associated to the constant functions.

As soon as $(g,\beta)$ belongs to the interior of $Met_0(\Sigma)\times \overline{X}$ Proposition \eqref{prop:compactnessgeneigen} and the norm equivalence \eqref{eq:equivalencenorms} also provide computations of the directional derivatives, the generalized directional derivatives, the classical subdifferential, and the Clarke subdifferential \cite{clarke} of 
$$E: (g, \beta) \mapsto F(\bar{\lambda}_1(g,\beta),\cdots,\bar{\lambda}_m(g,\beta)) $$
where $F : \left(\mathbb{R}_{+}^\star \right)^m \to \mathbb{R}_{+}^\star $ such that $\partial_i F \leq 0$ for any $i$. 
$$ \partial E(g,\beta) \subset co \left\{ \sum_{i=1}^m d_i(g,\beta) \bar{\lambda}_i(g,\beta) \left((\phi_i,\phi_i) - (1,1) \right) ; (\phi_1,\cdots,\phi_m)\in \mathbf{O}_m(\beta) \right\}$$
where $d_i(g,\beta) = \partial_i F(\bar{\lambda}_1(g,\beta),\cdots,\bar{\lambda}_m(g,\beta))$ and $\mathbf{O}_m(\beta)$ is the set of orthonormal families $(\phi_1,\cdots,\phi_m)$ with respect to $\beta$ such that for all $1\leq i \leq m$, $\phi_i$ is an eigenfunction with respect to $\lambda_i(g,\beta)$.

In the current paper, we will need right directional derivatives on points $(g,\beta)$ that do not belong to the interior of $Met_0(\Sigma)\times \overline{X}$. For that reason, we will not use the Clarke supdifferential in the current paper. However, the abstract analysis in \cite{pt} also works as soon as the variation $(g,\beta)+ t (h,b)$ belongs to the admissible set as $t \searrow 0$. For the sake of completeness, we write this computation in our context:

We denote by 
$$ i(k) := \min\{ i \in \mathbb{N}^* ; \lambda_i = \lambda_k \}$$
$$ I(k) := \max\{ i \in \mathbb{N}^* ; \lambda_i = \lambda_k \}$$

\begin{prop} \label{prop:firstderivative}
For $(g,\beta) \in Met_0(\Sigma)\times \bar{X}$, and $(h,b) \in  S_0^2(\Sigma) \times \bar{X}$,
\begin{equation} \label{eq:firstderivativeminmax}
\begin{split} \lim_{t \searrow 0 } \frac{\bar{\lambda}_k(g+th,\beta+tb) - \bar{\lambda}_k(g,\beta)}{t} = & \min_{ V \in \mathcal{G}_{k-i(k)+1}(E_k(g,\beta)) } \max_{\phi \in V\setminus \{0\}} \mathcal{Q}_{(h,b)}(\phi) \\ 
= & \max_{ V \in \mathcal{G}_{I(k)-k+1}(E_k(g,\beta)) } \min_{\phi \in V\setminus \{0\}}  \mathcal{Q}_{(h,b)}(\phi)
\end{split}
\end{equation}
where 
$$ \mathcal{Q}_{(h,b)}(\phi) =  \int_\Sigma \left( \frac{\left\vert \nabla \phi \right\vert^2_{g}}{2} g - d\phi \otimes d\phi ,h \right)_{g}dA_{g}  +  \bar{\lambda}_k(g,\beta) \left(b( 1,1 ) -  b(\phi,\phi) \right)$$
\end{prop}

\begin{proof} The right-hand terms are equal as a consequence of the min-max formula for the quotients of a quadratic form by a positive definite quadratic form on finite-dimensional spaces. Notice that from Proposition \ref{prop:Lipeigen}, we have that $\lambda_k(\beta+tb) \to \lambda_k(\beta)$ as $t \searrow 0$. 

We denote by 
$$\phi_{i(k)}^t,\cdots, \phi_{I(k)}^t$$
a family of $\beta$-orthonormal eigenfunctions associated to the eigenvalues 
$$\lambda_{i(k)}(g+th,\beta+tb) \leq \cdots \leq \lambda_{I(k)}(g+th,\beta+tb)$$
 we rename $\lambda_{i(k)}^t \leq \cdots \leq \lambda_{I(k)}^t$ that all converge to $\lambda_k := \lambda_k(g,\beta)$ as $t\to 0$.
Up to the extraction of a subsequence as $t \to 0$, $\phi_i^t$ converges to $\phi_i$ weakly in $H^1$, and 
$$ \left(\beta+tb\right)(\phi-\phi_i^t,\phi-\phi_i^t) \to 0$$
 as $t\to 0$. Passing to the weak limit on the equation satisfied by $\phi_i^t$ and to the strong limit on $\left(\beta+tb\right)(\phi_i^t,\phi_j^t) = \delta_{i,j}$, we obtain
$$ \Delta_g \phi_i = \lambda_k \beta( \phi_i,\cdot) \text{ and } \beta(\phi_i,\phi_j) = \delta_{i,j} $$
for $i(k)\leq i,j\leq I(k)$. Integrating the equation with respect to $\phi_i$ proves that
\begin{equation*}
\begin{split} \int_M \vert \nabla \phi_i \vert_g^2 dA_g & =  \lambda_k \beta(\phi_i,\phi_i) 
=  \lim_{t\to 0}  \lambda_i^t \beta(\phi_i^t,\phi_i^t) = \lim_{t\to 0} \int_M \vert \nabla \phi_i^t \vert_g^2 dA_g 
\end{split} \end{equation*}
so that $\phi_i^t$ converges strongly in $H^1$.

\medskip 

For $i(k) \leq i \leq I(k)$. We set $R_i^t := \phi_i^t - \pi_k(\phi_i^t) $ where for $v\in H^1$
$$ \pi_k(v) := v - \sum_{i=i(k)}^{I(k)} \beta(v,\phi_i) \phi_i $$
is the orthogonal projection of $v$ on $E_k(g,\beta)$ with respect to $\beta$. We have for $v \in H^1$
\begin{equation*}
\begin{split} \int_\Sigma \langle \nabla R_i^t \nabla v \rangle_g dA_g  - \lambda_k \beta( R_i^t,v) = & \lambda_i^t (\beta + tb)( \phi_i^t,v) - \lambda_k \beta( \phi_i^t,v) \\
 & + \left(\int_\Sigma \langle \nabla v \nabla \phi_i^t \rangle_g dA_g - \int_\Sigma \langle \nabla v \nabla \phi_i^t \rangle_{g+th} dA_{g+th}\right) \\
= & (\lambda_i^t - \lambda_k) \beta( \phi_i^t,\cdot) + \lambda_i^t tb( \phi_i^t,\cdot) \\
 & + \left(\int_\Sigma \langle \nabla v \nabla \phi_i^t \rangle_g dA_g - \int_\Sigma \langle \nabla v \nabla \phi_i^t \rangle_{g+th} dA_{g+th}\right) 
\end{split}
\end{equation*}
so that setting 
\begin{equation} \label{eq:defialphait} \alpha_i^t := \left\vert \lambda_i^t - \lambda_k \right\vert + t + \sqrt{\beta( R_i^t ,R_i^t)} \end{equation}
and
\begin{equation} \label{eq:defitilderit} \tilde{R}_i^t = \frac{R_i^t}{\alpha_i^t} \hspace{5mm} \tau_i^t = \frac{t}{\alpha_i^t} \hspace{5mm} \delta_i^t := \frac{\lambda_i^t - \lambda_k}{\alpha_i^t}, \end{equation}
Let's prove that $\tilde{R}_i^t $ is bounded in $H^1$. Let $v \in H^1$, we have that
\begin{equation*}
\begin{split} \int_\Sigma \nabla \tilde{R}_i^t \nabla v dA_g = & \lambda_k \beta(\tilde{R}_i^t,v) + \delta_i^t \beta(\phi_i^t,v) + \lambda_i^t b(\phi_i^t,v) \\ 
& + \frac{1}{t} \left(\int_\Sigma \langle \nabla v \nabla \phi_i^t \rangle_g dA_g - \int_\Sigma \langle \nabla v \nabla \phi_i^t \rangle_{g+th} dA_{g+th}\right)
\end{split}  \end{equation*}
so that
\begin{equation*}
\begin{split} & \left\vert \int_\Sigma \nabla \tilde{R}_i^t \nabla v dA_g + \beta( \tilde{R}_i^t, v) \right\vert \\ & \leq \left( \left(\lambda_k + 1\right) \sqrt{\beta(\tilde{R}_i^t,\tilde{R}_i^t )} \Vert \beta \Vert  + \left(\delta_i^t \Vert \beta \Vert + \lambda_i^t \Vert b \Vert + C \right) \Vert  \phi_i^t \Vert_{H^1}  \right) \Vert v \Vert_{H^1}
\end{split}  \end{equation*}
so that by the Riesz theorem associated to the Hilbert norm $N_\beta$, and the equivalence of the $H^1$ norm and the $N_\beta$ norm, we obtain that
$ \tilde{R}_i^t $ is bounded with respect to $N_\beta$ as $t\to 0$.
By equivalence between the $H^1$ norm and the norm $N_\beta$, again, $ \tilde{R}_i^t $ is bounded in $H^1$. Then, up to the extraction of a subsequence as $t\to 0$,
$$ \tilde{R}_i^t \to \tilde{R}_i \text{ weakly in } H^1 \hspace{5mm} \tau_i^t \to \tau_i \hspace{5mm} \delta_i^t \to \delta_i.$$
Passing to the weak limit in the equation, we obtain for $v\in H^1$
\begin{equation} \label{eq:limritilde} \begin{split} \int_\Sigma\langle \nabla v \nabla \tilde{R}_i\rangle_g dA_g - \lambda_k \beta(  \tilde{R}_i , v) = & \delta_i \beta( \phi_i,v ) + \tau_i \lambda_k b( \phi_i, v) \\ & + \int_\Sigma \left(dv \otimes d\phi_i - \frac{ \langle\nabla v \nabla \phi_i \rangle_g}{2} g , h  \right)_g dA_g
\end{split} \end{equation}
In addition, up to the extraction of a subsequence, 
$$\beta(\tilde{R}_i^t - \tilde{R}_i,\tilde{R}_i^t - \tilde{R}_i) \to 0$$
as $t\to 0$ and we obtain because of the definitions \eqref{eq:defialphait} and \eqref{eq:defitilderit}
\begin{equation} \label{eq:sumRitildedeltatau} \beta(\tilde{R}_i,\tilde{R}_i) + \vert \delta_i \vert + \tau_i = 1 \end{equation}
Now, we integrate \eqref{eq:limritilde} against $\phi_i$ and we obtain that
\begin{equation} \label{eq:eqriintegrated}
\delta_i \beta(\phi_i,\phi_i) + \tau_i \lambda_k  b( \phi_i,\phi_i ) + \int_\Sigma \left(d\phi_i \otimes d\phi_i - \frac{ \vert \nabla \phi_i \vert_g^2}{2} g , h  \right)_g dA_g = 0.
\end{equation}
Now, we assume by contradiction that $\tau_i = 0$, then by \eqref{eq:eqriintegrated}, $\delta_i=0$ and by \eqref{eq:limritilde}, $\tilde{R}_i \in E_k(\beta) \cap E_k(\beta)^{\perp_{Q(\beta,\cdot)}} = \{0\}$. This contradicts \eqref{eq:sumRitildedeltatau}. Therefore $\tau_i \neq 0$ and
$$ \frac{\delta_i}{\tau_i} = \frac{-\lambda_k b(\phi_i,\phi_i) - \int_\Sigma \left(d\phi_i \otimes d\phi_i - \frac{ \vert \nabla \phi_i \vert_g^2}{2} g , h  \right)_g dA_g }{\beta(\phi_i,\phi_i)} $$
Integrating \eqref{eq:limritilde} against $\phi_j$ for $j\neq i$, we obtain that
$$ B_{(h,b)}(\phi_i,\phi_j):= - \lambda_k b( \phi_i, \phi_j ) - \int_\Sigma \left(d\phi_i \otimes d\phi_j - \frac{ \langle \nabla \phi_i \nabla \phi_j \rangle_g}{2} g , h  \right)_g dA_g = 0 $$
so that $\phi_{i(k)},\cdots,\phi_{I(k)}$ are nothing but an orthonormal basis with respect to $\beta$ that is orthogonal with respect to $ B_{(h,b)} $. Since in addition we have that $\frac{\delta_{i(k)}}{\tau_{i(k)}}\leq \cdots \leq \frac{\delta_{I(k)}}{\tau_{I(k)}}$, classical min-max formulae for orthonormal diagonalization give
$$ \frac{\delta_i}{\tau_i} = \min_{ V \in \mathcal{G}_{i-i(k)+1}(E_k(\beta)) } \max_{v \in V\setminus \{0\}} \frac{B_{(h,b)}(v,v)}{\beta(v,v)} $$
Since the right-hand term is independent of the choice of the subsequence as $t\to 0$, we obtain that the directional derivative exists and
$$ \lim_{t\searrow 0} \frac{\lambda_i^t-\lambda_i}{t} = \lim_{t\searrow 0} \frac{\delta_i^t}{\tau_i^t} = \frac{\delta_i}{\tau_i} $$
completes the proof of the proposition.
\end{proof}

\subsection{Regularization of minimizing sequences into minimizing Palais-Smale-like sequences}

Let $\delta_\eps >0$ be such that $\delta_\eps \to 0$ as $\eps \to 0$ and $(\tilde{g}_\eps,\tilde{\beta}_{\eps}) \in \mathcal{A}$ be such that
\begin{equation} \label{eq:tildeminseq} E(\tilde{g}_\eps,\tilde{\beta}_\eps) \leq \inf E + \delta_\eps^2, \end{equation}
where the choice of $\delta_\eps$ will depend on the construction of the initial minimizing sequence. We would like to transform this minimizing sequence into a Palais-Smale sequence using the Ekeland variational principle \cite{ekeland}.

Since this fonctional is lower semi-continuous in the complete set 
$$\mathcal{A}_\eps = \{g\in Met_0(\Sigma); \delta(g,\tilde{g}_\eps)\leq 1\} \times \{\beta\in \overline{X}, \beta(1,1) \geq 1\}$$
where $\mathcal{A}_\eps$ is endowed with the distance 
$$d_\eps((g_1,\beta_1),(g_2,\beta_2)) = \max \left( \delta(g_1,g_2) ; \Vert \beta_1-\beta_2\Vert_{\tilde{g}_\eps} \right)$$
the Ekeland variational principle gives the existence of $(g_\eps,\beta_{\eps}) \in \mathcal{A}_\eps$ such that
\begin{equation} \label{eq:disttilderegularized} d_\eps((\tilde{g}_\eps,\tilde{\beta}_{\eps})  ,(g_\eps,\beta_{\eps}) ) \leq \delta_\eps \end{equation}
and for any $(g,\beta)\in \mathcal{A}_\eps$,
\begin{equation}\label{eq:mainekeland} E(g_\eps,\beta_\eps) - E(g,\beta) \leq \delta_\eps d_\eps\left((g_\eps,\beta_\eps),(g,\beta)\right).\end{equation}

Now, we prove the existence of a quasi-Euler-Lagrange equation for the Palais-Smale sequence $(g_\eps,\beta_\eps)$. We set $\lambda_i^\eps = \lambda_i(g_\eps,\beta_\eps)$ the $i$-th Laplace generalized eigenvalue. In the following proposition we can replace $\lambda_i^\eps$ by the $\sigma_i^\eps := \sigma_i(g_\eps,\beta_\eps) $ in the context of generalized Steklov eigenvalues and the proof is very similar if in Step 3 we replace integrals of $b_\Phi V$ and $\theta^2 V$ in $\Sigma$ by integrals on $\partial\Sigma$.

\begin{prop} \label{prop:eullag}
There is a map $\Phi_\eps : \Sigma \to \mathbb{R}^{±n_\eps} \in H^1\left(\Sigma, \mathbb{R}^{n_\eps}\right)$ such that
$$ \Delta_{g_\eps} \Phi_\eps = \beta_\eps\left( \Lambda_\eps\Phi_\eps,.\right)  $$
where $\Lambda_\eps = diag(\lambda_1^\eps,\cdots,\lambda_{k}^\eps,\cdots,\lambda_k^\eps) \in \mathcal{M}_{n_\eps}(\R)$ and $\beta_\eps\left( \Lambda_\eps\Phi_\eps,.\right) : H^1(\Sigma,\mathbb{R}^{n_\eps}) \to \R$ 
and 
$$ \left\vert \Phi_\eps \right\vert_{\Lambda_\eps}^2 \geq_{a.e} 1 - \theta_\eps^2 \text{ in } \Sigma \text{ (resp. in } \partial \Sigma \text{ in the Steklov case)} $$
where $\Vert \theta_\eps \Vert_{H^1(\tilde{g}_\eps)}^2 \leq c_\eps \delta_\eps$ and
$$\forall h \in S^2_0(\Sigma), \left\vert \int_{\Sigma}  \left(d\Phi_\eps \otimes d\Phi_\eps - \frac{\left\vert \nabla \Phi_\eps \right\vert_{g_\eps}^2}{2} g_\eps , h\right)_{g_\eps} dA_{g_\eps} \right\vert \leq 2c_\eps \delta_\eps \Vert h \Vert_{g_\eps}$$
where $\Vert h \Vert_{g_\eps} = \sup_{x\in \Sigma} \sqrt{ \langle h,h\rangle_{g_\eps}(x)}$ and 
$$
\beta_\eps\left( \Lambda_\eps \Phi_\eps,\Phi_\eps \right)= \beta_\eps(1,1) =  1 + O(\delta_\eps) $$
and
$$ c_\eps = \left( \sum_{i=1}^m - \lambda_i^\eps \partial_i F(\lambda_1^\eps,\cdots,\lambda_m^\eps)\right)^{-1} $$
In addition, we have that for any $1 \leq k \leq m$,
\begin{equation} \label{eq:sumtieps} \sum_{i \in \{ j ; \lambda_j^\eps = \lambda_k^\eps \}} \beta_\eps(\phi_i^\eps,\phi_i^\eps) = \sum_{i \in \{ j ; \lambda_j^\eps = \lambda_k^\eps \}} t_i^\eps \end{equation}
where for $1\leq i \leq m$,
\begin{equation} \label{eq:deftieps} t_i^\eps := - c_\eps \cdot \partial_i F(\lambda_1^\eps,\cdots,\lambda_m^\eps). \end{equation}
\end{prop}

\begin{proof} We follow several steps. Step 1 follows from the computation of a right directional derivative. Step 2 is a direct consequence of Step 1 and Ekeland's variational principle. Step 3 uses a Hahn-Banach separation argument. 

\medskip
\noindent\textbf{Step 1:} Let $(h,b) \in S^2_0(\Sigma)\times \overline{X}$, then there is $(\phi_1,\cdots,\phi_m) \in \mathbf{O}_m(g_\eps,\beta_\eps)$ such that
\begin{equation} \lim_{t \downarrow 0} \frac{E(g_\eps+th,\beta_\eps+tb) - E(g_\eps,\beta_\eps)}{t} = Q_{(h,b)}(\phi_1,\cdots,\phi_m) \end{equation}
where
\begin{equation} \begin{split} Q_{(h,b)}(\phi_1,\cdots,\phi_m)  := \sum_{i=1}^{m} -d_i^\eps \int_\Sigma \left( \frac{\left\vert \nabla \phi_i \right\vert^2_{g_\eps}}{2} g_\eps - d\phi_i \otimes d\phi_i,h \right)_{g_\eps}dA_{g_\eps} \\ + \sum_{i=1}^{m} -d_i^\eps  \lambda_i^\eps \left(b( 1,1 ) -  b(\phi_i,\phi_i) \right)  \end{split}
\end{equation}
and
$$ d_i^\eps = d_i(g_\eps,\beta_\eps) := -\partial_i F( \lambda_1(g_\eps,\beta_\eps),\cdots,\lambda_m(g_\eps,\beta_\eps) ) > 0. $$

\medskip

\noindent\textbf{Proof of Step 1:} It is a straightforward consequence of proposition \ref{prop:firstderivative} by a chain rule on directional derivatives.

\medskip
\noindent\textbf{Step 2:} For all $(h,b) \in S^2_0(\Sigma)\times \overline{X}$, there is $\Phi:= (\phi_1,\cdots,\phi_m) \in \mathbf{O}_m(g_\eps,\beta_\eps)$ such that
$$ Q_{(h,b)}(\Phi) \geq -\eps \max( \Vert h \Vert_{g_\eps} , \Vert b \Vert_{\tilde{g}_\eps} ) $$

\medskip
\noindent\textbf{Proof of Step 2:} Let $(h,b) \in S^2_0(\Sigma)\times \overline{X}$, then we test $(g_\eps + t h,\beta_\eps+tb)$ in \eqref{eq:mainekeland} for $t>0$ small enough and divide by $t$. We obtain
$$ \frac{E(g_\eps,\beta_\eps) - E(g_\eps+th,\beta_\eps+th)}{t} \leq \eps \max\left( \frac{\delta(g_\eps,g_\eps+th)}{t}, \Vert b \Vert_{\tilde{g}_\eps} \right) $$
Letting $t\searrow 0$, Step 1 gives the existence of $\Phi:= (\phi_1,\cdots,\phi_m) \in \mathbf{O}_m(g_\eps,\beta_\eps)$ such that the left-hand term converges to $-Q_{(h,b)}(\Phi)$. For the convergence of the right-hand term, we just use that $\delta$ is nothing but the geodesic distance on the set $Met_0(\Sigma)$ endowed with the Riemannian metric defined by $\Vert h \Vert_g$ on the tangent space $S_0^2(\Sigma)$ of $g \in Met_0(\Sigma)$.

\medskip
\noindent\textbf{Step 3:} We prove that $K \cap F \neq \emptyset$ where $K$ and $F$ are two subsets of $(S^2_0(\Sigma))^\star \times L^2(\tilde{g}_\eps)$
$$ K = co \left\{ (a_\Phi,b_\Phi)  + (L,\theta^2)   ; \Phi \in \mathbf{O}_m(g_\eps,\beta_\eps) ; \theta \in H^1(\tilde{g}_\eps) ;  \Vert \theta \Vert_{H^1(\tilde{g}_\eps)}^2 \leq \delta_\eps ;  \Vert L \Vert_{g_\eps}^\star \leq 2\delta_\eps \right\} $$
where we denote for $\Phi \in \mathbf{O}_m(g_\eps,\beta_\eps)$
$$ (a_\Phi,b_\Phi):= \sum_{i=1}^m d_i^\eps  \left(  d\phi_i\otimes d\phi_i -  \frac{\left\vert \nabla \phi_i \right\vert^2_{g_\eps}}{2} g_\eps  , \lambda_i^\eps \left( \phi_i^2 - 1 \right) \right) $$
and for $L \in (S_0^2(\Sigma)),\Vert \cdot \Vert_{g_\eps})^\star$ the dual norm as
$$ \Vert L \Vert_{g_\eps}^\star := \sup_{h \in S^2_0(\Sigma)} \frac{ \left\vert \langle L, h \rangle \right\vert}{\Vert h \Vert_{g_\eps} }$$
and 
$$ F = \{ ( 0 , f  )  ;  f \geq 0   \} $$

\medskip

\noindent\textbf{Proof of Step 3:} Notice that $S^2_0(\Sigma)$ can be seen as a subset of $(S^2_0(\Sigma))^\star$ via the injection
$$ k \in S^2_0(\Sigma) \mapsto \left( h \mapsto \int_\Sigma \left(k,h\right)_{g_\eps}dA_{g_\eps}\right) \in  \left(S^2_0(\Sigma)\right)^\star$$
We assume that $K \cap F = \emptyset$. Then, by Hahn-Banach separation theorem, there are $(h,V) \in S^2_0(\Sigma)\times L^2(\tilde{g}_\eps)$ and $\kappa \in \R$ such that
$$ f \in L^2(\tilde{g}_\eps), f\geq 0 ;  \int_{\Sigma} Vf dA_{\tilde{g}_\eps} \geq -\kappa $$
\begin{equation*}
\begin{split} \forall \Phi \in \mathbf{O}_m(g_\eps,\beta_\eps), \forall L \in (S^2_0(\Sigma))^\star, \Vert L \Vert_{g_\eps}^\star \leq 2\delta_\eps ; \forall \theta \in H^1(\tilde{g}_\eps), \Vert \theta \Vert_{H^1(\tilde{g}_\eps)}^2 \leq\delta_\eps,  \\
\int_{\Sigma}\left(a_\Phi ,h\right)_{g_\eps}dA_{g_\eps} + \langle L , h \rangle + \int_{\Sigma} b_\Phi V dA_{\tilde{g}_\eps} + \int_\Sigma \theta^2 V  < -\kappa 
\end{split}
\end{equation*}
We first notice that $\kappa \geq 0$ (choose $f=0$ in the first inequation). We also notice that $V$ has to be non-negative almost everywhere (choose $f = n \max\{ -V,0\}$ in the first equation and let $n\to +\infty$). By Step 2, let $\Phi \in \mathbf{O}_m(\Sigma)$ be such that $Q_{(h,V)}(\Phi) \geq - \eps \max( \Vert h \Vert_{g_\eps} , \Vert V \Vert_{\tilde{g}_\eps} ) $. The second equation implies that for all $\theta \in H^1(\tilde{g}_\eps)$ such that $ \Vert \theta \Vert_{H^1(\tilde{g}_\eps)}^2 \leq \delta_\eps$ and $L\in (S^2_0(\Sigma))^\star$ such that $ \Vert L \Vert_{g_\eps}^\star \leq \delta_\eps$
$$ Q_{(h,V)}(\Phi) + \int_\Sigma \theta^2 V dA_{\tilde{g}_\eps} + \langle L,h\rangle < -\kappa $$
and we obtain that 
$$ \int_\Sigma \theta^2 V dA_{\tilde{g}_\eps} + \langle L,h\rangle < -\kappa + \delta_\eps \max( \Vert h \Vert_{g_\eps} , \Vert V \Vert_{\tilde{g}_\eps} ) $$
Choosing $L = 0$ and knowing that the following supremum is realized, we obtain
$$  \max_{\theta \in H^1(\tilde{g}_\eps), \Vert \theta \Vert^2_{H^1(\tilde{g}_\eps) \leq \eps}} \int_{\Sigma} \theta^2 V dA_{\tilde{g}_\eps} < \delta_\eps \max( \Vert h \Vert_{g_\eps} , \Vert V \Vert_{\tilde{g}_\eps} ) $$
so that $  \max( \Vert h \Vert_{g_\eps} , \Vert V \Vert_{\tilde{g}_\eps} ) =\Vert h \Vert_{g_\eps}  $.
It is clear that $h\neq 0$ and chosing $\theta =0$ and taking $L$ and $-L$ and then the supremum, 
$$ \sup_{L \in (S^2_0(\Sigma))^\star, \Vert L \Vert_{g_\eps}^\star\leq 2\delta_\eps} \frac{\left\vert \langle L,h \rangle \right\vert}{\Vert h \Vert_{g_\eps}} \leq \delta_\eps $$
which gives a contradiction and we get Step 3.

\medskip

We conclude the proof of the proposition by taking an element of $K\cap F$ and renormalizing by $c_\eps$.
\end{proof}

\section{Laplace spectral functionals in the handle case} \label{sec:handle}

\subsection{Choice of the initial minimizing sequence}
This subsection is devoted to the construction of $\widetilde{\Sigma}_\eps, \tilde{g}_\eps \tilde{\beta}_\eps, g_\eps, \beta_\eps$ that satisfies \eqref{eq:tildeminseq}, \eqref{eq:disttilderegularized}, \eqref{eq:mainekeland} in order to apply later Proposition \ref{prop:eullag} for a well chosen $\delta_\eps \to 0$ (see \eqref{eq:defdeltaeps}). Let $\Sigma$ be a compact surface. We assume that a Riemannian metric $g$ realizes the absolute minimizer
$$ E(g,1) = \inf_{\tilde{g}\in Met_0(\Sigma)} E(\tilde{g},1) $$
We now take $p,q$ two distinct points on $\Sigma$ and $l>0$ and we denote
$$ \Sigma_\eps := \Sigma \setminus \left( \mathbb{D}_\eps(p) \cup \mathbb{D}_\eps(q) \right)  $$
and
$$ C_{l,\eps} := \eps\mathbb{S}^1 \times \left[-\frac{l\eps}{2},\frac{l\eps}{2}\right] $$
that we glue along their boundary:
$$ \tilde{\Sigma}_\eps = \left( \Sigma_\eps \cup C_{l,\eps} \right) / \sim $$
where $\sim$ is a glueing along $\partial  \mathbb{D}_\eps(p)$ and $\partial_1 C_{l,\eps} := \eps\mathbb{S}^1 \times \{-\frac{l\eps}{2}\}$ and along $\partial \mathbb{D}_\eps(q)$ and $\partial_2 C_{l,\eps}:= \eps\mathbb{S}^1 \times \{\frac{l\eps}{2}\}$ that preserves the orientation. We denote $\tilde{g}_\eps$ the $L^\infty$ metric on $\tilde{\Sigma}_\eps$ equal to $g$ on $\Sigma_\eps$ and to the flat metric on $C_{l,\eps}$. Up to a standard regularisation procedure by the heat kernel of $\tilde{g}_\eps$, we can assume that $\tilde{g}_\eps$ is continuous on $\tilde{\Sigma}_\eps$ without affecting the following estimates on eigenvalues.
We aim at computing an asymptotic expansion of $\lambda_i(\tilde{\Sigma}_\eps,\tilde{g}_\eps)$ and of $E(\tilde{g}_\eps,1)$. For that, we use Laplace eigenvalues $\mu_k^\eps$ with Neumann boundary conditions on $\Sigma_\eps := \Sigma \setminus \left( \mathbb{D}_\eps(p) \cup \mathbb{D}_\eps(q) \right)$ endowed with the metric $g$. First we give uniform estimates on eigenfunctions

\begin{cl} \label{cl:linftyesteigenfunctions}
For $\Lambda>0$, there is a constant $C:= C(\Sigma,\Lambda)$ such that any eigenfunction $\varphi_\eps$ associated to a Laplace eigenvalue on $(\tilde{\Sigma}_\eps,\tilde{g}_\eps)$ bounded by $\Lambda$ such that $\Vert \varphi_\eps \Vert_{L^2(\tilde{\Sigma}_\eps)}=1$ and any eigenfunction $\psi_\eps$ associated to a Laplace eigenvalue with Neumann boundary condition on $(\Sigma_\eps,g)$ bounded by $\Lambda$ such that $\Vert \psi_\eps \Vert_{L^2(\tilde{\Sigma}_\eps)}=1$, we have 
$$\Vert \varphi_\eps \Vert_{L^{\infty}\left(\tilde{\Sigma}_\eps\right)} \leq \sqrt{\ln \frac{1}{\eps}} \text{ and }\Vert \psi_\eps \Vert_{L^{\infty}\left(\tilde{\Sigma}_\eps\right)} \leq \sqrt{\ln \frac{1}{\eps}} $$
\end{cl}

\begin{proof}
Without loss of generality, we assume that $g$ is conformally flat on a disk $\mathbb{D}_{2}$ at the neighborhood of $0$. We set 
$$ f_\eps(r) = \frac{1}{2\pi r}\int_{\mathbb{S}_r} \varphi_\eps. $$
By a trace Sobolev inequality, we have that
$$ \vert f_\eps(1) \vert \leq C \Vert \varphi_\eps \Vert_{H^1(\Sigma,g)} \leq C \sqrt{1+\Lambda}$$
and on the annulus $\mathbb{D}\setminus \mathbb{D}_r$, we have that
$$ f_\eps(1)-f_\eps(r) = \int_{\mathbb{D}\setminus \mathbb{D}_r} \nabla \varphi_\eps \nabla \ln \vert x \vert dx $$
so that by conformal invariance of the Dirichlet energy,
$$ \vert f_\eps(r) \vert \leq \vert f_\eps(1) \vert + \Vert \nabla\varphi_\eps \Vert_{L^2(\Sigma,g)} \Vert \nabla \ln \vert x \vert \Vert_{L^2(\mathbb{D}\setminus\mathbb{D}_r)} \leq C\sqrt{1+\Lambda} + \sqrt{2\pi \Lambda \ln \frac{1}{r}}.$$

\medskip

Now, for $4 \eps \leq r \leq 1 $, we write the equation on $\tilde{\varphi}_\eps(x):= \varphi_\eps(rx)$ in $\mathbb{D}_2\setminus \mathbb{D}_{\frac{1}{2}}$ 
$$ \Delta \tilde{\varphi}_\eps = r^2 V_{\eps,r} \tilde{\varphi}_\eps $$
where $\Vert V_{\eps,r} \Vert_{L^\infty} \leq K$ where $K$ is independent of $\eps$ and $r$ 

$$ \Vert \tilde{\varphi}_\eps - f_\eps(r) \Vert_{L^\infty\left(\mathbb{D}_2\setminus \mathbb{D}_{\frac{1}{2}}\right)} \leq \Vert \tilde{\varphi}_\eps - f_\eps(r) \Vert_{L^2\left(\mathbb{D}_2\setminus \mathbb{D}_{\frac{1}{2}}\right)} + \Vert \Delta \tilde{\varphi}_\eps  \Vert_{L^2\left(\mathbb{D}_2\setminus \mathbb{D}_{\frac{1}{2}}\right)} $$
where by a Poincaré inequality and conformal invariance of the Dirichlet energy,
$$ \Vert \tilde{\varphi}_\eps - f_\eps(r) \Vert_{L^2\left(\mathbb{D}_2\setminus \mathbb{D}_{\frac{1}{2}}\right)} \leq C \Vert \nabla\left( \tilde{\varphi}_\eps - f_\eps(r)\right) \Vert_{L^2\left(\mathbb{D}_2\setminus \mathbb{D}_{\frac{1}{2}}\right)} \leq  C' \Vert \nabla \varphi_\eps \Vert_{L^2(\Sigma)} $$
and where
$$ \Vert \Delta\left( \tilde{\varphi}_\eps - f_\eps(r)\right) \Vert_{L^2\left(\mathbb{D}_2\setminus \mathbb{D}_{\frac{1}{2}}\right)} \leq r^2 (1+K) \Vert \tilde{\varphi}_\eps \Vert_{L^{\infty}\left(\mathbb{D}_2\setminus \mathbb{D}_{\frac{1}{2}}\right)} $$ 
so that gathering all the previous inequalities, we obtain
$$ \Vert \varphi_\eps \Vert_{L^\infty\left(\mathbb{D}_{2r}\setminus \mathbb{D}_{\frac{r}{2}}\right)} \leq C \sqrt{\ln \frac{1}{r}} + C' \sqrt{\Lambda} + r^2(1+K)\Vert \varphi_\eps \Vert_{L^\infty\left(\mathbb{D}_{2r}\setminus \mathbb{D}_{\frac{r}{2}}\right)} $$
so that letting $2 \eps \leq r \leq r_0$ with $r_0$ small enough,
$$\vert \varphi_\eps(r,\theta) \vert \leq \Vert \varphi_\eps \Vert_{L^\infty\left(\mathbb{D}_{4r}\setminus \mathbb{D}_{r}\right)} \leq C\sqrt{\ln \frac{1}{r}}. $$
Notice that exactly the same arguments holds for $\psi_\eps$ so that
$$\vert \psi_\eps(r,\theta) \vert \leq \Vert \varphi_\eps \Vert_{L^\infty\left(\mathbb{D}_{4r}\setminus \mathbb{D}_{r}\right)} \leq C\sqrt{\ln \frac{1}{r}}. $$

\medskip

In order to complete the proof, we rescale via a conformal diffeomorphism $\theta_\eps$ the equation of $\varphi_\eps$ in
$$\widetilde{\Sigma}_\eps \setminus \Sigma_{4\eps} = \left(\left(\mathbb{D}_{4\eps}(p) \cup \mathbb{D}_{4\eps}(q)\right) \cap \Sigma_\eps\right) \cup C_{l,\eps}$$
into an equation of a function $\tilde{\varphi}_\eps$ in a cylinder $C_\eps$ where $C_\eps = \theta_\eps^{-1}(\widetilde{\Sigma}_\eps \setminus \Sigma_{4\eps})$ has a uniformly bounded modulus:
$$ \Delta \tilde{\varphi}_\eps = V_\eps  \tilde{\varphi}_\eps \text{ in } C_\eps $$
where $V_\eps$ is uniformly bounded. Similarly, the equation of $\psi_\eps$ in 
$$\Sigma_\eps \setminus \Sigma_{4\eps} = \left(\mathbb{D}_{4\eps}(p) \cup \mathbb{D}_{4\eps}(q)\right) \cap \Sigma_\eps$$
into an equation of $\tilde{\psi}_\eps$ on two flat annuli $\mathbb{D}_4 \setminus \mathbb{D}_1$
$$\begin{cases} \Delta \tilde{\psi}_\eps = V_\eps \tilde{\psi}_\eps \text{ in } \mathbb{D}_4 \setminus \mathbb{D}_1 \\
\partial_r  \tilde{\psi}_\eps  = 0 \text{ on } \partial \mathbb{D}_1.
\end{cases}$$
In both cases, we can write again by elliptic estimates and a Poincaré inequality that 
$$ \Vert \tilde{\varphi}_\eps \Vert_{L^\infty(\Omega)} \leq \vert f_\eps(\eps) \vert + \Vert \nabla \tilde{\varphi}_\eps \Vert_{L^2(\Omega)} + \Vert \Delta \tilde{\varphi}_\eps \Vert_{L^2(\Omega)} $$
that gives with similar arguments as before 
$$ \Vert \varphi_\eps \Vert_{L^\infty(\tilde{\Sigma}_\eps \setminus \Sigma_{4\eps})} \leq C \sqrt{\ln \frac{1}{\eps}}.$$
Similarly, regularity for elliptic equations with Neumann boundary condition gives
$$ \Vert \tilde{\psi}_\eps \Vert_{L^\infty(\Omega)} \leq \vert f_\eps(\eps) \vert + \Vert \nabla \tilde{\psi}_\eps \Vert_{L^2(\Omega)} + \Vert \Delta \tilde{\psi}_\eps \Vert_{L^2(\Omega)} $$
that implies
$$ \Vert \psi_\eps \Vert_{L^\infty(\Sigma_\eps \setminus \Sigma_{4\eps})} \leq C \sqrt{\ln \frac{1}{\eps}}.$$
and gathering all the previous inequalities ends the proof of the Claim.
\end{proof}

\begin{prop}[Matthiesen-Siffert \cite{MS3}, Lemma 5.5 and Theorem 5.1] \label{propneumann}
Given $k\in \mathbb{N}^\star$, we denote $\mu_1^\eps,\cdots,\mu_k^\eps$ the $k$ first non-zero Neumann eigenvalues of $\Sigma_\eps := \Sigma \setminus \left( \mathbb{D}_\eps(p) \cup \mathbb{D}_\eps(q) \right)$. Then
$$ \forall 1 \leq i \leq k, \mu_i^\eps \geq \lambda_i(\Sigma,g)- O(\eps^2)$$
\end{prop}

\begin{proof}
\noindent\textbf{Step 1:} We let $\hat{\psi}_\eps$ be the extension of $\psi_\eps$ on $\Sigma$ obtained by a harmonic extension on $\mathbb{D}_\eps(p)\cup \mathbb{D}_\eps(q)$. We prove that there is a constant $C$ such that 
$$ \Vert \hat{\psi}_\eps \Vert_{W^{1,2}(\mathbb{D}_\eps(p)\cup \mathbb{D}_\eps(q))} \leq C \Vert \psi_\eps \Vert_{W^{1,2}(\Sigma_\eps)} $$

\medskip

\noindent\textbf{Proof of step 1:} It suffices to prove this inequality in a chart at the neighborhood of $p$. Letting $f_\eps(x):= \hat{\psi}_\eps(p+\eps x)$, we have by trace embedding theorems and elliptic theory for harmonic functions that
$$ \Vert f_\eps \Vert_{H^1(\mathbb{D})} \leq C \Vert f_\eps \Vert_{H^1(\mathbb{D}_2 \setminus \mathbb{D})} $$
$$ \Vert \nabla f_\eps \Vert_{L^2(\mathbb{D})} \leq C \Vert \nabla f_\eps \Vert_{L^2(\mathbb{D}_2 \setminus \mathbb{D})} $$
that rescale to
$$ \Vert \psi_\eps \Vert_{L^2(\mathbb{D}_\eps(p))}^2 \leq C^2 \left( \Vert \psi_\eps \Vert_{L^2(\mathbb{D}_{2\eps}(p) \setminus \mathbb{D}_\eps(p))}^2 + \eps^2 \Vert \nabla \psi_\eps \Vert_{L^2(\mathbb{D}_{2\eps}(p) \setminus \mathbb{D}_\eps(p))}^2 \right) $$
$$ \Vert \nabla \psi_\eps \Vert_{L^2(\mathbb{D}_\eps(p))} \leq C \Vert \nabla \psi_\eps \Vert_{L^2(\mathbb{D}_{2\eps}(p) \setminus \mathbb{D}_\eps(p))}   $$
completing the proof of Step 1.

\medskip

\noindent\textbf{Step 2:} We prove that there is a constant $C:= C(\Sigma,\Lambda)$ such that for any Laplace eigenfunction with Neumann boundary conditions $\psi_\eps$ in $(\Sigma_\eps,g)$ such that $\Vert \psi_\eps \Vert_{L^2\left(\Sigma_\eps,g\right)} = 1$ with eigenvalue $\mu_\eps$ bounded by $\Lambda$ satisfies
$$ \int_{\mathbb{S}_\eps(p)\cup \mathbb{S}_\eps(q)} \vert \partial_\tau \psi_\eps \vert^2 \leq C \eps $$

\noindent

\noindent\textbf{Proof of step 2:} By Step 1, there is a constant $C>0$ such that
$$ \Vert \hat{\psi}_\eps \Vert_{W^{1,2}(\Sigma)} \leq C \Vert \psi_\eps \Vert_{W^{1,2}(\Sigma_\eps)} $$
We let $\alpha_\eps$ be the solution of the following equation
$$ \begin{cases} \Delta \alpha_\eps = \mu_\eps \hat{\psi}_\eps \text{ in } \mathbb{D} \\
\alpha_\eps = 0 \text{ on } \mathbb{S}^1
\end{cases} $$
so that $h_\eps := \psi_\eps - \alpha_\eps$ is a harmonic function in $\mathbb{D}\setminus \mathbb{D}_\eps$. Since $\Delta \alpha_\eps$ is bounded in $W^{1,2}$, we obtain by standard elliptic regularity that $\alpha_\eps$ is bounded in $C^1$ by $\Vert \psi_\eps \Vert_{W^{1,2}(\Sigma_\eps)}$  so that
$$ \int_{\mathbb{S}_\eps(p)\cup \mathbb{S}_\eps(q)} \vert \partial_\tau \alpha_\eps \vert^2 \leq C \Vert \psi_\eps \Vert_{W^{1,2}(\Sigma_\eps)}  \eps. $$
Now, let's focus on $h_\eps$. We have that
$$ \Vert h_\eps \Vert_{W^{1,2}(\mathbb{D}\setminus\mathbb{D}_\eps)} \leq C \Vert \psi_\eps \Vert_{W^{1,2}(\Sigma_\eps)}  $$
$$ \vert \partial_\nu h_\eps \vert = \vert \partial_\nu a_\eps \vert \leq C\Vert \psi_\eps \Vert_{W^{1,2}(\Sigma_\eps)}  \text{ on } \mathbb{S}_\eps(p)\cup \mathbb{S}_\eps(q) $$
Now let's write the Fourier expansion of $h_\eps$
$$ h_\eps = a + b \ln r + \sum_{n \in \mathbb{Z}^\star}\left( c_n r^n + \overline{c_n} r^{-n}\right) e^{in\theta} $$
where we drop the index $\eps$ of all the coefficients $a,b \in \R, c_n \in \mathbb{C}$. We have that for $\eps \leq \frac{1}{2}$
\begin{equation*}\begin{split} \int_{\mathbb{D}\setminus \mathbb{D}_\eps} h_\eps^2  \geq & \int_{\mathbb{D}\setminus \mathbb{D}_\eps} \left\vert\sum_{n \in \mathbb{Z}^\star}\left( c_n r^n + \overline{c_n} r^{-n}\right) e^{in\theta}\right\vert^2 \\
\geq & \sum_{n\in \mathbb{Z}^\star}\left( \vert c_n \vert^2 \int_\eps^1 r^{2n+1}dr + 2 Re(c_n c_{-n}) \int_\eps^1 r dr + \vert c_{-n} \vert^2 \int_\eps^1 r^{-2n+1}dr \right) \\
\geq &  \sum_{ n  \geq 2}\left( \frac{\vert c_n \vert^2}{2n+2} \left(1-\eps^{2n+2}\right) - \vert c_n\vert \vert c_{-n}\vert + \frac{\vert c_{-n}\vert^2}{2n-2}  (\eps^{2-2n}-1)\right) \\ 
&+  \left( \frac{\vert c_1 \vert^2}{4} \left(1-\eps^{4}\right) - \vert c_1\vert \vert c_{-1}\vert + \vert c_{-1} \vert^2  \ln\frac{1}{\eps} \right) \\
\geq &  \sum_{ n  \geq 2}\left( \frac{\vert c_n \vert^2}{2n+2} \left(1-\eps^{2n+2} - \frac{1}{2}\right) + \frac{\vert c_{-n}\vert^2}{2n-2}  (\eps^{2-2n}-1-2(n^2-1))\right) \\ 
&+  \left( \frac{\vert c_1 \vert^2}{4} \left(1-\eps^{4}-\frac{1}{2}\right) + \vert c_{-1} \vert^2 \left( \ln\frac{1}{\eps} - 8 \right) \right) \\
\geq & \frac{1}{4} \left( \sum_{n\geq 2} \left( \frac{\vert c_n \vert^2}{2n+2} + \frac{\vert c_{-n} \vert^2}{2n-2} \eps^{2-2n} \right) +  \frac{\vert c_1 \vert^2}{4} +    \vert c_{-1} \vert^2  \ln\frac{1}{\eps}   \right) \end{split} \end{equation*}
so that it is clear that the harmonic function $ \sum_{n\geq 1} c_n e^{i n\theta} r^n $
is bounded by by $\Vert \psi_\eps \Vert_{W^{1,2}(\Sigma_\eps)}$ in $L^2(\mathbb{D})$. Similarly, we also have that $ \sum_{n\geq 1} \overline{c_n} e^{-i n\theta} r^n $ is bounded by by $\Vert \psi_\eps \Vert_{W^{1,2}(\Sigma_\eps)}$ in $L^2(\mathbb{D})$. Therefore the function 
$$ \beta_\eps := \sum_{n\geq 1} \left(c_n e^{i n\theta} + \overline{c_{-n}} e^{-i n\theta} \right) r^n $$
is bounded in $\mathcal{C}^1$ by $\Vert \psi_\eps \Vert_{W^{1,2}(\Sigma_\eps)}$ and 
$$ \int_{\mathbb{S}_\eps(p)\cup \mathbb{S}_\eps(q)} \vert \partial_\tau \beta_\eps \vert^2 \leq C \Vert \psi_\eps \Vert_{W^{1,2}(\Sigma_\eps)}  \eps. $$
Now, we let $\gamma_\eps := \psi_\eps - \alpha_\eps - \beta_\eps $ and we have that
$$ \gamma_\eps = a +b\ln r + \sum_{n\geq 1} \left(c_{-n} e^{-i n\theta} + \overline{c_n} e^{i n\theta} \right) r^{-n} $$
By a classical Pohozaev identity on harmonic functions,
$$ \int_{\mathbb{S}^1} \left( r^2\left\vert \gamma_r(r,\theta) \right\vert^2 - \left\vert\gamma_\theta(r,\theta)\right\vert^2 \right) d\theta  $$
does not depend on $r$. We have that 
$$ \int_{\mathbb{S}^1} \left\vert\gamma_\theta(r,\theta)\right\vert^2 d\theta = O\left(\frac{1}{r^2}\right) \text{ as } r\to +\infty $$
and that 
$$ \int_{\mathbb{S}^1} r^2 \left\vert \gamma_r(r,\theta) \right\vert^2 d\theta = \int_{\mathbb{S}^1} r^2 \left\vert \left( \partial_r \left(\gamma - b\ln r \right) \right)\right\vert^2 d\theta + 2\pi b^2 = 2\pi b^2 + O\left(\frac{1}{r^2}\right)  $$
as $r\to +\infty$ so that
$$ \int_{\mathbb{S}_\eps(p)}\vert\partial_\tau \gamma_\eps \vert^2 = \int_{\mathbb{S}_\eps(p)}\vert\partial_\nu \gamma_\eps \vert^2  -  \frac{2\pi b^2}{\eps} \leq \int_{\mathbb{S}_\eps(p)}\vert\partial_\nu \gamma_\eps \vert^2 \leq C \Vert \psi_\eps \Vert_{W^{1,2}(\Sigma_\eps)}  \eps $$
and the same property holds on $\mathbb{S}_\eps(q)$. Gathering all the previous inequalities completes the proof of Step 1.

\medskip

\noindent\textbf{Step 3:} Let $\hat{\psi}_\eps$ be the continuous extension of $\psi_\eps$ in $\Sigma$ that is harmonic in $\mathbb{D}_\eps(p)\cup \mathbb{D}_\eps(q)$. We have that
$$ \int_{\mathbb{D}_\eps(p)\cup \mathbb{D}_\eps(q)} \vert \nabla \hat{\psi}_\eps \vert^2  \leq C \Vert \psi_\eps \Vert_{W^{1,2}(\Sigma_\eps)}^2 \eps^2 $$
as $\eps \to 0$.

\medskip

\noindent\textbf{Proof of Step 3:}
In polar coordinates in a conformally flat chart centered at $p$, we use
$$ f_\eps := \frac{r}{\eps}\left( \psi_\eps(\eps,\theta) - \int_{\mathbb{S}^1} \psi_\eps(\eps,\alpha)d\alpha  \right) $$
as a competitor for the energy of the harmonic extension of $\psi_\eps(\eps,\theta) - \int_{\mathbb{S}^1} \psi_\eps(\eps,\alpha)d\alpha $ which has the same energy as $\hat{\psi}_\eps$. Then
\begin{equation*}
\begin{split}\int_{\mathbb{D}_\eps(p)} \vert\nabla \hat{\psi}_\eps\vert^2  \leq & \int_{\mathbb{D}_\eps(p)} \vert\nabla f_\eps\vert^2 \\ 
= &\frac{1}{\eps^2} \left( \int_0^\eps  \left( \int_{\mathbb{S}^1}\partial_\theta \psi_\eps(\eps,\theta)^2 d\theta + \int_{\mathbb{S}^1} \left( \psi_\eps(\eps,\theta) - \int_{\mathbb{S}^1} \psi_\eps(\eps,\alpha)d\alpha \right)^2d\theta\right)r dr \right) \\
\leq & C \int_{\mathbb{S}^1} \partial_\theta \psi_\eps(\eps,\theta)^2d\theta = C \eps \int_{\mathbb{S}_\eps(p)} \left( \partial_\tau \psi_\eps \right)^2 \leq C \Vert \psi_\eps \Vert_{W^{1,2}(\Sigma_\eps)}^2 \eps^2.
\end{split} \end{equation*}
where we used the Poincaré inequality on the circle to obtain the first term in the third line and Step 1 to obtain the last inequality. Of course, such a computation also holds in the neighborhood of $q$.

\medskip

\noindent\textbf{Step 4:} We test the eigenfunctions with Neumann boundary conditions extended harmonically in $\mathbb{D}_\eps(p)\cup \mathbb{D}_\eps(q)$ denoted by $\hat{\psi}_0^\eps,\cdots,\hat{\psi}_k^\eps$ in the variational characterization of $\lambda_k(\Sigma,g)$ and complete the proof of the proposition.

\medskip

\noindent\textbf{Proof of Step 4:} We have the existence of $a_\eps\in \mathbb{S}^k$ such that
\begin{equation*}
\begin{split}\lambda_k(\Sigma,g) & \leq \frac{\int_\Sigma\vert \nabla \left(\sum_{i=0}^k a_i^\eps \hat{\psi}_i^\eps\right) \vert^2 }{\int_\Sigma \left(\sum_{i=0}^k a_i^\eps \hat{\psi}_i^\eps\right)^2} \leq \frac{\sum_{i=0}^k \left(a_i^\eps\right)^2 \mu_i^\eps + \int_{\mathbb{D}_\eps(p)\cup \mathbb{D}_\eps(q)}\vert \nabla \left(\sum_{i=0}^k a_i^\eps \hat{\psi}_i^\eps\right) \vert^2 }{ \sum_{i=0}^k \left(a_i^\eps\right)^2 } \\
& \leq \mu_k^\eps + \sum_{i=0}^k \int_{\mathbb{D}_\eps(p)\cup \mathbb{D}_\eps(q)}\vert \nabla \hat{\psi}_i^\eps\vert^2 \leq \mu_k^\eps + O\left(\eps^2\right)
\end{split}
\end{equation*}
where we used Cauchy-Schwarz inequalities and Step 2.
\end{proof}

\begin{cor} \label{cor:boundsenergy}
We have that
\begin{equation} \label{eqlambdaiepsepss} \lambda_i(\tilde{\Sigma}_\eps,\tilde{g}_\eps) \geq \lambda_i(\Sigma,g) - O\left(\eps^2\ln \frac{1}{\eps}\right) \end{equation}
and
\begin{equation} \label{eqEepsepss} E(\tilde{\Sigma}_\eps,\tilde{g}_\eps) \leq E(\Sigma,g) +O\left(\eps^2\ln\frac{1}{ \eps}\right) \end{equation}
\end{cor}

\begin{proof}
Let $\varphi_0^\eps,\varphi_1^\eps,\cdots,\varphi_i^\eps$ be an orthonormal family of eigenfunctions associated to the eigenvalues $\lambda_0(\tilde{\Sigma}_\eps,\tilde{g}_\eps),\cdots, \lambda_i(\tilde{\Sigma}_\eps,\tilde{g}_\eps)$, we have the existence of $a_\eps \in \mathbb{S}^{i}$ such that the function $\psi_\eps := \sum_{k=0}^i a_i^\eps \varphi_i^\eps$ satisfies
\begin{equation*}
\begin{split} \mu_i^\eps \leq \frac{\int_{\Sigma\setminus\left(\mathbb{D}_{\eps}(p)\cup \mathbb{D}_\eps(q)\right)}\left\vert \nabla \psi_\eps \right\vert^2_g dA_g}{\int_{\Sigma\setminus\left(\mathbb{D}_{\eps}(p)\cup \mathbb{D}_\eps(q)\right)}\left( \psi_\eps \right)^2 dA_g} \leq \frac{\int_{\tilde{\Sigma}_\eps} \vert \nabla \psi_\eps \vert_{\tilde{g}_\eps}^2 dA_{\tilde{g}_\eps} }{ \int_{\tilde{\Sigma}_\eps} \left(\psi_\eps\right)^2 dA_{\tilde{g}_\eps} - \int_{C_{l,\eps}} \left(\psi_\eps\right)^2 dA_{\tilde{g}_\eps} } \\
 \leq \lambda_i(\tilde{\Sigma}_\eps,\tilde{g}_\eps) + O\left(l\eps^2\ln \frac{1}{\eps}\right) 
\end{split}
\end{equation*}
as $\eps \to 0$ where we used Claim \ref{cl:linftyesteigenfunctions} and we use proposition \ref{propneumann} to conclude for equality \eqref{eqlambdaiepsepss}. \eqref{eqEepsepss} easily follows since $F$ is a $\mathcal{C}^1$ function, eigenvalues are uniformly bounded and $A(\Sigma_\eps, \tilde{g}_\eps) = A(\Sigma,g) - O(\eps^2)$.
\end{proof}
As a conclusion, we define 
\begin{equation}\label{eq:defdeltaeps}\delta_\eps := c \eps \sqrt{\ln\left(\frac{1}{\eps}\right)}\end{equation} 
for some well chosen constant $c>0$ and construct the previous Palais-Smale approximation \eqref{eq:tildeminseq}, \eqref{eq:disttilderegularized}, \eqref{eq:mainekeland} to the sequence $\tilde{g}_\eps$ and $\tilde{\beta}_{\eps} = \frac{1}{A_{\tilde{g}_\eps}(\tilde{\Sigma}_\eps)}$ on $\tilde{\Sigma}_{\eps}$ pullbacked on a fixed surface $\Sigma$ by a bi-Lipschitz diffeomorphism.

\subsection{Some convergence of $\omega_{\eps}$ to $1$ and first replacement of $\Phi_{\eps}$}
We set 
$$\omega_{\eps} = \sqrt{\left\vert \Phi_{\eps} \right\vert_{\Lambda_{\eps}}^2 + \theta_\eps^2 } =  \left\vert \Phi_{\eps} \right\vert_{\Lambda_{\eps}} 
$$ 
We first prove that $\int_{\tilde{\Sigma}_\eps}  \vert\nabla\omega_{\eps}\vert_{\tilde{g}_\eps}^2 dA_{\tilde{g}_\eps}$ converges to $0$ and that $\Phi_{\eps}$ has a similar $H^1(\tilde{g}_\eps)$ behaviour as $\frac{\Phi_{\eps}}{\omega_{\eps}}$

\begin{cl} \label{cl:cvreplacement} We have that
\begin{equation}  \label{eqomegaepsto1}   \int_{\tilde{\Sigma}_\eps} \left\vert \nabla \omega_\eps \right\vert^2_{\tilde{g}_\eps} dA_{\tilde{g}_\eps} + \int_{\tilde{\Sigma}_\eps} \left\vert \nabla\left( \Phi_\eps - \frac{\Phi_\eps}{\omega_\eps} \right) \right\vert_{\tilde{g}_\eps,\Lambda_\eps}^2 dA_{\tilde{g}_\eps}  \leq O(\delta_\eps) \end{equation}
as $\eps\to 0$.
\end{cl}

\begin{proof} We first prove
\begin{equation} \label{eqomega_epsminus1} L_\eps\left( \left\vert \Lambda_\eps \Phi_\eps \right\vert^2 \left(1- \frac{1 }{\omega_\eps}\right) \right) \leq O(\delta_\eps) \end{equation}
as $\eps\to 0$. Since $\omega_\eps \geq 1$,  
and $\left\vert \Phi_\eps \right\vert_{\Lambda_\eps}^2 \leq \omega_\eps^2$, we have that
\begin{equation*}
\begin{split}
 L_\eps\left( \left\vert \Lambda_\eps \Phi_\eps \right\vert^2 \left(1- \frac{1 }{\omega_\eps} \right) \right) \leq & \left(\max{\lambda_i^\eps}\right) L_\eps\left(  \left(\omega_\eps^2-\omega_\eps\right) \right) \\
\leq & \max\lambda_i^\eps  \left( L_\eps\left(  \left\vert \Phi_\eps \right\vert_{\Lambda_\eps}^2 \right) + L_\eps \left(  \theta_\eps^2 \right) - L_{\eps}(1) \right) \\
 \end{split} \end{equation*}
 so that
$$ L_\eps\left( \left\vert \Lambda_\eps \Phi_\eps \right\vert^2 \left(1- \frac{1 }{\omega_\eps}\right) \right) \leq  \left(\max\lambda_i^\eps \right)  L_\eps( \theta_\eps^2) \leq  \left(\max\lambda_i^\eps \right)  \left\Vert \beta_\eps \right\Vert_{\tilde{g}_\eps} \left\Vert \theta_\eps \right\Vert^2_{H^1(\tilde{g}_\eps)} \leq  \left\Vert \beta_\eps \right\Vert_{\tilde{g}_\eps} O(\delta_\eps) $$
as $\eps \to 0$ by proposition \ref{prop:eullag} and since we know that
$$  \left\Vert \beta_\eps \right\Vert_{\tilde{g}_\eps} \leq  \left\Vert 1 \right\Vert_{\tilde{g}_\eps} +  \left\Vert 1- \beta_\eps \right\Vert_{\tilde{g}_\eps} \leq 1 + O(\delta_\eps) $$
by \eqref{eq:disttilderegularized} and we obtain \eqref{eqomega_epsminus1}.

We know prove \eqref{eqomegaepsto1}. Along the following computations, in all the integrations with respect to $\tilde{\Sigma}_\eps$, the gradient and the area measure are taken with respect to $\tilde{g}_\eps$:
\begin{equation*} 
\begin{split}
\int_{\tilde{\Sigma}_\eps} & \left\vert \nabla \frac{\Phi_\eps}{\omega_\eps} \right\vert_{\Lambda_\eps}^2 - \int_{\tilde{\Sigma}_\eps}  \left\vert \nabla \Phi_\eps \right\vert_{\Lambda_\eps}^2 - \int_{\tilde{\Sigma}_\eps}  \left\vert \nabla\left( \Phi_\eps - \frac{\Phi_\eps}{\omega_\eps} \right) \right\vert_{\Lambda_\eps}^2 
\\ = & - 2 \int_{\tilde{\Sigma}_\eps} \left\langle \nabla \Phi_\eps, \nabla\left( \Phi_\eps - \frac{ \Phi_\eps}{\omega_\eps}\right) \right\rangle_{\Lambda_\eps} 
= - 2 \int_{\tilde{\Sigma}_\eps} \Delta \Phi_\eps  {\Lambda_\eps}.\left( \Phi_\eps - \frac{ \Phi_\eps}{\omega_\eps}\right)  \\
= & - 2 \beta_\eps\left( \Lambda_\eps.\Phi_\eps , {\Lambda_\eps}.\left( \Phi_\eps - \frac{ \Phi_\eps}{\omega_\eps}\right) \right)
= - 2 L_\eps\left(   \left\vert \Lambda_\eps \Phi_\eps \right\vert^2 \left(1- \frac{1 }{\omega_\eps} \right) \right) = O(\delta_\eps)
\end{split}
\end{equation*}
where we tested $ \Delta \Phi_{\eps} =  \beta_\eps( \Lambda_\eps\Phi_\eps,.) $ in $\tilde{\Sigma}_\eps$ against $ \Lambda_\eps. \left( \Phi_\eps-\frac{ \Phi_\eps}{\omega_\eps}\right)$, and we used \eqref{eqomega_epsminus1}.

In particular, we have
$$ 0 \leq \int_{\tilde{\Sigma}_\eps} \left\vert \nabla\left( \Phi_\eps - \frac{\Phi_\eps}{\omega_\eps} \right) \right\vert_{\Lambda_\eps}^2 \leq \int_{\tilde{\Sigma}_\eps}  \left( \left\vert \nabla \frac{\Phi_\eps}{\omega_\eps} \right\vert_{\Lambda_\eps}^2 - \left\vert \nabla \Phi_\eps \right\vert_{\Lambda_\eps}^2\right) + O(\delta_\eps) $$
as $\eps\to 0$ and knowing that with the straightforward computations we have
\begin{equation*}
\begin{split} \left\vert \nabla \frac{\Phi_\eps}{\omega_\eps} \right\vert_{\Lambda_\eps}^2 - \left\vert \nabla \Phi_\eps \right\vert_{\Lambda_\eps}^2 = & \left( 1 - \omega_\eps^2  \right)\left\vert \nabla \frac{\Phi_\eps}{\omega_\eps} \right\vert_{\Lambda_\eps}^2 - \left\vert \nabla \omega_\eps \right\vert^2 \frac{\omega_\eps^2+\theta_\eps^2}{\omega_\eps^2}  + 2 \frac{\theta_\eps}{\omega_\eps} \nabla \omega_\eps \nabla \theta_\eps \\
=  & \left( 1 - \omega_\eps^2  \right)\left\vert \nabla \frac{\Phi_\eps}{\omega_\eps} \right\vert_{\Lambda_\eps}^2 - \left\vert \nabla \omega_\eps \right\vert^2 - \left\vert \frac{\theta_\eps}{\omega_\eps} \nabla \omega_\eps - \nabla \theta_\eps \right\vert^2 + \left\vert \nabla \theta_\eps \right\vert^2
\end{split} \end{equation*}
where 
$$ \left\vert \frac{\theta_\eps}{\omega_\eps} \nabla \omega_\eps - \nabla \theta_\eps \right\vert^2 = \omega_\eps^2 \left\vert \nabla \frac{\theta_\eps}{\omega_\eps} \right\vert^2  $$
we obtain that 
\begin{equation*}
\begin{split} \int_{\tilde{\Sigma}_\eps} \left( \omega_\eps^2 - 1 \right)\left\vert \nabla \frac{\Phi_\eps}{\omega_\eps} \right\vert_{\Lambda_\eps}^2 + \int_{\tilde{\Sigma}_\eps} \left\vert \nabla \omega_\eps \right\vert^2 +  \int_{\tilde{\Sigma}_\eps} \omega_\eps^2 \left\vert \nabla \frac{\theta_\eps}{\omega_\eps} \right\vert^2 + \int_{\tilde{\Sigma}_\eps} \left\vert \nabla\left( \Phi_\eps - \frac{\Phi_\eps}{\omega_\eps} \right) \right\vert_{\Lambda_\eps}^2 \\
 \leq \int_{\tilde{\Sigma}_\eps} \left\vert \nabla \theta_\eps \right\vert^2 + O\left(\delta_\eps \right) 
\end{split} \end{equation*}
as $\eps \to 0$ and we conclude by Proposition \ref{prop:eullag} again.
\end{proof}

\subsection{Quantitative convergence of eigenvalues and quantitative energy bounds}

We recall that $\lambda_k^\eps := \lambda_k(\tilde{\Sigma}_\eps,g_\eps,\beta_\eps)$ and that $\lambda_k := \lambda_k(\Sigma,g)$

\begin{cl} \label{cl:upperboundlambdakeps}
For all $k \in \mathbb{N}^\star$ 
$$ \lambda_k^\eps \leq \lambda_k + O\left( \frac{1}{\ln\frac{1}{\eps}} \right) $$
as $\eps \to 0$.
\end{cl}

\begin{proof}

We let $\eta_\eps \in \mathcal{C}^\infty\left(\Sigma_\eps \right)$ be a function such that $\eta_\eps=1$ in $\Sigma_{\sqrt{\eps}}$, $0\leq \eta_\eps\leq 1$ and $\int_{\Sigma} \vert \nabla \eta_\eps \vert^2 \leq \frac{C}{\ln \frac{1}{\eps}}$ and we test $\langle \varphi_0 \eta_\eps,\cdots, \varphi_{k} \eta_\eps \rangle$ in the variational characterization of $\lambda_k(\tilde{\Sigma}_\eps, g_\eps, \beta_\eps)$. We have that

\begin{equation*} 
\begin{split}\int_\Sigma & \nabla\left( \eta_\eps \varphi_i^\eps\right)\nabla\left( \eta_\eps \varphi_j^\eps\right)  = \int_\Sigma (\eta_\eps^2 -1)\nabla \varphi_i \nabla \varphi_k + 2 \int_\Sigma \eta_\eps \nabla \eta_\eps \left(\varphi_i \nabla \varphi_j + \varphi_j \nabla \varphi_i\right)  \\
& + \int_\Sigma \vert \nabla \eta_\eps \vert^2 \varphi_i \varphi_j   = O(\eps) + O\left(\frac{\sqrt{\eps}}{\sqrt{\ln\frac{1}{\eps}}}\right) + O\left(\frac{1}{\ln \frac{1}{\eps}}\right)
\end{split} \end{equation*}
and for any $1\leq i,j \leq k$,
$$ \vert \beta_\eps( \varphi_{i} \eta_\eps , \varphi_{j} \eta_\eps ) - \int_{\tilde{\Sigma}_\eps} \eta_\eps^2 \varphi_i^\eps \varphi_j^\eps dA_{\tilde{g}_\eps} \vert \leq \Vert \beta_\eps - 1 \Vert_{\tilde{\Sigma}_{\eps},\tilde{g}_\eps} \Vert \varphi_i \eta_\eps \Vert_{H^1(\tilde{\Sigma}_{\eps},\tilde{g}_\eps)} \Vert \varphi_j \eta_\eps \Vert_{H^1(\tilde{\Sigma}_{\eps},\tilde{g}_\eps)} = O\left( \delta_\eps \right) $$
by \eqref{eq:disttilderegularized} so that
$$   \beta_\eps( \varphi_{i} \eta_\eps , \varphi_{j} \eta_\eps ) = \int_\Sigma \varphi_i\varphi_j + O(\delta_\eps). $$
Then
$$ \lambda_k^\eps \leq \sup_{a \in \mathbb{S}^k} \frac{ \int_{\tilde{\Sigma}_\eps} \vert \nabla \left( \sum_i a_i \eta_\eps \varphi_\eps^i \right) \vert^2_{g_\eps} dA_{g_\eps} }{ \beta_\eps\left( \left( \sum_i a_i \eta_\eps \varphi_\eps^i \right),\left( \sum_i a_\eps^i \eta_\eps \varphi_\eps^i \right)\right) } =  \frac{ \sum_{i,j} a_i^\eps a_j^\eps \int_{\tilde{\Sigma}_\eps} \langle \nabla \eta_\eps \varphi_i^\eps,\nabla \eta_\eps \varphi_j^\eps \rangle_{g_\eps} }{\sum_{i,j} a_i^\eps a_j^\eps \beta_\eps(\eta_\eps\varphi_i^\eps, \eta_\eps\varphi_j^\eps) } $$
where $a_\eps$ realizes the supremum. Then, by \eqref{eq:defdeltaeps}
\begin{equation*}
\begin{split}
 \frac{ \sum_{i,j} a_i^\eps a_j^\eps \int_{\tilde{\Sigma}_\eps} \langle \nabla \eta_\eps \varphi_i^\eps,\nabla \eta_\eps\varphi_j^\eps \rangle_{g_\eps} }{\sum_{i,j} a_i^\eps a_j^\eps \beta_\eps(\varphi_i\eta_\eps,\varphi_j\eta_\eps) } & =  \frac{\sum_i (a_i^\eps)^2 \int_{\Sigma} \vert \nabla \varphi_i \vert^2_{g} dA_g + O\left(\sqrt{\frac{\eps}{\ln\frac{1}{\eps}}}\right) + O\left(\frac{1}{\ln \frac{1}{\eps}}\right) }{\sum_i (a_i^\eps)^2 \int_\Sigma \varphi_i^2 - O(\eps\ln \frac{1}{\eps}) }  \\
 & \leq \frac{\lambda_k+ O\left(\frac{1}{\ln \frac{1}{\eps}}\right)}{1-O(\delta_\eps)} 
\end{split}
\end{equation*}
and we obtain the expected result.
\end{proof}

Up to the extraction of a subsequence, we assume that for all $i$, 
$$\lambda_i^\eps \to \nu_i$$
as $\eps \to 0$. Since $\sharp \left( \{\lambda_i^\eps\}_{i \leq n_\eps} \right) \leq m$, $\{\nu_i\}$ is also finite even if $n_\eps \to +\infty$ as $\eps\to 0$. We denote $\left( \mu_j \right)$ an increasing sequence such that $\{ \lambda_k \}_{1\leq k \leq m} = \{ \mu_j \}_{1\leq j\leq J}$
$$ A_j := \{ i \in \mathbb{N} ; \nu_i = \mu_j \} $$
$$ A_j^\eps := A_j \cap \{1,\cdots,n_\eps\} $$
Since the increasing sequence $\{\lambda_i^\eps\}_{1\leq i \leq n_\eps}$ contains at most $m$ numbers, if $n_\eps \to +\infty$, $A_J$ is infinite and $A_j^\eps = A_j$ are finite for $j<J$.

We also denote for $\mu \geq 0 $. $E_{=\mu}(g)$ the set of eigenfunctions of $(\Sigma,g)$ associated to the eigenvalue $\mu$. Of course, if $\mu$ is not an eigenvalue, $E_{=\mu}(g) = \{0\}$. We also denote
$$ E_{\leq \mu} := \sum_{0 \leq \nu \leq \mu} E_{=\nu}(g) \text{ and } E_{< \mu} := \sum_{0\leq \nu < \mu} E_{=\nu}(g).$$
We also denote $\pi_{=\mu}$, $\pi_{\leq \mu}$, $\pi_{<\mu}$ the orthogonal projections with respect to $L^2(g)$ of $E_{=\mu}(g)$, $E_{\leq\mu}(g)$, $E_{<\mu}(g)$.

We let $\eta_\eps \in \mathcal{C}^\infty\left(\Sigma_\eps \right)$ be a function such that $\eta_\eps =1$ in $\Sigma_{\sqrt{\eps}}$, $0\leq \eta_\eps \leq 1$ and 
\begin{equation} \label{eq:defetaepscap} \int_{\Sigma} \vert \nabla \eta_\eps \vert^2 \leq \frac{C}{\ln \frac{1}{\eps}} \end{equation}

\begin{lem} \label{lem:mainproj}
Let $1 \leq j_0 \leq J$ and $\varphi_1,\cdots, \varphi_{a_{j_0}}$ be an orthonormal family of eigenfunctions associated to all the eigenvalues that belong to $\Lambda_{j_0} = \{ \lambda_i ; \lambda_i \neq \mu_j, 0\leq i \leq m  \} $. Then for all $i \in A_{j_0}^\eps$,
$$  \sum_{j=1}^{a_{j_0}} \left( \int_{\Sigma} \eta_\eps \frac{\phi_\eps^i}{\omega_\eps} \varphi_j dA_g \right)^2 \leq  \frac{C}{\ln\frac{1}{\eps}} \beta_\eps(\phi_i^\eps,\phi_i^\eps) $$
\end{lem}

\begin{proof}
Let $1 \leq j\leq a_{j_0}$, $\lambda$ the eigenvalue of $\varphi_j$ and $i \in A_{j_0}^\eps$, we have that
\begin{equation*}
\begin{split} \lambda \int_\Sigma \eta_\eps \frac{\phi^i_\eps}{\omega_\eps} \varphi_j & = \int_\Sigma \eta_\eps \frac{\phi^i_\eps}{\omega_\eps} \Delta_g \varphi_j = \int_\Sigma \nabla\left( \eta_\eps \frac{\phi^i_\eps}{\omega_\eps} \right) \nabla \varphi_j dA_{g} \\
 &=  \int_{\tilde{\Sigma}_\eps}  \frac{\phi_i^\eps}{\omega_\eps} \nabla \eta_\eps  \nabla \hat{\varphi}_j dA_{\tilde{g}_\eps} + \int_{\tilde{\Sigma}_\eps}  \eta_\eps \nabla \left( \frac{\phi_i^\eps}{\omega_\eps} \right) \nabla \hat{\varphi}_j dA_{\tilde{g}_\eps}  \\
\end{split}
\end{equation*}
that
\begin{equation*}
\begin{split}
\int_{\tilde{\Sigma}_\eps} \nabla \left( \frac{\phi_i^\eps}{\omega_\eps} \right) \nabla \hat{\varphi}_j dA_{g_\eps} = \int_{\tilde{\Sigma}_\eps} \nabla \left( \phi_i^\eps\left(\frac{1}{\omega_\eps}-1\right) \right) \nabla \hat{\varphi}_j dA_{g_\eps} +  \int_{\tilde{\Sigma}_\eps} \nabla \phi_i^\eps \nabla \hat{\varphi}_j dA_{g_\eps} \\
=  \int_{\tilde{\Sigma}_\eps} \nabla \left( \phi_i^\eps\left(\frac{1}{\omega_\eps}-1\right) \right) \nabla \hat{\varphi}_j dA_{g_\eps} + \lambda_i^\eps \beta_\eps(\phi_i^\eps,\hat{\varphi}_j) 
\end{split}
\end{equation*}
so that
\begin{equation*}
\begin{split}
\left(\lambda - \lambda_i^\eps \right) & \int_\Sigma \eta_\eps \frac{\phi^i_\eps}{\omega_\eps} \varphi_j  dA_g =  I + II \hspace{1mm} / + III \hspace{1mm} / + IV + V \hspace{1mm} / + VI + VII  \\ = & \int_{\Sigma}  \frac{\phi_i^\eps}{\omega_\eps} \nabla \eta_\eps  \nabla \varphi dA_{\tilde{g}_\eps} + \int_{\tilde{\Sigma}_\eps} \left( \eta_\eps-1 \right) \nabla \left( \frac{\phi_i^\eps}{\omega_\eps} \right) \nabla \hat{\varphi}_j dA_{\tilde{g}_\eps} \\
& + \left( \int_{\tilde{\Sigma}_\eps} \left\langle \nabla \left( \frac{\phi_i^\eps}{\omega_\eps} \right) \nabla \hat{\varphi}_j \right\rangle_{\tilde{g}_\eps} dA_{\tilde{g}_\eps} - \int_{\tilde{\Sigma}_\eps} \left\langle \nabla \left( \frac{\phi_i^\eps}{\omega_\eps} \right) \nabla \hat{\varphi}_j \right\rangle_{g_\eps} dA_{g_\eps} \right) \\
& +  \int_{\tilde{\Sigma}_\eps} \nabla \left( \phi_i^\eps\left(\frac{1}{\omega_\eps}-1\right) \right) \nabla \hat{\varphi}_j dA_{g_\eps} + \lambda_i^\eps \beta_\eps\left(\phi_i^\eps \left(1-\frac{1}{\omega_\eps}\right),\hat{\varphi}_j\right) \\
&+ \lambda_i^\eps \left( \beta_\eps\left( \frac{\phi_i^\eps}{\omega_\eps},\hat{\varphi}_j\right) - \int_{\tilde{\Sigma}_\eps} \frac{\phi_i^\eps}{\omega_\eps}\hat{\varphi}_j dA_{\tilde{g}_\eps} \right) + \lambda_i^\eps \int_{\tilde{\Sigma}_\eps} \frac{\phi_i^\eps}{\omega_\eps}\varphi_j \left(\eta_\eps dA_{g_\eps} - dA_g \right)  
\end{split}
\end{equation*}
and by Cauchy-Schwarz inequalities
$$ I^2 \leq C   \int_\Sigma \left\vert \nabla \eta_\eps \right\vert_g^2 dA_g \int_{\tilde{\Sigma}_\eps} \left( \frac{\phi_i^\eps}{\omega_\eps} \right)^2 dA_{\tilde{g}_\eps} \leq \frac{C}{\ln \frac{1}{\eps}} \int_{\tilde{\Sigma}_\eps} \left( \frac{\phi_i^\eps}{\omega_\eps} \right)^2 dA_{\tilde{g}_\eps} \text{ By \eqref{eq:defetaepscap}} $$
$$ II^2 \leq \int_{\tilde{\Sigma}_\eps} \left\vert \nabla  \frac{\phi_i^\eps}{\omega_\eps}\right\vert^2 dA_{\tilde{g}_\eps} \int_{D_{\sqrt{\eps}}(p,q) \setminus D_\eps(p,q)} \left\vert \nabla \hat{\varphi}_j \right\vert_g^2 dA_g \leq C \eps \int_{\tilde{\Sigma}_\eps} \left\vert \nabla  \frac{\phi_i^\eps}{\omega_\eps}\right\vert^2 dA_{\tilde{g}_\eps}   $$
$$ III^2 \leq \delta_\eps^2 \int_{\tilde{\Sigma}_\eps} \left\vert \nabla  \frac{\phi_i^\eps}{\omega_\eps}\right\vert_{\tilde{g}_\eps}^2 dA_{\tilde{g}_\eps} \text{ By \eqref{eq:disttilderegularized}}  $$
\begin{equation*} 
\begin{split}IV^2 \leq & 2 C \int_{\tilde{\Sigma}_\eps} \left( 1-\frac{1}{\omega_\eps}\right)^2 \vert \nabla \phi_i^\eps \vert_{\tilde{g}_\eps}^2  dA_{\tilde{g}_\eps} \int_{\tilde{\Sigma}_\eps}  \vert \nabla \phi_i^\eps \vert_{\tilde{g}_\eps}^2  dA_{\tilde{g}_\eps} 
\\ & + 2C \int_{\tilde{\Sigma}_\eps} \left(\frac{\phi_i^\eps}{\omega_\eps^2}\right)^2 dA_{\tilde{g}_\eps} \int_{\tilde{\Sigma}_\eps} \vert \nabla \omega_\eps\vert_{\tilde{g}_\eps}^2 dA_{\tilde{g}_\eps}  \\
\leq & C' \delta_\eps \left( \int_{\tilde{\Sigma}_\eps} \vert \nabla \phi_i^\eps \vert_{\tilde{g}_\eps}^2  dA_{\tilde{g}_\eps} + \int_{\tilde{\Sigma}_\eps} \left(\frac{\phi_i^\eps}{\omega_\eps^2}\right)^2 dA_{\tilde{g}_\eps} \right) \text{by \eqref{eqomegaepsto1}} 
\end{split}
\end{equation*}
$$ V^2 \leq C \beta_\eps( \phi_i^\eps ,\phi_i^\eps) \beta_\eps \left(1-\frac{1}{\omega_\eps},1-\frac{1}{\omega_\eps}\right) \leq C \delta_\eps \beta_\eps( \phi_i^\eps ,\phi_i^\eps) $$
since $L_\eps(\omega_\eps^2) = L_\eps(1) + L_\eps(\theta_\eps^2) = 1 + O(\delta_\eps)$ and $\omega_\eps \geq 1$ (beginning of the proof of Claim \eqref{cl:cvreplacement})
$$ VI^2 \leq C \delta_\eps \left( \int_{\tilde{\Sigma}_\eps} \left\vert \nabla \frac{\phi_i^\eps}{\omega_\eps} \right\vert_{\tilde{g}_\eps}^2  dA_{\tilde{g}_\eps} + \int_{\tilde{\Sigma}_\eps} \left(\frac{\phi_i^\eps}{\omega_\eps}\right)^2 dA_{\tilde{g}_\eps} \right)  \text{ by \eqref{eq:disttilderegularized}}  $$
and
$$ VII^2 \leq C \left(\eps^2+\delta_\eps^2\right) \int_{\tilde{\Sigma}_\eps} \left(\frac{\phi_i^\eps}{\omega_\eps}\right)^2 dA_{\tilde{g}_\eps}$$
Then using again that $\left\Vert  \frac{\phi_i^\eps}{\omega_\eps} \right\Vert_{H^1(\tilde{g}_\eps)}^2$ and $\left\Vert  \phi_i^\eps \right\Vert_{H^1(\tilde{g}_\eps)}^2$ are controled by $\beta_\eps(\phi_i^\eps,\phi_i^\eps)$, and using \eqref{eq:defdeltaeps} we obtain
$$  \sum_{j=1}^{a_{j_0}} \left( (\lambda-\lambda_i^\eps) \int_{\Sigma} \eta_\eps \frac{\phi_\eps^i}{\omega_\eps} \varphi_j  \right)^2 \leq \frac{C}{\ln \frac{1}{\eps}} \beta_\eps(\phi_i^\eps,\phi_i^\eps). $$
We then use that $\vert\lambda-\lambda_i^\eps\vert$ is uniformy lower bounded to conclude the proof.
\end{proof}

\begin{cl} \label{cl:estimatediffeigen} If $k$ is such that for any $x$, $\vert\partial_k F(x)\vert < 0$,
$$  \vert \lambda_k^\eps - \lambda_k \vert = O\left(\frac{1}{\ln \frac{1}{\eps}}\right) $$
and 
$$ \int_{C_{l,\eps}} \left\vert \nabla \frac{\Phi_\eps}{\omega_\eps} \right\vert_{g_\eps}^2 dA_{g_\eps}  =: V_\eps = O\left(\frac{1}{\ln \frac{1}{\eps}}\right) $$
as $\eps \to 0$.
\end{cl}

\begin{proof}
\noindent\textbf{Step 1:} $\lambda_k^\eps \to \lambda_k$ as $\eps \to 0$:

\medskip 

\noindent\textbf{Proof of Step 1:} We let $\nu_k := \lim_{\eps\to 0} \lambda_k^\eps$. By Claim \ref{cl:upperboundlambdakeps}, $\nu_k \leq \lambda_k$ and 
$$  F(\lambda_1,\cdots,\lambda_m) \leq F( \nu_1,\cdots,\nu_m ) = \lim_{\eps\to 0} F( \lambda_1^\eps,\cdots,\lambda_m^\eps ) = F(\lambda_1,\cdots,\lambda_m) $$
and by monotonicity assumptions on $F$, we deduce $\lambda_k = \nu_k$.

\medskip

From now on, we let $j$ be such that $\mu_j = \lambda_k$.

\medskip

\noindent\textbf{Step 2:} 
$$ \sum_{i\in A_j^\eps} \left( \int_\Sigma \eta_\eps^2 \left\vert \nabla \left( \frac{\phi^i_\eps}{\omega_\eps}\right) \right\vert^2_g dA_g + \int_{C_{l,\eps}} \left\vert \nabla \left( \frac{\phi^i_\eps}{\omega_\eps}\right) \right\vert^2_{\tilde{g}_\eps} dA_{\tilde{g}_\eps} \right) \leq 
\sum_{i\in A_j^\eps}\beta_\eps\left(\phi_\eps^i,\phi_\eps^i\right) \lambda_i^\eps + O\left(\delta_\eps^{\frac{1}{2}}\right) $$
and
$$ \sum_{i\in A_j^\eps} \int_\Sigma \left( \eta_\eps \frac{\phi^i_\eps}{\omega_\eps} - \pi_{< \lambda_k}\left( \eta_\eps \frac{\phi^i_\eps}{\omega_\eps} \right) \right)^2 dA_g  \geq \sum_{i\in A_j^\eps}\beta_\eps\left(\phi_\eps^i,\phi_\eps^i\right) \left(1-   O\left(\frac{1}{\ln \frac{1}{\eps}}\right)\right)  $$ 

\medskip

\noindent\textbf{Proof of Step 2:} 
Let's prove the first inequality. We have 
$$ \sum_{i\in A_j^\eps} \left\vert \nabla \frac{\phi_\eps^i}{\omega_\eps} \right\vert^2 = \sum_{i\in A_j^\eps} \frac{\left\vert \nabla \phi_\eps^i \right\vert^2}{\omega_\eps^2} + \sum_{i\in A_j^\eps} \left(\frac{\phi_\eps^i}{\omega_\eps}\right)^2 \frac{\left\vert \nabla \omega_\eps \right\vert^2}{\omega_\eps^2} - 2 \sum_{i\in A_j^\eps} \frac{\phi_i^\eps \nabla  \phi_i^\eps}{\omega_\eps} \frac{\nabla \omega_\eps}{\omega_\eps^2} $$ 
so that 
$$ \int_{\tilde{\Sigma}_\eps} \sum_{i\in A_j^\eps} \left\vert \nabla \frac{\phi_\eps^i}{\omega_\eps} \right\vert^2 \leq \sum_{i\in A_j^\eps}\beta_\eps\left(\phi_\eps^i,\phi_\eps^i\right) \lambda_i^\eps + O\left(\delta_\eps\right) + 2  \left(\sum_{i\in A_j^\eps} \int_{\tilde{\Sigma}_\eps} \left\vert \nabla \phi_i^\eps \right\vert^2 \right)^{\frac{1}{2}}  \delta_\eps^{\frac{1}{2}} = O\left(\delta_\eps^{\frac{1}{2}}\right) $$
and the first inequality follows. For the second inequality, we notice that
$$ \int_\Sigma \left( \eta_\eps \frac{\phi^i_\eps}{\omega_\eps} - \pi_{< \lambda_k}\left( \eta_\eps \frac{\phi^i_\eps}{\omega_\eps} \right) \right)^2 dA_g = \int_\Sigma \eta_\eps^2 \left(\frac{\phi^i_\eps}{\omega_\eps} \right)^2 dA_g -  \sum_{l=0}^{L(j)} \left(\int_\Sigma \varphi_{l} \eta_\eps \frac{\phi^i_\eps}{\omega_\eps} dA_g \right)^2 $$
where $\varphi_0,\cdots,\varphi_{L(j)}$ is an orthonormal family of all eigenfunctions on $(\Sigma,g)$ associated to eigenvalues $\lambda$ such that $\lambda < \mu_j$ and $L(j)$ is the maximal integer $l$ such that $\lambda_l < \mu_j$ and we obtain that
\begin{equation*}
\begin{split} \int_\Sigma \eta_\eps^2 \left(\frac{\phi^i_\eps}{\omega_\eps} \right)^2 dA_g = & \int_{\tilde{\Sigma}_\eps} \left(\eta_\eps^2-1\right) \left(\frac{\phi^i_\eps}{\omega_\eps} \right)^2 dA_{g_\eps} + \left( \int_{\tilde{\Sigma}_\eps}  \left(\frac{\phi^i_\eps}{\omega_\eps} \right)^2 dA_{g_\eps} - \beta_\eps\left(\frac{\phi^i_\eps}{\omega_\eps},\frac{\phi^i_\eps}{\omega_\eps}\right) \right) \\
& + L_\eps\left(  \left(\phi_\eps^i\right)^2 \left(\frac{1}{\omega_\eps^2}-1\right) \right) + \beta_\eps(\phi_\eps^i,\phi_\eps^i) 
\end{split} \end{equation*}
and summing over $A_j^\eps$ yields:
$$ \sum_{i\in A_j^\eps} \int_\Sigma \eta_\eps^2 \left(\frac{\phi^i_\eps}{\omega_\eps} \right)^2 dA_g \geq \sum_{i\in A_j^\eps}\beta_\eps(\phi_\eps^i,\phi_\eps^i) - O\left( \eps + \delta_\eps \right)  $$
and Lemma \ref{lem:mainproj} concludes the proof of Step 2.

\medskip

\noindent\textbf{Step 3:}
We test $\sqrt{\lambda_i^\eps} \left(\eta_\eps \frac{\phi_i^\eps}{\omega_\eps} - \pi_{< \lambda_k}\left(\eta_\eps \frac{\phi_i^\eps}{\omega_\eps}\right)\right)$ in the variational characterization of $\mu_j = \lambda_k = \lambda_{L(j)+1}$ for any $i\in A_j^\eps$. We obtain
\begin{equation*}
\begin{split} \lambda_k = \mu_j \leq \frac{\sum_{i\in A_j^\eps} \lambda_i^\eps \int_\Sigma \vert\nabla \eta_\eps \frac{\phi_i^\eps}{\omega_\eps} \vert^2_g dA_g }{ \sum_{i\in A_j^\eps} \lambda_i^\eps \int_\Sigma \left(\eta_\eps \frac{\phi_i^\eps}{\omega_\eps} - \pi_{<\lambda_k}\left(\eta_\eps \frac{\phi_i^\eps}{\omega_\eps} \right) \right)^2 dA_g }  =: \frac{N_j^\eps}{D_j^\eps}
\end{split}
\end{equation*}
\begin{equation*}
\begin{split}
& N_j^\eps  = \sum_{i\in A_j^\eps} \lambda_i^\eps \left( \int_\Sigma \eta_\eps^2 \left\vert\nabla  \frac{\phi_i^\eps}{\omega_\eps} \right\vert^2_g dA_g + 2 \int_\Sigma \eta_\eps  \frac{\phi_i^\eps}{\omega_\eps} \nabla \eta_\eps \nabla  \frac{\phi_i^\eps}{\omega_\eps}  dA_g + \int_\Sigma \left(\frac{\phi_i^\eps}{\omega_\eps}\right)^2 \vert\nabla  \eta_\eps \vert^2_g dA_g \right) \\
& \leq \sum_{i\in A_j^\eps} \lambda_i^\eps \left( \int_{\tilde{\Sigma}_{\eps}} \left\vert\nabla  \frac{\phi_i^\eps}{\omega_\eps} \right\vert^2_g dA_g - \int_{C_{l,\eps}} \left\vert\nabla  \frac{\phi_i^\eps}{\omega_\eps} \right\vert^2_g dA_g +  \int_\Sigma \eta_\eps  \nabla \eta_\eps \nabla \left( \frac{\phi_i^\eps}{\omega_\eps} \right)^2 dA_g \right)  + O\left(\frac{1}{\ln \frac{1}{\eps}}\right)
\end{split}
\end{equation*}

\medskip

From now on, all the following computations come by summing over all $k$ such that $F$ is not constant with respect to the $k$-th coordinate of $F$. However, if $k$ does not satisfy this property, we have that $t_k^\eps = t_k = 0$ that appear in all the terms we add if we sum over all the coordinates $k$. Therefore, we can sum over all $j \in \{1,\cdots,J\}$
$$ \sum_{j =1}^J\sum_{i\in A_j^\eps} \lambda_i^\eps \int_\Sigma \eta_\eps  \nabla \eta_\eps \nabla \left( \frac{\phi_i^\eps}{\omega_\eps} \right)^2 dA_g = - \int_{\Sigma} \eta_\eps \nabla \eta_\eps \nabla  \left( \frac{\theta_\eps}{\omega_\eps} \right)^2 $$
and using a Cauchy-Schwarz inequality,
$$ \left\vert \int_{\Sigma} \eta_\eps \nabla \eta_\eps \nabla  \left( \frac{\theta_\eps}{\omega_\eps} \right)^2 \right\vert \leq O\left(\frac{1}{\sqrt{\ln \frac{1}{\eps}}}\right) \int_{\tilde{\Sigma}_{\eps}} \left( \frac{\theta_\eps}{\omega_\eps} \right)^2  \left\vert \nabla \left( \frac{\theta_\eps}{\omega_\eps} \right)^2 \right\vert^2 = O\left(\sqrt{\frac{\delta_\eps}{\ln \frac{1}{\eps}}}\right) $$ 
and we obtain that
$$ \sum_{j=1}^J N_j^\eps  \leq \int_{\tilde{\Sigma}_{\eps}} \left\vert\nabla  \frac{\Phi_\eps}{\omega_\eps} \right\vert^2_{g_\eps,\Lambda_\eps} dA_{g_\eps} - \int_{C_{l,\eps}} \left\vert\nabla  \frac{\Phi_\eps}{\omega_\eps} \right\vert^2_{g_\eps,\Lambda_\eps} dA_{g_\eps} + O\left(\sqrt{\frac{\delta_\eps}{\ln \frac{1}{\eps}}}\right) + O\left(\frac{1}{\ln \frac{1}{\eps}}\right) $$
where we compute the first right-hand term as 
\begin{equation*}\begin{split} \int_{\tilde{\Sigma}_{\eps}} \left\vert\nabla  \frac{\Phi_\eps}{\omega_\eps} \right\vert^2_{g_\eps,\Lambda_\eps} dA_{g_\eps} + O(\delta_\eps) =  \int_{\tilde{\Sigma}_{\eps}} \left\vert\nabla \Phi_\eps \right\vert^2_{g_\eps,\Lambda_\eps} dA_{g_\eps}  \\ =  \sum_{i=1}^{n_\eps} \lambda_i^\eps \int_{\tilde{\Sigma}_{\eps}} \left\vert\nabla \phi_i^\eps \right\vert^2_{g_\eps} dA_{g_\eps} =  \sum_{i=1}^{n_\eps} \left(\lambda_i^\eps\right)^2 \beta_\eps(\phi_i^\eps,\phi_i^\eps) = \sum_{k=1}^m \left(\lambda_k^\eps\right)^2 t_k^\eps \end{split}\end{equation*}
where the latter inequality comes from \eqref{eq:sumtieps}. As a consequence we obtain
$$ \sum_{j=1}^J \mu_j  D_j^\eps \leq \sum_{j=1}^J N_j^\eps \leq \sum_{k=1}^m \left(\lambda_k^\eps\right)^2 t_k^\eps - W_\eps +O \left(\frac{1}{\ln \frac{1}{\eps}}\right)  $$
where 
$$ W_\eps = \int_{C_{l,\eps}} \left\vert \nabla \frac{\Phi_\eps}{\omega_\eps} \right\vert_{g_\eps, \Lambda_\eps}^2 dA_{g_\eps}. $$
In addition, we have from Step 2 and \eqref{eq:sumtieps} again that
$$ \sum_{j=1}^J \mu_j  D_j^\eps \geq \sum_{k=1}^m t_k^\eps \lambda_k\lambda_k^\eps \left(1 - O\left(\frac{1}{\ln \frac{1}{\eps}}\right)\right) $$
as $\eps\to 0$. We deduce that
$$ \sum_{k=1}^m (\lambda_k-\lambda_k^\eps) \lambda_k^\eps t_k^\eps + W_\eps \leq O\left(\frac{1}{\ln \frac{1}{\eps}}\right) $$
as $\eps \to 0$. Then
$$ \sum_{k ; \lambda_k > \lambda_k^\eps} (\lambda_k-\lambda_k^\eps) \lambda_k^\eps t_k^\eps + W_\eps \leq \sum_{k ; \lambda_k^\eps > \lambda_k} (\lambda_k^\eps-\lambda_k) \lambda_k^\eps t_k^\eps + O\left(\frac{1}{\ln \frac{1}{\eps}}\right) \leq O\left(\frac{1}{\ln \frac{1}{\eps}}\right) $$
as $\eps \to 0$ by Claim \ref{cl:upperboundlambdakeps}. Since $\lambda_k^\eps$ and $t_k^\eps$ are uniformly lower bounded by a positive constant, we obtain the expected claim.
\end{proof}


\subsection{A replacement of $\Phi_\eps$ in $\Sigma$}
We let $\Psi_\eps$ be equal to $\frac{\Phi_\eps}{\omega_\eps}$ in $\Sigma_\eps$ and the harmonic extension of $\frac{\Phi_\eps}{\omega_\eps}$ in $\mathbb{D}_{\eps}(p)\cup \mathbb{D}_\eps(q)$.

We set $B_\eps$ the bilinear form on $H^1(\Sigma,\mathbb{R}^{n_\eps})$ defined as
$$ B_\eps(X,X):= \sum_{j=1}^J \sum_{i \in A_j^\eps} \left( \int_\Sigma \vert \nabla X_i \vert^2_g dA_g - \mu_j \int_\Sigma \left( X_i - \pi_{<\mu_j}(X_i) \right)^2 \right) $$
where $\pi_{< \lambda}$ is the projection on the sum of the eigenspaces associated to eigenvalues on $(\Sigma,g)$ strictly less than $\lambda$. It is clear that $B$ is a non-negative bilinear form so that for any $X : \Sigma \to \R^n$, the Cauchy-Schwarz inequality gives that
$$ B_\eps(X, \Psi_\eps) \leq \sqrt{B_\eps(\Psi_\eps,\Psi_\eps)} \sqrt{B_\eps(X,X)} $$
We aim at estimating $B_\eps(\Psi_\eps,\Psi_\eps)$. We recall the definition of
$$ V_\eps := \int_{C_{l,\eps}} \left\vert \nabla \frac{\Phi_\eps}{\omega_\eps} \right\vert^2  =O\left(\frac{1}{\ln \frac{1}{\eps}}\right)$$

\begin{lem}[Karpukhin-Kusner-McGrath-Stern \cite{kkms}] \label{lem:extension}
There is a universal constant such that for any $l>0$ and $\psi$ a $H^1$ function defined on the cylinder $C_l:= \mathbb{S}^1 \times [0, \frac{l}{2} ] $, the harmonic extension of $\psi$ on $\mathbb{S}^1\times \{0\} \to \R$ to the disk $\hat{\psi} : \mathbb{D}\to \R$ satisfies
$$ \int_\mathbb{D} \vert \nabla \hat{\psi} \vert^2 \leq \left(1+Ce^{-l}\right) \int_{\left[0,\frac{l}{2}\right]\times\mathbb{S}^1} \vert \nabla \psi \vert^2 $$
\end{lem}

\begin{proof} We reduce the study of this estimate to the case of $\psi : C_l \to \R$ being the energy minimizing extension of the eigenfunction of the circle $f : \mathbb{S}^1 \times \{0\} \to \R $ with eigenvalue $k^2$. This map $\psi(\theta,s)$ can be obtained explicitly as a solution by separation of variables of the equation
$$ \Delta \psi = 0 \text{ in } C_l \text{ and } \psi = f \text{ on } \mathbb{S}^1 \times \{0\} \text{ and } \partial_s \psi = 0 \text{ on } \mathbb{S}^1 \times \left\{\frac{l}{2} \right\}.$$
$$ \psi(e^{i \theta},s) = \left( \frac{e^{ks}}{1+e^{kl}} + \frac{e^{-ks}}{1+e^{-kl}} \right) f(e^{i\theta}).  $$
We then compute the energy
$$ \int_{C_l} \vert \nabla \psi \vert^2 = - \int_{\mathbb{S}^1\times\{0\}} \psi \partial_s \psi = \left( \frac{k}{1+e^{-kl}} - \frac{k}{1+e^{kl}} \right) \int_{\mathbb{S}^1} f^2  $$
and we have
$$ \frac{k}{1+e^{-kl}} - \frac{k}{1+e^{kl}} = \frac{\vert k \vert}{1+e^{-\vert k\vert l}} - \frac{\vert k\vert}{1+e^{\vert k\vert l}} \geq \vert k \vert \left( \frac{1}{1+e^{- l}} - \frac{1}{1+e^{l}} \right) \geq \vert k \vert \left( 1 - 4 e^{-l}\right)  $$
and
$$ \int_{C_l} \vert \nabla \psi \vert^2 \geq \vert k \vert \left( 1 - 4 e^{-l}\right) \int_{\mathbb{S}^1} f^2 . $$
Now, the harmonic extension of $f$ in $\mathbb{D}$ satisfies $\hat{\psi}(e^{i\theta},r) = r^k f(e^{i \theta})$ and
$$ \vert \nabla \hat{\psi} \vert^2 = k^2 r^{2k-2} f^2 + r^{2k-2} \vert \partial_\theta f \vert^2 $$
so that
\begin{equation*}
\begin{split} \int_{\mathbb{D}} \vert \nabla \hat{\psi} \vert^2 = \int_{0}^{1} \left( \int_{\mathbb{S}^1} k^2 r^{2k-2} f^2 + r^{2k-2} \vert \partial_\theta f \vert^2\right)d\theta rdr \\ = \frac{1}{2k} \left( k^2 \int_{\mathbb{S}^1} f^2 + \int_{\mathbb{S}^1} \vert \partial_\theta f \vert^2 \right) = k \int_{\mathbb{S}^1}f^2
\end{split} \end{equation*}
and we obtain that
$$ (1-4 e^{-l}) \int_{\mathbb{D}} \vert \nabla \hat{\psi} \vert^2 \leq \int_{C_l}  \vert \nabla \psi \vert^2 $$
and the estimate follows.
\end{proof}

\begin{prop} \label{prop:estonBPsieps}
There is a constant $C>0$ such that
$$ B(\Psi_\eps,\Psi_\eps) \leq C \frac{V_\eps}{l}   + o\left(\frac{1}{\ln \frac{1}{\eps}}\right)  $$
and
$$  \sum_{j=1}^{J} \sum_{i\in A_j^\eps}\int_\Sigma \left(\pi_{<\mu_j}(\psi_i^\eps)\right)^2 dA_g \leq C \frac{V_\eps}{l} + O(\delta_\eps) $$
as $\eps\to 0$.
\end{prop}

\begin{proof}

\noindent\textbf{Step 1:} We prove
\begin{equation}\label{eq:leqBPsieps} B_\eps(\Psi_\eps,\Psi_\eps) \leq C e^{-l} \int_{C_{l,\eps}} \left\vert \nabla \frac{\Phi_\eps}{\omega_\eps} \right\vert^2 + \sum_{k=1}^{n_\eps} (\lambda_k^\eps-\lambda_k) t_k^\eps +  \sum_{j=1}^J \sum_{l=0}^{L(j)} \mu_j \sum_{i\in A_j^\eps} \left(\int_\Sigma \psi_i^\eps \varphi_l \right)^2 + O(\delta_\eps^{\frac{1}{2}}) \end{equation}
where $\left( \varphi_l \right)_{0 \leq l \leq L(j)}$ is an orthonormal family of all eigenfunctions on $(\Sigma,g)$ associated to eigenvalues $\lambda < \mu_j$ and $L(j)$ is the maximal integer $l$ such that $\lambda_l < \mu_j$.

\medskip

\noindent\textbf{Proof of Step 1:} We have by Lemma \ref{lem:extension} rescaled to $C_{l,\eps}$ and $\mathbb{D}_\eps$ that
$$ \int_{\Sigma} \vert \nabla \Psi_\eps \vert^2_g dA_g \leq \int_{\tilde{\Sigma}_\eps} \left\vert \nabla \frac{\Phi_\eps}{\omega_\eps} \right\vert^2_{\tilde{g}_\eps} dA_{\tilde{g}_\eps} + C e^{-l} \int_{C_{l,\eps}} \left\vert \nabla \frac{\Phi_\eps}{\omega_\eps} \right\vert^2 $$
and we have 
$$ \int_{\tilde{\Sigma}_\eps} \left\vert \nabla \frac{\Phi_\eps}{\omega_\eps} \right\vert^2_{\tilde{g}_\eps} dA_{\tilde{g}_\eps} \leq \int_{\tilde{\Sigma}_\eps} \left\vert \nabla \Phi_\eps \right\vert^2_{\tilde{g}_\eps} dA_{\tilde{g}_\eps} + O\left(\delta_\eps^{\frac{1}{2}}\right) = \sum_{k=1}^{n_\eps} \lambda_k^\eps t_k^\eps + O(\delta_\eps^{\frac{1}{2}})$$
Given $1\leq j \leq J$, we also have that
$$  \int_\Sigma ( \psi_i^\eps - \pi_{<\mu_j}(\psi_i^\eps)  )^2 =  \int_\Sigma \left(\psi_i^\eps\right)^2 -  \sum_{l =0}^{ L(j)} \left(\int_\Sigma \psi_i^\eps \varphi_l\right)^2.  $$
We compute
\begin{equation*} 
\begin{split} \int_\Sigma \left(\psi_i^\eps\right)^2 = &  \int_{\mathbb{D}_\eps(p,q)} \left(\psi_i^\eps \right)^2 - \int_{C_{l,\eps}} \left(\frac{\phi_i^\eps}{\omega_\eps}\right)^2  + \left( \int_{\tilde{\Sigma}_\eps} \left(\frac{\phi_i^\eps}{\omega_\eps}\right)^2  -   \beta_\eps\left(\frac{\phi_i^\eps}{\omega_\eps},\frac{\phi_i^\eps}{\omega_\eps}\right) \right) \\ 
& + L_\eps\left(\left(\phi_i^\eps\right)^2\left(\frac{1}{\omega_\eps^2}-1\right)\right) + \beta_\eps(\phi_i^\eps,\phi_i^\eps) 
\end{split}
\end{equation*}
so that taking the sum over $j$,
$$ - \sum_{j=1}^J \sum_{i \in A_j^\eps} \mu_j  \int_\Sigma ( \psi_i^\eps - \pi_{<\mu_j}(\psi_i^\eps)  )^2 = - \sum_{k=1}^{n_\eps} \lambda_k t_k^\eps +  \sum_{j=1}^J \mu_j \sum_{l=0}^{L(j)} \sum_{i \in A_j^\eps} \left(\int_\Sigma \psi_i^\eps \varphi_l\right)^2 + O(\delta_\eps^{\frac{1}{2}}) $$
and \eqref{eq:leqBPsieps}  holds.

\medskip

\noindent\textbf{Step 2:} Let's estimate the second right-hand term of \eqref{eq:leqBPsieps}. We have that
$$ F(\lambda_1,\cdots,\lambda_m) \leq F(\lambda_1^\eps,\cdots,\lambda_m^\eps) \leq F(\lambda_1,\cdots,\lambda_m) + O(\delta_\eps)$$
so that 
$$ F(\Lambda)-F(\Lambda_\eps) = \sum_{k=1}^m t_k^\eps \left(\lambda_k-\lambda_k^\eps\right) + \int_0^1 \left(DF((1-t)\Lambda+t \Lambda_\eps) - DF(\Lambda_\eps)\right)\cdot\left( \Lambda-\Lambda_\eps \right)  $$
implies with the use of Claim \ref{cl:estimatediffeigen}, and assumptions on $F$
$$ \sum_{k=1}^{m} (\lambda_k^\eps-\lambda_k) t_k^\eps = O(\delta_\eps) + \sup_{t\in [0,1]} \Vert DF((1-t)\Lambda+t\Lambda_\eps) - DF(\Lambda_\eps) \Vert \cdot O\left(\frac{1}{\ln \frac{1}{\eps}}\right) = o\left(\frac{1}{\ln \frac{1}{\eps}}\right) $$
as $\eps \to 0$. Indeed, $F$ is a $\mathcal{C}^1$ function and given $k \in \{1,\cdots,m\}$, 
\begin{itemize}
\item either $\lambda_k^\eps \to \lambda_k$ as $\eps \to 0$ by Claim \ref{cl:estimatediffeigen}
\item or $F$ is a constant function with respect to the $k$-th coordinate.
\end{itemize}

\medskip

\noindent\textbf{Step 3:} Let's estimate the third right-hand term of \eqref{eq:leqBPsieps} we have that for $i\in A_j^\eps$ and $l \in \{0,\cdots,L(j)\}$, We denote $\lambda$ the eigenvalue associated to $\varphi_l$.
\begin{equation*}
\begin{split} \lambda \int_\Sigma \psi_i^\eps \varphi_l & = \int_\Sigma \psi_i^\eps \Delta_g \varphi_l = \int_\Sigma \nabla \psi_i^\eps \nabla \varphi_l \\
 &= \int_{\tilde{\Sigma}_\eps} \nabla \left( \frac{\phi_i^\eps}{\omega_\eps} \right) \nabla \hat{\varphi}_l - \int_{C_{l,\eps}} \nabla \left( \frac{\phi_i^\eps}{\omega_\eps} \right) \nabla \hat{\varphi}_l + \int_{\mathbb{D}_\eps(p)\cup \mathbb{D}_\eps(q)} \nabla \psi_i^\eps  \nabla \varphi_l \\
\end{split}
\end{equation*}
where $\hat{\varphi}_l$ denotes the harmonic extension on $\tilde{\Sigma}_\eps$ of $\varphi_l : \Sigma \setminus \mathbb{D}_\eps(p,q) \to \R$  and we have that
\begin{equation*} 
\begin{split} \int_{\tilde{\Sigma}_\eps} & \nabla \left( \frac{\phi_i^\eps}{\omega_\eps} \right) \nabla \hat{\varphi}_l  =  \int_{\tilde{\Sigma}_\eps} \nabla \left( \phi_i^\eps \left(\frac{1}{\omega_\eps}-1\right) \right) \nabla \hat{\varphi}_l + \lambda_i^\eps \beta_\eps\left( \phi_i^\eps \left(1-\frac{1}{\omega_\eps}\right), \hat{\varphi}_l \right) \\ 
 & +  \lambda_i^\eps \left( \beta_\eps\left( \frac{ \phi_i^\eps}{\omega_\eps}, \hat{\varphi}_l \right) - \int_{\tilde{\Sigma}_\eps} \frac{ \phi_i^\eps}{\omega_\eps} \hat{\varphi}_l dA_{\tilde{g}_\eps}\right)  + \lambda_i^\eps \int_{C_{l,\eps}} \frac{ \phi_i^\eps}{\omega_\eps} \hat{\varphi}_l dA_{\tilde{g}_\eps} - \lambda_i^\eps \int_{\mathbb{D}_\eps(p,q)} \psi_i^\eps \varphi_l dA_g \\
 & + \lambda_i^\eps   \int_\Sigma \psi_i^\eps \varphi_l dA_g
\end{split} \end{equation*}
so that
$$ \sum_{i\in A_{j}^\eps} \left( \int_{\tilde{\Sigma}_\eps} \nabla \left( \frac{\phi_i^\eps}{\omega_\eps} \right) \nabla \hat{\varphi}_l \right)^2 = \sum_{i\in A_{j}^\eps} \left( \lambda_i^\eps   \int_\Sigma \psi_i^\eps \varphi_l \right)^2 + O(\delta_\eps)  $$
and
\begin{equation*}
\begin{split} \sum_{i\in A_{j}^\eps} \left( \int_{C_{l,\eps}} \nabla \left( \frac{\phi_i^\eps}{\omega_\eps} \right) \nabla \hat{\varphi}_l \right)^2 \leq \sum_{i\in A_{j}^\eps} \left( \int_{C_{l,\eps}} \left\vert \nabla \left( \frac{\phi_i^\eps}{\omega_\eps} \right) \right\vert^2 \right) \left( \int_{C_{l,\eps}} \vert \nabla \hat{\varphi}_l \vert^2 \right) \\ \leq  V_\eps \left( \frac{l\eps^2}{l^2 \eps^2} \vert \varphi_l(p) - \varphi_l(q) \vert^2 +o(1) \right) \leq \frac{C V_\eps}{l} 
\end{split}\end{equation*}
since $\Vert \hat{\varphi}_l - \left( \left( t-\frac{l\eps}{2} \right) \frac{\varphi_l(q)}{l\eps}  + \left( t+\frac{l\eps}{2} \right)  \frac{\varphi_l(q)}{l\eps} \right) \Vert_{H^1\left( C_{l,\eps} \right)} \leq o(1) $ and
$$ \sum_{i\in A_{j}^\eps} \left( \int_{\mathbb{D}_\eps(p)\cup \mathbb{D}_\eps(q)} \nabla \psi_i^\eps  \nabla \varphi_l \right)^2 \leq C(1+e^{-l})V_\eps  \Vert \nabla \varphi_l \Vert_{\infty}^2 \eps^2 \leq O(\eps^2) . $$
Therefore
$$ \sum_{i\in A_{j}^\eps} \left((\lambda - \lambda_i^\eps) \int_{\Sigma} \psi_i^\eps \varphi_l \right)^2 \leq C \frac{V_\eps}{l} +O(\delta_\eps)$$
and since for $i\in A_j^\eps$, $\vert \lambda - \lambda_i^\eps \vert$ is uniformly lower bounded,
$$ \sum_{j=1}^J \mu_j  \sum_{l=0}^{L(j)} \sum_{i \in A_j^\eps} \left(\int_\Sigma \psi_i^\eps \varphi_l\right)^2 \leq C' \frac{V_\eps}{l} + O(\delta_\eps). $$
The proof of the proposition is complete.
\end{proof}

\subsection{Estimates on the rest}
We set for $i \in A_j^\eps$
$$ \psi_i^\eps =   F_i^\eps + \pi_{<\mu_j}\left( \psi_i^\eps \right) + R_i^\eps = F_i^\eps + S_i^\eps $$
where $F_i^\eps = \pi_{\mu_j}(\psi_i^\eps)$ is the projection of $\psi_i^\eps$ in on the eigenspace associated to $\mu_j$ in $L^2(\Sigma,g)$. We then have that
$$  B_\eps(R_\eps,R_\eps ) = B_\eps(\Psi_\eps,R_\eps) \leq \sqrt{B_\eps(\Psi_\eps,\Psi_\eps)}\sqrt{B_\eps(R_\eps,R_\eps)} $$
so that
$$ B_\eps(R_\eps,R_\eps) \leq B_\eps(\Psi_\eps,\Psi_\eps) \leq C   \frac{V_\eps}{l}  + o\left(\frac{1}{\ln \frac{1}{\eps}}\right) $$
as $\eps \to 0$. Since $R_i^\eps \in \bigoplus_{\lambda > \mu_j} E_\lambda$, we obtain 
$$B_\eps(R_\eps,R_\eps) \geq \sum_{j=1}^J\sum_{i\in A_j^\eps} \left(1- \frac{\mu_j}{\mu_{j+1}}\right) \int_{\Sigma} \vert \nabla R_i^\eps \vert_g^2 dA_g  $$
and using in addition Proposition \ref{prop:estonBPsieps} we obtain that
\begin{equation} \label{eq:w12estimateReps} \Vert S_\eps \Vert_{W^{1,2}}^2 \leq C   \frac{V_\eps}{l}  + o\left(\frac{1}{\ln \frac{1}{\eps}}\right)  \end{equation}
as $\eps \to 0$. In particular, letting $\eps \to 0$ and then $l\to +\infty$, we obtain that $\vert S_\eps \vert^2 + \vert \nabla S_\eps \vert^2 \to 0$ in $L^1$. In addition, we have the following claim:

\begin{cl} \label{cl:estonSeps}
$$   \left\vert \frac{1}{\eps} \int_{\mathbb{S}_\eps(p)} S_\eps dL_g \right\vert^2 +   \left\vert  \frac{1}{\eps}\int_{\mathbb{S}_\eps(q)} S_\eps dL_g \right\vert^2 \leq C \frac{1}{l} +o(1)   $$
as $\eps \to 0$.
\end{cl}
\begin{proof} We can follow the beginning of the proof of Claim \ref{cl:linftyesteigenfunctions} in order to prove that 
$$  \left\vert \frac{1}{\eps} \int_{\mathbb{S}_\eps(p)} S_\eps dL_g \right\vert^2 +   \left\vert \frac{1}{\eps} \int_{\mathbb{S}_\eps(q)} S_\eps dL_g \right\vert^2 \leq C\ln \frac{1}{\eps} \Vert S_\eps \Vert_{W^{1,2}}^2 $$
and \eqref{eq:w12estimateReps} and the estimate on $V_\eps$ in Claim \ref{cl:estimatediffeigen} complete the proof of the claim.
\end{proof}

\subsection{Conclusion}
We have that $\Psi_\eps = F_\eps + S_\eps$, where
$$  \left\vert \frac{1}{\eps} \int_{\mathbb{S}_\eps(p)} \Psi_\eps dL_{\tilde{g}_\eps} - \frac{1}{\eps} \int_{\mathbb{S}_\eps(q)} \Psi_\eps dL_{\tilde{g}_\eps} \right\vert^2 =  \left\vert \frac{1}{\eps} \int_{C_{l,\eps}} \partial_t \Psi_\eps \right\vert^2 \leq C \frac{l\eps^2}{\eps^2} V_\eps \leq O\left(\frac{1}{\ln \frac{1}{\eps}}\right)  $$
as $\eps\to 0$ so that from Claim \ref{cl:estonSeps}, estimates on $S_\eps$ give that
\begin{equation} \label{eqmeanofFeps}  \left\vert \frac{1}{\eps} \int_{\mathbb{S}_\eps(p)} F_\eps dL_g - \frac{1}{\eps} \int_{\mathbb{S}_\eps(q)} F_\eps dL_g \right\vert^2 \leq C \frac{1}{l} +o(1) + O\left(\frac{1}{\ln \frac{1}{\eps}}\right)  \end{equation}
as $\eps\to 0$. 

\begin{cl}[Petrides-Tewodrose \cite{pt}, Mixing lemma, lemma 2.1] \label{cl:rearrangement} There is an orthogonal family $\varphi_\eps := \left(\varphi_1^\eps,\cdots,\varphi_n^\eps\right)$ of eigenfunctions associated to $\lambda_1,\cdots,\lambda_m$ in $E:= \bigoplus_{j=1}^J E_{\mu_j}$ where $n$ is the dimension of $E$ such that for any bilinear map $A : E \times E \to F$ where $F$ is a vector space,
$$ \sum_{i=1}^n A\left(\varphi_i^\eps,\varphi_i^\eps\right) = \sum_{i=1}^{n_\eps} A\left(F_i^\eps,F_i^\eps\right) $$
\end{cl}
We apply this claim to 
\begin{itemize}
\item $A(f,f) = \Vert f \Vert_{L^2}^2 + \Vert \nabla f \Vert_{L^2}^2 $. Then $(\varphi_i^\eps)$ is bounded in $W^{1,2}$. Since it  belongs to the space of eigenfunctions associated to $\left(\mu_j\right)_{j=1,\cdots, J}$ on $(\Sigma,g)$ it belongs to a finite dimensional space. Then, up to the extraction of a subsequence, $\varphi_\eps$ converges to some map $\Phi$ in $C^k$ for any $k$. 
\item $A$ defined by $A(f_i,f_i) = \lambda_i f_i^2$ for $f_i$ an eigenfunction associated to $\lambda_i$. Then
$ \vert \varphi_\eps \vert_\Lambda^2 = \vert F_\eps \vert_\Lambda^2 = \vert \Psi_\eps - S_\eps \vert_\Lambda^2 $ so that on $\Sigma_\eps$,
$$ \vert \varphi_\eps \vert_\Lambda^2 - 1 = \left( \left\vert \frac{\Phi_\eps}{\omega_\eps} \right\vert_{\Lambda_{\eps}}^2 - 1 \right) + \left( \left\vert \frac{\Phi_\eps}{\omega_\eps} \right\vert_{\Lambda}^2-\left\vert \frac{\Phi_\eps}{\omega_\eps} \right\vert_{\Lambda_{\eps}}^2  \right) + \left( \vert \Psi_\eps - S_\eps \vert_\Lambda^2 - \vert \Psi_\eps \vert^2  \right)  $$
We have that $\vert S_\eps \vert^2$ and $\theta_\eps^2$ converge to $0$ in $L^1$ and that $\left(\lambda_i^\eps - \lambda_i \right) t_i^\eps$ converges to $0$ (by Claim \ref{cl:estimatediffeigen} and assumptions on $F$). Then $\vert \varphi_\eps \vert_\Lambda^2 - 1$ converges to $0$ and at the limit, $\vert \Phi \vert_{\Lambda}^2 =1$.
\item $A(f,f) =  \left(  \frac{1}{\eps} \int_{\mathbb{S}_\eps(p)} f dL_g -  \frac{1}{\eps} \int_{\mathbb{S}_\eps(q)} f dL_g \right)^2 $. It is clear that
$$ \left\vert  \frac{1}{\eps} \int_{\mathbb{S}_\eps(p)} \varphi_\eps dL_g -  \frac{1}{\eps} \int_{\mathbb{S}_\eps(q)} \varphi_\eps dL_g \right\vert \to \vert \Phi(p) - \Phi(q) \vert$$
as $\eps\to 0$. Letting $\eps \to 0$ and then $l\to +\infty$ in \eqref{eqmeanofFeps}, we obtain that $\Phi(p) = \Phi(q)$. 
\item $A(f,f) = df \otimes df - \frac{\vert \nabla f \vert^2_g}{2} g$. We let $h  \in S^2_0(\Sigma)$ such that $supp(h) \subset \Sigma \setminus \{p,q\}$. Then
$$ \int_\Sigma \left( \sum_{i=1}^n A(\varphi_i^\eps,\varphi_i^\eps), h\right)_{g_\eps} dA_{g_\eps} = \int_{\tilde{\Sigma}_\eps} \left( \sum_{i=1}^{n_\eps} A(\Psi_i^\eps - S_i^\eps,\Psi_i^\eps - S_i^\eps), h\right)_{g_\eps} dA_{g_\eps} $$
so that
\begin{equation*}
\begin{split}
& \left\vert \int_\Sigma \left( \sum_{i=1}^n A(\varphi_i^\eps,\varphi_i^\eps), h\right)_{g_\eps} dA_{g_\eps} \right\vert \leq 2 \delta_\eps \Vert h \Vert_{g_\eps}  \\
& + \left\vert  \int_{\tilde{\Sigma}_\eps} \left( \sum_{i=1}^{n_\eps} \left( A(\Phi_i^\eps,\Phi_i^\eps) - A\left(\frac{\Phi_i^\eps}{\omega_\eps} - S_i^\eps,\frac{\Psi_i^\eps}{\omega_\eps} - S_i^\eps\right) \right), h\right)_{g_\eps} dA_{g_\eps} \right\vert
\end{split}
\end{equation*}
and since the right-hand term converges to $0$ and $g_\eps \to g$ in $supp(h)$, we obtain letting $\eps\to 0$ that 
$$ \int_\Sigma \left( \sum_{i=1}^n A(\Phi,\Phi), h\right)_{g} dA_{g} = 0 $$
and this is true for any $h$ such that $supp(h) \subset \Sigma \setminus \{p,q\}$. We obtain
$$ d\Phi \otimes d\Phi - \frac{\vert \nabla \Phi \vert^2_g}{2} g = 0 $$
\end{itemize}
The conclusion is that $\Phi : \Sigma \to \mathcal{E}_\Lambda$ is a (possibly branched) conformal minimal immersion such that $\Phi(p) = \Phi(q)$.

At the very end of our analysis, letting $q \to p$ along a vector $X \in T_p \Sigma $, we obtain that for all $p \in \Sigma$ and $X \in T_p\Sigma$, there is a possibly branched conformal minimal immersion $ \Phi : \Sigma \to \mathcal{E}_{\Lambda} $ such that $D \Phi(p).X = 0$. Since $\Phi$ is conformal, we have that $\vert D\Phi(p)\cdot X^{\perp} \vert = \vert D \Phi(p) \cdot X \vert = 0$ where $X^\perp$ is a vector such that $g(X,X^\perp)=0$ and $g(X^\perp,X^\perp) = g(X,X)$. We obtain that $\nabla \Phi(p) = 0$. Then $p$ is a branched point. Since $p$ was chosen arbitrarily, every point of $\Sigma$ is a conical singularity for $g$ and we obtain a contradiction.

\section{Steklov spectral functionals} \label{sec:strip}

\subsection{Choice of the initial minimizing sequence }
Let $\Sigma$ be a compact surface with a non-empty boundary $\partial\Sigma$. We assume that a Riemannian metric $g$ realizes the absolute minimizer
$$ E(g,1) = \inf_{\tilde{g}\in Met_0(\Sigma)} E(\tilde{g},1) $$
We now take $p,q$ two distinct points on $\partial \Sigma$ and $l>0$ and we denote
$$ I_\eps := \partial\Sigma \setminus \left( \mathbb{D}_\eps(p) \cup \mathbb{D}_\eps(q) \right)  $$
and
$$ R_{l,\eps} := [-\eps,\eps] \times \left[-\frac{l\eps}{2},\frac{l\eps}{2}\right] $$
and we glue $\Sigma$ and $R_{l,\eps}$
$$ \tilde{\Sigma}_\eps = \left( \Sigma_\eps \cup R_{l,\eps} \right) / \sim $$
where $\sim$ is a glueing along $\partial\Sigma \cap  \mathbb{D}_\eps(p)$ and $\partial_1 R_{l,\eps} := [-\eps,\eps] \times \{-\frac{l\eps}{2}\}$ and along $\partial \Sigma \cap \mathbb{D}_\eps(q)$ and $\partial_2 R_{l,\eps}:= [-\eps,\eps] \times \{\frac{l\eps}{2}\}$ that preserves the orientation or reverses the orientation. We denote $\tilde{g}_\eps$ the $L^\infty$ metric on $\tilde{\Sigma}_\eps$ equal to $g$ on $\Sigma_\eps$ and to the flat metric on $R_{l,\eps}$. Up to a standard regularisation procedure by the heat kernel of $\tilde{g}_\eps$, we can assume that $\tilde{g}_\eps$ is continuous on $\tilde{\Sigma}_\eps$ without affecting the following estimates on eigenvalues.
We aim at computing an asymptotic expansion of $\sigma_i(\tilde{\Sigma}_\eps,\tilde{g}_\eps)$ and of $E(\tilde{g}_\eps,1)$. The proof of the following claim is similar to the proof of Claim \ref{cl:linftyesteigenfunctions}.
\begin{cl} \label{cl:linftyesteigenfunctions3}
For $\sigma>0$, there is a constant $C:= C(\Sigma,\sigma)$ such that any eigenfunction $\varphi_\eps$ associated to a Steklov eigenvalue on $(\tilde{\Sigma}_\eps,\tilde{g}_\eps)$ bounded by $\sigma$ such that $\Vert \varphi_\eps \Vert_{L^2(\partial\tilde{\Sigma}_\eps)}=1$, we have 
$$\Vert \varphi_\eps \Vert_{L^{\infty}\left(\tilde{\Sigma}_\eps\right)} \leq \sqrt{\ln \frac{1}{\eps}} $$
\end{cl}

\begin{cl} \label{cl:boundsenergysteklov}
We have that
\begin{equation} \label{eqsigmaiepsepss} \sigma_i(\tilde{\Sigma}_\eps,\tilde{g}_\eps) \geq \sigma_i(\Sigma,g) + O\left(\eps \sqrt{\ln \frac{1}{\eps}}\right)  \end{equation}
and
\begin{equation} \label{eqEepsepss3} E(\tilde{\Sigma}_\eps,\tilde{g}_\eps) \leq E(\Sigma,g) +O\left(\eps\right) \end{equation}
\end{cl}

\begin{proof}
Let $\varphi_0^\eps,\varphi_1^\eps,\cdots,\varphi_i^\eps$ be an orthonormal family of eigenfunctions associated to the eigenvalues $\sigma_0(\tilde{\Sigma}_\eps,\tilde{g}_\eps),\cdots, \sigma_i(\tilde{\Sigma}_\eps,\tilde{g}_\eps)$, we have the existence of $a_\eps \in \mathbb{S}^{i}$ such that the function $\psi_\eps := \sum_{k=0}^i a_i^\eps \varphi_i^\eps$ satisfies
\begin{equation*}
\begin{split} \sigma_i(\Sigma,g) \leq \frac{\int_{\Sigma}\left\vert \nabla \psi_\eps \right\vert^2_g dA_g}{\int_{\partial\Sigma}\left( \psi_\eps \right)^2 dA_g} \leq \frac{\int_{\tilde{\Sigma}_\eps} \vert \nabla \psi_\eps \vert_{\tilde{g}_\eps}^2 dA_{\tilde{g}_\eps} }{ \int_{\partial \tilde{\Sigma}_\eps} \left(\psi_\eps\right)^2 dA_{\tilde{g}_\eps}+O\left(\eps \sqrt{\ln \frac{1}{\eps}}\right) } \leq \sigma_i(\tilde{\Sigma}_\eps,\tilde{g}_\eps) + O\left(\eps \sqrt{\ln \frac{1}{\eps}}\right)
\end{split}
\end{equation*}
as $\eps \to 0$ where we used Claim \ref{cl:linftyesteigenfunctions3} to conclude for the inequality \eqref{eqsigmaiepsepss}. \eqref{eqEepsepss3} follows since $F$ is a $\mathcal{C}^1$ function, eigenvalues are uniformly bounded and $L(\partial\Sigma_\eps, \tilde{g}_\eps) = L(\partial\Sigma,g) - O(\eps)$.
\end{proof}
As a conclusion, we take $\delta_\eps := c \sqrt{\eps \sqrt{\ln \frac{1}{\eps}}} $ for some well chosen constant $c>0$ in the previous section and construct the previous Palais-Smale approximation to the sequence $\tilde{g}_\eps$ and $\tilde{\beta}_{\eps} = \frac{1}{L_{\tilde{g}_\eps}(\tilde{\Sigma}_\eps)}$ on $\tilde{\Sigma}_{\eps}$ pullbacked on a fixed surface $\Sigma$ by a bi-Lipschitz diffeomorphism.

\subsection{Some convergence of $\omega_{\eps}$ to $1$ and first replacement of $\Phi_{\eps}$} 

We set $\omega_\eps$ the harmonic extension of the following map defined on $\partial \Sigma$
$$ \omega_{\eps} = \sqrt{\left\vert \Phi_{\eps} \right\vert_{\sigma_{\eps}}^2 + \theta_\eps^2 } 
\text{ on } \partial\Sigma \text{ and } \Delta_g \omega_\eps = 0$$ 
We first prove that $\nabla\omega_{\eps}$ converges to $0$ in $L^2(g_\eps)$ and that $\Phi_{\eps}$ has a similar $H^1(g_\eps)$ behaviour as $\frac{\Phi_{\eps}}{\omega_{\eps}}$

\begin{cl} We have that
\begin{equation}  \label{eqomegaepsto1steklov}   \int_{\tilde{\Sigma}_\eps}  \left\vert \nabla \omega_\eps \right\vert^2dA_{g_\eps} + \int_{\tilde{\Sigma}_\eps} \left\vert \nabla\left( \Phi_\eps - \frac{\Phi_\eps}{\omega_\eps} \right) \right\vert^2 dA_{g_\eps}  \leq O(\delta_\eps) \end{equation}
as $\eps\to 0$.
\end{cl}

The proof is similar to the proof of Claim  \ref{eqomegaepsto1steklov} but needs a particular attention because of the harmonic extension of $\omega_\eps$

\begin{proof} We first prove
\begin{equation} \label{eqomega_epsminus1steklov} L_\eps\left( \left\vert \sigma_\eps \Phi_\eps \right\vert^2 \left(1- \frac{1 }{\omega_\eps}\right) \right) \leq O(\delta_\eps) \end{equation}
as $\eps\to 0$. Since $\omega_\eps \geq 1$,  
and $\left\vert \Phi_\eps \right\vert_{\sigma_\eps}^2 \leq \omega_\eps^2$, we have that
\begin{equation*}
\begin{split}
 L_\eps\left( \left\vert \sigma_\eps \Phi_\eps \right\vert^2 \left(1- \frac{1 }{\omega_\eps} \right) \right) \leq & \left(\max{\sigma_i^\eps}\right) L_\eps\left(  \left(\omega_\eps^2-\omega_\eps\right) \right) \\
\leq & \max\sigma_i^\eps  \left( L_\eps\left(  \left\vert \Phi_\eps \right\vert_{\sigma_\eps}^2 \right) + L_\eps \left(  \theta_\eps^2 \right) - L_{\eps}(1) \right) \\
 \end{split} \end{equation*}
 so that
$$ L_\eps\left( \left\vert \sigma_\eps \Phi_\eps \right\vert^2 \left(1- \frac{1 }{\omega_\eps}\right) \right) \leq  \left(\max\sigma_i^\eps \right)  L_\eps( \theta_\eps^2) \leq  \left(\max\sigma_i^\eps \right)  \left\Vert \beta_\eps \right\Vert_{\tilde{g}_\eps} \left\Vert \theta_\eps \right\Vert^2_{H^1(\tilde{g}_\eps)} \leq O(\delta_\eps) $$
as $\eps \to 0$ since we know that
$$  \left\Vert \beta_\eps \right\Vert_{\tilde{g}_\eps} \leq  \left\Vert 1 \right\Vert_{\tilde{g}_\eps} +  \left\Vert 1- \beta_\eps \right\Vert_{\tilde{g}_\eps} \leq 1 + O(\delta_\eps) $$ 
and we obtain \eqref{eqomega_epsminus1steklov}.

We know prove \eqref{eqomegaepsto1steklov}:
\begin{equation*} 
\begin{split}
\int_{\tilde{\Sigma}_\eps} & \left\vert \nabla \frac{\Phi_\eps}{\omega_\eps} \right\vert_{\sigma_\eps}^2 - \int_{\tilde{\Sigma}_\eps}  \left\vert \nabla \Phi_\eps \right\vert_{\sigma_\eps}^2 - \int_{\tilde{\Sigma}_\eps}  \left\vert \nabla\left( \Phi_\eps - \frac{\Phi_\eps}{\omega_\eps} \right) \right\vert_{\sigma_\eps}^2 
\\ = & - 2 \int_{\tilde{\Sigma}_\eps} \left\langle \nabla \Phi_\eps, \nabla\left( \Phi_\eps - \frac{ \Phi_\eps}{\omega_\eps}\right) \right\rangle_{\sigma_\eps} 
= - 2 \int_{\tilde{\Sigma}_\eps} \Delta \Phi_\eps  {\sigma_\eps}.\left( \Phi_\eps - \frac{ \Phi_\eps}{\omega_\eps}\right)  \\
= & - 2 \beta_\eps\left( \sigma_\eps.\Phi_\eps , {\sigma_\eps}.\left( \Phi_\eps - \frac{ \Phi_\eps}{\omega_\eps}\right) \right)
= - 2 L_\eps\left(   \left\vert \sigma_\eps \Phi_\eps \right\vert^2 \left(1- \frac{1 }{\omega_\eps} \right) \right) = O(\delta_\eps)
\end{split}
\end{equation*}
where we tested $ \Delta \Phi_{\eps} =  \beta_\eps( \sigma_\eps\Phi_\eps,.) $ in $\Sigma$ against $ \sigma_\eps. \left( \Phi_\eps-\frac{ \Phi_\eps}{\omega_\eps}\right)$, and we used \eqref{eqomega_epsminus1steklov}.

In particular, we have
$$ 0 \leq \int_{\tilde{\Sigma}_\eps} \left\vert \nabla\left( \Phi_\eps - \frac{\Phi_\eps}{\omega_\eps} \right) \right\vert_{\sigma_\eps}^2 \leq \int_{\tilde{\Sigma}_\eps}  \left( \left\vert \nabla \frac{\Phi_\eps}{\omega_\eps} \right\vert_{\sigma_\eps}^2 - \left\vert \nabla \Phi_\eps \right\vert_{\sigma_\eps}^2\right) + O(\delta_\eps) $$
as $\eps\to 0$ and knowing that with the straightforward computations we have
\begin{equation*}
\begin{split} \left\vert \nabla \frac{\Phi_\eps}{\omega_\eps} \right\vert_{\sigma_\eps}^2 - \left\vert \nabla \Phi_\eps \right\vert_{\sigma_\eps}^2 = & \left( 1 - \omega_\eps^2  \right)\left\vert \nabla \frac{\Phi_\eps}{\omega_\eps} \right\vert_{\sigma_\eps}^2 - \left( \left\vert \nabla \omega_\eps \right\vert^2 \frac{\left\vert \Phi_\eps \right\vert^2_{\sigma_\eps}}{\omega_\eps^2}  + \omega_\eps \nabla \omega_\eps \nabla \frac{\left\vert \Phi_\eps \right\vert^2_{\sigma_\eps}}{\omega_\eps^2} \right) \\
=  & \left( 1 - \omega_\eps^2  \right)\left\vert \nabla \frac{\Phi_\eps}{\omega_\eps} \right\vert_{\sigma_\eps}^2 - \nabla \omega_\eps \nabla  \frac{\left\vert \Phi_\eps \right\vert^2_{\sigma_\eps}}{\omega_\eps} 
\end{split} \end{equation*}
Computing that
\begin{equation*}
\begin{split} \int_{\tilde{\Sigma}_\eps}  \nabla \omega_\eps, \nabla  \frac{\left\vert \Phi_\eps \right\vert^2_{\sigma_\eps}}{\omega_\eps}  & = \int_{\partial{\tilde{\Sigma}_\eps}} \partial_\nu \omega_\eps \left(\omega_\eps - \frac{\theta_\eps^2}{\omega_\eps}\right)  \\
& = - \int_{{\tilde{\Sigma}_\eps}} \Delta \frac{\omega_\eps^2}{2} - \int_{{\tilde{\Sigma}_\eps}} \nabla \omega_\eps \nabla \frac{\theta_\eps^2}{\omega_\eps} \\
& = \int_{{\tilde{\Sigma}_\eps}} \left\vert \nabla \omega_\eps \right\vert^2 + \int_{{\tilde{\Sigma}_\eps}} \frac{\theta_\eps^2}{\omega_\eps^2} \left\vert \nabla \omega_\eps \right\vert^2 - 2 \int_{{\tilde{\Sigma}_\eps}} \frac{\theta_\eps}{\omega_\eps} \nabla \theta_\eps \nabla \omega_\eps \\
& \geq \int_{{\tilde{\Sigma}_\eps}} \left\vert \nabla \omega_\eps \right\vert^2 -  \int_{{\tilde{\Sigma}_\eps}} \left\vert \nabla \theta_\eps \right\vert^2
\end{split} \end{equation*}
and we obtain since $\frac{\theta_\eps}{\omega_\eps}$ is uniformly bounded by $1$ that
\begin{equation*}
\begin{split} \int_{\tilde{\Sigma}_\eps} \left\vert \nabla\left( \Phi_\eps - \frac{\Phi_\eps}{\omega_\eps} \right) \right\vert_{\sigma_\eps}^2 + \int_{{\tilde{\Sigma}_\eps}} \left( \omega_\eps^2 - 1 \right)\left\vert \nabla \frac{\Phi_\eps}{\omega_\eps} \right\vert_{\sigma_\eps}^2 + \int_{{\tilde{\Sigma}_\eps}} \left\vert \nabla \omega_\eps \right\vert^2 \leq O\left(\delta_\eps \right)  
\end{split} \end{equation*}
as $\eps \to 0$.

\end{proof}

\subsection{Quantitative convergence of eigenvalues and quantitative energy bounds}

We recall that $\sigma_k^\eps := \sigma_k(\tilde{\Sigma}_\eps,g_\eps,\beta_\eps)$. The following claims can be proved translating Claim \ref{cl:upperboundlambdakeps} and Claim \ref{cl:estimatediffeigen} from the context of Laplace eigenvalues to the context of Steklov eigenvalues.

\begin{cl} \label{cl:upperboundsigmakeps}
For all $k \in \mathbb{N}^\star$ 
$$ \sigma_k^\eps \leq \sigma_k + O\left( \frac{1}{\ln\frac{1}{\eps}} \right) $$
as $\eps \to 0$.
\end{cl}

\begin{cl} \label{cl:estimatediffeigensteklov} If $k$ is such that for any $x$, $\partial_kF(x) < 0$,
$$  \vert \sigma_k^\eps - \sigma_k \vert = O\left(\frac{1}{\ln \frac{1}{\eps}}\right) $$
and 
$$ \int_{R_{l,\eps}} \left\vert \nabla \frac{\Phi_\eps}{\omega_\eps} \right\vert_{g_\eps}^2 dA_{g_\eps}  =: V_\eps = O\left(\frac{1}{\ln \frac{1}{\eps}}\right) $$
as $\eps \to 0$.
\end{cl}

\subsection{A replacement of $\Phi_\eps$ in $\Sigma$}
We let $\Psi_\eps$ be equal to $\frac{\Phi_\eps}{\omega_\eps}$ in $\Sigma_\eps$.

We set $B_\eps$ the bilinear form on $H^1(\Sigma,\mathbb{R}^n)$ defined as
$$ B_\eps(X,X):= \sum_{j=1}^J \sum_{i\in A_j^\eps} \left( \int_\Sigma \vert \nabla X_i \vert^2_g dA_g - \mu_j \int_{\partial \Sigma} \left( X_i - \pi_{<\mu_j}(X_i) \right)^2 dL_g \right) $$
where $\pi_{< \sigma}$ is the projection in $L^2(\partial \Sigma,g)$ on the sum of the Steklov eigenspaces associated to eigenvalues on $(\Sigma,g)$ strictly less than $\sigma$. It is clear that $B$ is a non-negative bilinear form so that for any $X : \Sigma \to \R^n$, the Cauchy-Schwarz inequality gives that
$$ B_\eps(X, \Psi_\eps) \leq \sqrt{B_\eps(\Psi_\eps,\Psi_\eps)} \sqrt{B_\eps(X,X)} $$
We aim at estimating $B(\Psi_\eps,\Psi_\eps)$ with respect to 
$$ V_\eps := \int_{R_{l,\eps}} \left\vert \nabla \frac{\Phi_\eps}{\omega_\eps} \right\vert^2 = O\left(\frac{1}{\eps}\right) $$
In the following proposition, analogous to Proposition \ref{prop:estonBPsieps}, notations for $\mu_j$ and $A_j^\eps$ are similar:

\begin{prop} \label{prop:estonBPsiepssteklov}
There is a constant $C>0$ such that
$$ B(\Psi_\eps,\Psi_\eps) \leq C   \frac{V_\eps}{l}  + o\left(\frac{1}{\ln \frac{1}{\eps}}\right) $$
and for $1 \leq i \leq n$,
$$ \sum_{j=1}^J \sum_{i\in A_j^\eps} \int_{\partial\Sigma} \left(\pi_{<\mu_j}(\psi_i^\eps)\right)^2 dL_g \leq C \frac{V_\eps}{l} + O(\delta_\eps) $$
as $\eps\to 0$.
\end{prop}

The proof follows the proof of Proposition \ref{prop:estonBPsieps} (the proof is simpler in this case since we do not have to consider a harmonic extension).

\subsection{Estimates on the rest}
We set for $i\in A_j^\eps$
$$ \psi_i^\eps =   F_i^\eps + \pi_{<\sigma_i}\left( \psi_i^\eps \right) + R_i^\eps = F_i^\eps + S_i^\eps $$
where $F_i^\eps = \pi_{\mu_j}(\psi_i^\eps)$ is the projection of $\psi_i^\eps$ in on the eigenspace associated to $\mu_j$ in $L^2(\partial\Sigma,g)$. We then have that
$$  B_\eps(R_\eps,R_\eps ) = B_\eps(\Psi_\eps,R_\eps) \leq \sqrt{B_\eps(\Psi_\eps,\Psi_\eps)}\sqrt{B_\eps(R_\eps,R_\eps)} $$
so that
$$ B_\eps(R_\eps,R_\eps) \leq B_\eps(\Psi_\eps,\Psi_\eps) \leq C   \frac{V_\eps}{l}  + o\left(\frac{1}{\ln \frac{1}{\eps}}\right) $$
as $\eps \to 0$. Since $R_i^\eps \in \bigoplus_{\sigma > \mu_j} E_\sigma$, we obtain,
$$B_\eps(R_\eps,R_\eps) \geq \sum_{i=1}^n \left(1- \frac{\mu_j}{\mu_{j+1}}\right) \int_{\Sigma} \vert \nabla R_i^\eps \vert_g^2 dA_g  $$
and using in addition Proposition \ref{prop:estonBPsieps} we obtain that
\begin{equation} \label{eq:w12estimateReps2} \Vert S_\eps \Vert_{W^{1,2}}^2 \leq C   \frac{V_\eps}{l}  + o\left(\frac{1}{\ln \frac{1}{\eps}}\right)  \end{equation}
as $\eps \to 0$. In particular, letting $\eps \to 0$ and then $l\to +\infty$, we obtain that $\vert S_\eps \vert^2 + \vert \nabla S_\eps \vert^2  \to 0$ in $L^1$. In addition, we have the following claim which proof follows the proof of Claim \ref{cl:estonSeps}

\begin{cl} \label{cl:estonSepssteklov}
$$   \left\vert \frac{1}{\eps} \int_{I_\eps(p)}  S_\eps dL_g \right\vert^2 +  \left\vert \frac{1}{\eps} \int_{I_\eps(q)} S_\eps dL_g \right\vert^2 \leq C \frac{1}{l} +o(1)   $$
as $\eps \to 0$.
\end{cl}

\subsection{Conclusion}
We have that $\Psi_\eps = F_\eps + S_\eps$, where
$$  \left\vert \frac{1}{\eps} \int_{I_\eps(p)} \Psi_\eps dL_{\tilde{g}_\eps} - \frac{1}{\eps} \int_{I_\eps(q)} \Psi_\eps dL_{\tilde{g}_\eps} \right\vert^2 =  \left\vert\frac{1}{\eps} \int_{R_{l,\eps}} \partial_t \Psi_\eps \right\vert^2 \leq C \frac{l\eps^2}{\eps^2} V_\eps \leq O\left(\frac{1}{\ln \frac{1}{\eps}}\right)  $$
as $\eps\to 0$ so that from Claim \ref{cl:estonSeps}, estimates on $S_\eps$ give that
$$ \left\vert  \frac{1}{\eps} \int_{I_\eps(p)} F_\eps dL_g -  \frac{1}{\eps} \int_{I_\eps(q)} F_\eps dL_g \right\vert^2 \leq C \frac{1}{l} +o(1) + O\left(\frac{1}{\ln \frac{1}{\eps}}\right)  $$ 
as $\eps\to 0$. By a claim similar to Claim \ref{cl:rearrangement}, we rearrange $F_\eps$ into $\varphi_\eps : \Sigma \to \R^n$ that converges, up to the extraction of a subsequence to a (possibly branched) free boundary minimal immersion $\Phi : \left(\Sigma,\partial\Sigma \right) \to \left(co\left(\mathcal{E}_{\sigma}\right), \mathcal{E}_{\sigma}\right)$ and the convergence
$$  \left\vert \frac{1}{\eps} \int_{I_\eps(p)} \varphi_\eps dL_g - \frac{1}{\eps} \int_{I_\eps(q)}  \varphi_\eps dL_g \right\vert \to \vert \Phi(p) - \Phi(q) \vert$$
as $\eps\to 0$ gives with Claim \ref{cl:rearrangement} again that letting $\eps \to 0$ and then $l\to +\infty$, $\Phi(p) = \Phi(q)$. 

At the very end of our analysis if we assume that $p$ and $q$ are in the same  connected component of the boundary, letting $q \to p$ along the boundary, we obtain that for all $p \in \partial \Sigma$, there is a possibly branched conformal free boundary minimal immersion $ \Phi : \Sigma \to co\left(\mathcal{E}_{\sigma}\right) $ such that $\partial_{\tau} \Phi(p) = 0$. Since $\Phi$ is conformal, we have that $\vert \partial_\tau \Phi \vert =  \vert \partial_\nu \Phi \vert = 0$. We obtain that $\nabla \Phi(p) = 0$. Then $p$ is a branched point. Since $p$ was chosen arbitrarily, every point of $\partial\Sigma$ is a conical singularity for $g$ and we obtain a contradiction. Then $(p,q)\in \Delta\left( \partial \Sigma\right)$.

If the conjecture \ref{conj} holds, we obtain a contradiction in the case of the maximization of $\bar{\sigma}_1$.

\end{document}